\documentclass[11pt,a4paper]{article}
\usepackage{hyperref}
\hypersetup{nesting=true,debug=true,naturalnames=true}
\usepackage{graphicx,upref}

\usepackage{amsmath,amssymb,amsthm,enumerate,mathtools,bigfoot}
\usepackage{authblk}

\usepackage{xypic} 
\usepackage{mathrsfs} 
\usepackage{xspace}

\newtheorem{thm}{Theorem}[section]
\newtheorem{lem}[thm]{Lemma}
\newtheorem{prop}[thm]{Proposition}
\newtheorem{cor}[thm]{Corollary}

\theoremstyle{definition}
\newtheorem{defn}[thm]{Definition}
\newtheorem{ex}[thm]{Example}
\newtheorem{que}{Question}
\newtheorem{rem}[thm]{Remark}

\newcommand{\defbold}{\textbf}

\newcommand{\inv}{^{-1}}

\newcommand{\QZ}{\mathrm{QZ}}

\newcommand{\CC}{\mathrm{C}}

\newcommand{\N}{\mathrm{N}}

\newcommand{\tdlc}{t.d.l.c.\@\xspace}
\newcommand{\tdlcsc}{t.d.l.c.s.c.\@\xspace}

\newcommand{\con}{\mathrm{con}}

\newcommand{\us}{\mathbf{u}}

\newcommand{\nub}{\mathrm{nub}}
\newcommand{\nubl}{\mathrm{lnub}}
\newcommand{\snub}{\mathrm{rnub}}

\newcommand{\triv}{\{1\}}

\newcommand{\Aut}{\mathrm{Aut}}
\newcommand{\Inn}{\mathrm{Inn}}

\newcommand{\Sym}{\mathrm{Sym}}

\newcommand{\Res}{\mathrm{Res}}
\newcommand{\Dist}{\mathrm{Dist}}

\newcommand{\bC}{\mathbb{C}}

\newcommand{\bF}{\mathbb{F}}
\newcommand{\bN}{\mathbb{N}}

\newcommand{\bQ}{\mathbb{Q}}

\newcommand{\bZ}{\mathbb{Z}}

\newcommand{\mc}[1]{\mathcal{#1}}

\makeindex

\begin{document}

\title{Dynamics of flat actions on totally disconnected, locally compact groups}

\author{Colin D. Reid\thanks{The author is an ARC DECRA fellow.  Research supported in part by ARC Discovery Project DP120100996.}}
\affil{University of Newcastle, School of Mathematical and Physical Sciences, Callaghan, NSW 2308, Australia \\ colin@reidit.net}

\maketitle

\begin{abstract}
Let $G$ be a totally disconnected, locally compact group and let $H$ be a virtually flat (for example, polycyclic) group of automorphisms of $G$.  We study the structure of, and relationships between, various subgroups of $G$ defined by the dynamics of $H$.  In particular, we consider the following four subgroups: the intersection of all tidy subgroups for $H$ on $G$ (in the case that $H$ is flat); the intersection of all $H$-invariant open subgroups of $G$; the smallest closed $H$-invariant subgroup $D$ such that no $H$-orbit on $G/D$ accumulates at the trivial coset; and the group generated by the closures of contraction groups of elements of $H$ on $G$.
\end{abstract}


\tableofcontents

\section{Introduction}

\addtocontents{toc}{\protect\setcounter{tocdepth}{1}}
\subsection{Background}

Since the groundbreaking article \cite{Willis94} of G. Willis in 1994, a suite of tools for studying totally disconnected, locally compact (\tdlc) groups $G$ has been developed using the dynamics of the action of automorphisms of $G$ on the space of compact open subgroups of $G$.  The key concepts are the \defbold{scale}, which is a measure of how far an automorphism $\alpha$ fails to normalize a compact open subgroup, and \defbold{tidy} subgroups, which are the compact open subgroups that have the least displacement under $\alpha$.  The scale is a numerical invariant that can be thought of as analogous to the spectral radius in operator theory, and moreover it turns out that the tidy subgroups form a class of subgroups on which the action of $\alpha$ is especially well-behaved, with important structural characterizations.  This area of research may thus be termed \defbold{scale theory} or \defbold{tidy theory}.  The trivial case of tidy theory is when there exist arbitrarily small compact open subgroups that are $\alpha$-invariant; in this case, we say $\alpha$ is \defbold{anisotropic}.  More generally, a group of automorphisms is defined to be anisotropic if every element is anisotropic, and $G$ is anisotropic if $\mathrm{Inn}(G)$ is anisotropic.

Tidy theory has since been generalized from actions of cyclic groups to endomorphisms (\cite{WillisEndo}) and also to \defbold{flat} group actions, which are defined to be actions of a group $H$ on the \tdlc group $G$, such that there exists a compact open subgroup $U$ that is tidy for every element of $H$.  The theory of flat groups was introduced in \cite{WillisFlat}, although the term `flat' itself appeared slightly later (see \cite{BaumgartnerGeometric}, which also gives a more geometric presentation of the results in \cite{WillisFlat}).  The class of flat groups is surprisingly large: for instance all finitely generated nilpotent groups of automorphisms are flat, and all polycyclic groups of automorphisms are virtually flat.  Nevertheless, flat groups possess a special structure: given a flat group $H$, the set of uniscalar elements $H_\us$ (that is, the normalizer of any compact open subgroup that is tidy for $H$) forms a normal subgroup of $H$, and the quotient $H/H_\us$ is a torsion-free abelian group.  If $H$ is \defbold{flat of finite rank}, that is, $H/H_{\us}$ is finitely generated, then the tidy subgroups for $H$ admit something akin to an eigenspace decomposition.

Tidy theory has also been deepened, especially in the case of actions of $\bZ$, by the investigation of the role played by certain subgroups in controlling the dynamics.  The \defbold{contraction group}\index{contraction group}\index{con@$\con_G(\alpha)$} $\con(\alpha)$ of an automorphism $\alpha$, that is, the set of elements $x \in G$ such that $\alpha^n(x)$ converges to the identity, plays a critical role in tidy theory.  One can show that $\alpha$ is anisotropic if and only if both $\alpha$ and $\alpha\inv$ have trivial contraction group.

 An important fact for the theory of \tdlc groups (which does not hold for connected locally compact groups) is the result of Baumgartner--Willis and Jaworski (\cite{BaumgartnerWillis},\cite{Jaw}) that the contraction group also controls contraction relative to a closed subgroup: specifically, if $K$ is an $\alpha$-invariant closed subgroup of $G$, then the set of elements $x \in G$ such that $\alpha^n(x)K$ converges to $K$ in the coset space $G/K$ is precisely $\con(\alpha)K$.  This suggests the idea of decomposing the action of $\alpha$ into an `anisotropic' action on the coset space $G/K$, where $K$ is the smallest closed subgroup containing $\con(\alpha)$ and $\con(\alpha\inv)$, and a residual action on the subgroup $K$ itself.  As we shall see, this idea can be usefully generalized to flat group actions.

\subsection{The relative Tits core}

Contraction groups were used in \cite{CRW-TitsCore} to define the \defbold{Tits core} $G^\dagger$ of a \tdlc group $G$:
$$ G^\dagger := \langle \overline{\con(\alpha)} \mid \alpha \in \Inn(G) \rangle.$$
In this paper, we consider the notion of the \defbold{relative Tits core}\index{relative Tits core}\index{G@$G^\dagger_X$} of the set $A$ of automorphisms of $G$ (or a subset of $G$):
$$ G^\dagger_A := \langle \overline{\con(\alpha)} \mid \alpha \in A \cup A\inv \rangle.$$
Of particular interest is the case when $A$ is a singleton (in which case we define $G^\dagger_\alpha = G^\dagger_{\{\alpha\}}$, and it will transpire that $G^\dagger_\alpha = G^\dagger_{\langle \alpha \rangle}$), or when $A$ is a flat group of automorphisms.  In fact, the invariance properties of the relative Tits core will allow us to work in many cases with subgroups $A$ of $G$ that are \defbold{almost flat}, that is, such that some closed cocompact subgroup of $\overline{A}$ is flat on $G$.  (In particular, virtually flat groups of automorphisms can be interpreted as almost flat in this sense.)

\begin{rem}A similar notion has been studied in the context of Lie groups, where given $x \in G$, the group $\overline{\langle \con(x), \con(x\inv) \rangle}$ is called the \defbold{Mautner subgroup} associated to $x$; see \cite{Dani}.\end{rem}

The relative Tits core is defined in terms of the contraction groups of individual elements of $A$.  However, in the case that $A$ is a flat subgroup such that $A/A_{\us}$ is finitely generated (or $A$ contains a cocompact subgroup of this form), we shall see that $G^\dagger_A$ plays an important role in the action of $A$ as a whole.

Using the results of \cite{BaumgartnerWillis} and \cite{CRW-TitsCore}, we will obtain some invariance properties of the relative Tits core.  Like the scale function, the relative Tits core $G^\dagger_{x}$ of $x \in G$ remains constant under sufficiently small perturbations of $x$.

\begin{thm}[See Proposition~\ref{prop:twosided:stable} and Theorem~\ref{titscore_closure}]\label{thmintro:TitsCore:invariance}Let $G$ be a \tdlc group.
\begin{enumerate}[(i)]
\item Let $x \in G$ and let $U$ be a compact open subgroup of $G$ that is tidy for $x$.  Let $u,v \in U$ and let $n \in \bZ \setminus \{0\}$.  Then
$$ G^\dagger_{x} = G^\dagger_{ux^nv}.$$
Consequently, $G^\dagger_{x} = G^\dagger_X$, where $X = \bigcup_{n \in \bZ}Ux^nU$.
\item Let $X$ be a subset of $G$ and let $Y$ be the set of all elements $y \in G$ such that $\overline{\con(y)} \le G^\dagger_X$.  Then $Y$ is a clopen subset of $G$.  In particular, $G^\dagger_X = G^\dagger_{\overline{X}}$.
\end{enumerate}
\end{thm}

\begin{cor}\label{corintro:TitsCore:invariance}Let $G$ be a \tdlc group and let $H$ be an almost flat subgroup of $G$.  Then $\N_G(G^\dagger_H)$ is open in $G$.\end{cor}

In particular, $G^\dagger_g$ has open normalizer for all $g \in G$.  This contrasts with the normalizers of $\con(g)$ and $\nub(g)$: see \S\ref{sec:TitsCore_examples}.

Another interesting case of invariance concerns subgroups that are either cocompact or of finite covolume.

\begin{thm}[See \S\ref{sec:TitsCore}]\label{thmintro:TitsCore:cocompact}
Let $G$ be a \tdlc group, let $H$ be a closed subgroup of $G$ and let $K$ be a subgroup of $H$.  Suppose that $\overline{K}$ is either cocompact in $H$ or of finite covolume in $H$ (or both).   Then for all $h \in H$, there exists $k \in K$ and $t \in G^\dagger_k$ such that $\con(h) = t\con(k)t\inv$.  As a consequence,
 \[
 \{G^\dagger_h \mid h \in H\} = \{G^\dagger_k \mid k \in K\},
 \]
 and hence $G^\dagger_H = G^\dagger_K$.
 \end{thm}

In \cite{CRW-TitsCore}, it was shown that if $D$ is a dense subgroup of the \tdlc group $G$ that is normalized by $G^\dagger$, then $G^\dagger \le D$.  Here is a relative version of this result.

\begin{thm}[See \S\ref{sec:TitsCore:subgroups}]\label{thmintro:titscore:containment}Let $G$ be a \tdlc group, let $D$ be a subgroup of $G$ (not necessarily closed), and let $X \subseteq \overline{D}$.  Suppose that there is an open subgroup $U$ of $G$ such that $ U \cap G^\dagger_X \le \N_G(D)$.  Then $G^\dagger_{X} \le D$.
\end{thm}

\subsection{The nub of a flat group}

Let $H$ be a flat group of automorphisms of the \tdlc group $G$.  The \defbold{nub} $\nub(H)$ of $H$ is the intersection of all tidy subgroups for $H$.  This generalizes the notion of the nub of an automorphism introduced in \cite{WillisNub}; in particular, $\nub(\alpha) = \nub(\langle \alpha \rangle)$.

For a general flat group, the nub is more mysterious than in the cyclic case.  The difficulties emerge already in the case that $H$ is uniscalar.  For instance, $\nub(H)$ can have proper $H$-invariant open subgroups (see Example~\ref{ex:badnub}).  However, we are able to obtain some structural results for the nub.   The nubs of the subgroups of $H$ are all normal in $\nub(H)$, and nubs of uniscalar flat groups have open normalizer (see Corollary~\ref{cor:nub:normalizer}).  If $H$ has a uniscalar normal subgroup $L$ such that $H/L$ is polycyclic, then the nub of $H$ can be written as a product of $\nub(L)$ and finitely many nubs of cyclic subgroups of $H$.

\begin{thm}[See Theorem~\ref{thm:flat:nubdecomp}]\label{thmintro:flat:nubdecomp}
Let $G$ be a \tdlc group and let $H$ be a flat group of automorphisms of $G$.  Let $L$ be a uniscalar normal subgroup of $H$ such that $H/L$ is polycyclic.  Then there is a finite subset $\{\alpha_1,\alpha_2,\dots,\alpha_n\}$ of $H$ such that
$$ \nub(H) = \nub(L) \nub(\alpha_1) \nub(\alpha_2) \dots \nub(\alpha_n).$$
\end{thm}

The automorphism groups $H$ satisfying the hypotheses of Theorem~\ref{thmintro:flat:nubdecomp} are exactly the flat groups $H$ of \defbold{finite rank} (that is, $H/H_{\us}$ is finitely generated); one can then always take $L = H_{\us}$.  In Theorem~\ref{thmintro:flat:nubdecomp}, we make the hypothesis that $H/L$ is polycyclic, rather than setting $L = H_{\us}$, in order to gain insight into the nubs of some possibly uniscalar groups.  In particular, we obtain the following corollary.

\begin{cor}\label{corintro:nubdecomp:polycyclic}
Let $G$ be a \tdlc group and let $H$ be a polycyclic flat group of automorphisms of $G$.  Then there is a finite subset $\{\alpha_1,\alpha_2,\dots,\alpha_n\}$ of $H$ such that 
$$ \nub(H) = \nub(\alpha_1) \nub(\alpha_2) \dots \nub(\alpha_n).$$
\end{cor}

\subsection{Residuals}

Let $G$ be a topological group.  The \defbold{discrete residual} $\Res(G)$ of $G$ is the intersection of all open normal subgroups of $G$.  More generally, given a group $H$ of automorphisms of $G$, one can define $\Res_G(H)$, the \defbold{discrete residual of $H$ on $G$}, to be the intersection of all open $H$-invariant subgroups of $G$.  It is straightforward to show that the action of $H$ on the coset space $G/\Res_G(H)$ is distal.  One can also define the \defbold{distal residual} $\Dist_G(H)$, which is the smallest closed $H$-invariant subgroup of $G$ such that $H$ acts distally on $G/\Dist_G(H)$.  We also define the \defbold{$1$-distal residual} $\Dist^*_G(H)$, which is the smallest closed $H$-invariant subgroup of $G$ such that no $H$-orbit on $G/\Dist^*_G(H)$ accumulates at the trivial coset.  (In general $\Dist^*_G(H) \le \Dist_G(H)$, and it is not clear if this inequality can be strict, but certainly $\Dist^*_G(H) = \triv$ if and only if $\Dist_G(H) = \triv$.)  Evidently $\Dist^*_G(H)$ contains the contraction group of every element of $H$, so $\overline{G^\dagger_H} \le  \Dist^*_G(H)$.

One can iterate the process of taking the discrete residual of an action, to produce a (possibly transfinite) descending chain of closed subgroups of $G$ such that $H$ has residually discrete action on each factor, terminating in a group $\Res^\infty_G(H)$, which is the largest $H$-invariant subgroup of $G$ that has no proper open $H$-invariant subgroup.  It is straightforward to show (see Lemma~\ref{lem:distal_extension}) that no $H$-orbit on $G/\Res^\infty_G(H)$ accumulates at the trivial coset.

 In general, one thus has the following inclusions:
\begin{equation}\label{eq:distres}
\overline{G^\dagger_H} \subseteq  \Dist^*_G(H) \subseteq A_G(H) \subseteq \Res_G(H),
\end{equation}
where $A_G(H)$ is either $\Dist_G(H)$ or $\Res^\infty_G(H)$.

\subsection{A characterization of compactly generated uniscalar flat subgroups}

Compactly generated subgroups of $G$ that normalize a compact open subgroup can be characterized in several ways.

\begin{thm}[See~Theorem~\ref{thm:distal}]\label{thmintro:distal}
Let $G$ be a \tdlc group, let $H$ be a compactly generated closed subgroup of $G$, acting by conjugation, and let $K$ be a closed $H$-invariant subgroup of $G$.

Then the following are equivalent:
\begin{enumerate}[(i)]
\item $\Dist_K(H)$ is compact;
\item $\Res_K(H)$ is compact;
\item $H$ normalizes a compact open subgroup of $K$.
\end{enumerate}
Moreover, if any of the above conditions is satisfied, then
\[
\nub_K(H) = \Res_K(H) =  \Res^\infty_K(H) = \Dist_K(H)
\]
and $H$ acts ergodically on $\nub_K(H)$, with $\nub_K(H) = \Dist^*_K(H)$ in the case that $\nub_K(H)$ is metrizable.
\end{thm}

The following corollary, which is a strengthening of \cite[Corollary~4.1]{CM}, follows from the special case $H = K$ and $\Dist_K(H) = \triv$.

\begin{cor}\label{corintro:distal_SIN}Let $G$ be a distal \tdlc group.  Then every compactly generated closed subgroup of $G$ is a SIN group.\end{cor}

Nilpotent groups are distal, so Corollary~\ref{corintro:distal_SIN} also immediately implies the main theorem of \cite{WillisNilpotent}, that compactly generated nilpotent \tdlc groups are SIN groups.  However, as noted in \cite{WillisNilpotent}, there are non-SIN nilpotent \tdlc groups, so distal \tdlc groups are not SIN groups in general.

We also obtain the following corollary from Theorems~\ref{thmintro:flat:nubdecomp} and \ref{thmintro:distal}.  Write $\nub^2_G(H)$ for $\nub_{\nub_G(H)}(H)$.

\begin{cor}\label{corintro:ergodic_nub}
Let $G$ be a \tdlc group and let $H$ be a finitely generated flat group of automorphisms of $G$.  Then the action of $H$ on $\nub^2_G(H)$ is ergodic.  If in addition $H_{\us}$ is finitely generated, then $\nub^2_G(H) = \nub_G(H)$.
\end{cor}

\subsection{The discrete residual of the action of an almost finite-rank flat subgroup}

If $G$ is a metrizable \tdlc group and $H$ is a compactly generated flat subgroup of $G$ (or more generally, $H$ has a cocompact subgroup of this kind), we can say more about the relationships between the subgroups in (\ref{eq:distres}) using tidy theory, even in the case that $\Res_G(H)$ is not compact.  In particular, all the groups in (\ref{eq:distres}) are actually equal, except that $\overline{G^\dagger_H}$ may be properly contained in $\Dist_G(H)$.

\begin{thm}[See Theorem~\ref{thm:relative_discrete:general}]\label{thmintro:relative_discrete}Let $G$ be a \tdlc group, let $H$ be a compactly generated closed subgroup of $G$, and suppose there is a cocompact closed subgroup $K$ of $H$ such that $K$ is flat on $G$.
\begin{enumerate}[(i)]
\item The following subgroups of $G$ are all equal to $\Res_G(H)$:
\[
\Res_G(K), \; \overline{G^\dagger_H}\nub_G(K), \; \overline{G^\dagger_H}\nub_G(K_{\us}), \; \Dist_G(H), \; \Res^\infty_G(H).
\]
\item The normalizer of $\Res_G(H)$ in $G$ is open.  Indeed, $\Res_G(H)$ is normalized by every tidy subgroup for the action of $K$ on $G$.
\item $H$ is anisotropic and flat on $\N_G(G^\dagger_H)/\overline{G^\dagger_H}$.
\item $\overline{G^\dagger_H}$ is a cocompact normal subgroup of $\Res_G(H)$.  Indeed, $\Res_G(H)/\overline{G^\dagger_H}$ is the nub of the action of $H$ on $\N_G(G^\dagger_H)/\overline{G^\dagger_H}$.
\item If $G$ is metrizable then $\Dist^*_G(H) = \Res_G(H)$.
\end{enumerate}
\end{thm}

We highlight the particular case when $H$ has a polycyclic subgroup with cocompact closure.

\begin{cor}[See \S\ref{sec:relative_discrete}]\label{corintro:relative_discrete:polycyclic}Let $G$ be a \tdlc group, let $H \le G$, and suppose there is a polycyclic subgroup $K$ of $\overline{H}$ such that $\overline{K}$ is cocompact in $\overline{H}$.  Let $\mc{V}$ be the set of open $H$-invariant subgroups of $G$.  Then $\{V/\overline{G^\dagger_H} \mid V \in \mc{V}\}$ is a base of neighbourhoods of the trivial coset in $G/\overline{G^\dagger_H}$.\end{cor}

In particular, if every element of the polycyclic subgroup $H$ has trivial contraction group, then there exist arbitrarily small open normal subgroups of $G$ normalized by $H$.  (Compare \cite[Theorem 4.1]{Raja}.)

Theorem~\ref{thmintro:relative_discrete}(ii) also has the potential to limit the possibilities for $\Res_G(H)$ in terms of the normal subgroup structure of compact open subgroups.  The following is an illustration of this idea.

\begin{cor}[See \S\ref{sec:relative_discrete}]\label{corintro:relative_discrete:hji}
Let $G$ be a non-discrete \tdlc group, let $H$ be a compactly generated closed subgroup of $G$, and suppose there is a cocompact closed subgroup $K$ of $H$ such that $K$ is flat on $G$.  Suppose that every compact open subgroup $U$ of $G$ is \defbold{just infinite}, that is, every non-trivial closed normal subgroup of $U$ has finite index.  Then the following dichotomy holds:
\begin{enumerate}[(a)]
\item If $H$ normalizes a compact open subgroup of $G$, then there is a base of neighbourhoods of the identity in $G$ consisting of compact open subgroups normalized by $H$;
\item If $H$ does not normalize any compact open subgroup of $G$, then $\Res_G(H)$ is the unique smallest open subgroup of $G$ normalized by $H$.
\end{enumerate}
\end{cor}

\subsection{The Mautner phenomenon and subgroups of finite covolume}

If $H$ is a subgroup of $G$ and $D$ is a subgroup of $G$ normalized by $H$, there is a smallest closed $H$-invariant subgroup $\Dist^*_{G/D}(H)$ of $G$ such that $\Dist^*_{G/D}(H) \ge D$ and the conjugation action of $H$ on $G/\Dist^*_{G/D}(H)$ is such that no orbit accumulates at the trivial coset.  It is clear that $\Dist^*_{G/D}(H) \ge \Dist^*_{G/E}(H)$ whenever $D \ge E$.  The residual $K = \Dist^*_{G/H}(H)$ is of particular significance: $K$ is then a \tdlc group containing $H$ such that $K = \Dist^*_{K/H}(H)$, and for any such group, a version of the Mautner phenomenon applies.
 
 \begin{thm}[See \S\ref{sec:Mautner}]\label{thmintro:Mautner}Let $G$ be a topological group and let $H$ be a subgroup of $G$ such that $G = \Dist^*_{G/H}(H)$.
 
Let $X$ be a topological space admitting an action of $G$ by homeomorphisms, such that the map $G \rightarrow X; g \mapsto gx$ is continuous for all $x \in X$.  Let $x \in X$; suppose that $x$ is fixed by $H$, and that no $H$-orbit on $X \setminus \{x\}$ accumulates at $x$.  Then $x$ is fixed by $G$.\end{thm}

Given a subgroup $H$ of $G$ of finite covolume, we can use the Mautner phenomenon to obtain a restriction on $\Dist^*_{G/H}(H)$, and hence on $G^\dagger$ (which is the same as the relative Tits core $G^\dagger_H$, by Theorem~\ref{thmintro:TitsCore:cocompact}).

\begin{thm}[See \S\ref{sec:covolume}]\label{thmintro:covolume}
Let $G$ be a metrizable \tdlc group, let $H$ be a closed subgroup of $G$ of finite covolume, let $\mc{U}$ be the set of identity neighbourhoods in $G$ and define $K(H) := \bigcap_{U \in \mc{U}}HUH$.
\begin{enumerate}[(i)]
\item We have
\[
G^\dagger \le \Dist^*_{G/H}(H) = \overline{K(H)}.
\]
\item The group $D = \Dist^*_{G/H}(H)$ is the unique largest closed subgroup $D$ of $G$ such that $H \le D$ and $H$ acts ergodically on $D/H$.
\end{enumerate}
\end{thm}

\begin{cor}\label{corintro:covolume:charsimple}
Let $G$ be a metrizable \tdlc group, and suppose that $G^\dagger$ is dense in $G$; equivalently, in every Hausdorff quotient $G/N$ of $G$, some element has non-trivial contraction group.  Let $H$ be a subgroup of $G$ of finite covolume.  Then $H$ acts ergodically on $G/H$ by left translation.
\end{cor}

\subsection{Reduced envelopes of flat subgroups}

\begin{defn}
Let $G$ be a \tdlc group and let $X \subseteq G$.  An \defbold{envelope} of $X$ in $G$ is an open subgroup of $G$ that contains $X$.  Say an envelope $E$ of $X$ is \defbold{reduced} if, whenever $E_2$ is an envelope of $X$, then $|E:E \cap E_2|$ is finite.
\end{defn}

The circumstances under which a subgroup of $G$ has a unique smallest envelope are quite special: consider for instance the case of a compact subgroup of $G$ that is not open.  However, there are general circumstances under which reduced envelopes exist, and when they exist, they are clearly unique up to commensurability.  If $H \le G$ normalizes a compact open subgroup $U$ of $G$, then $HU$ is a reduced envelope for $H$.  More generally, if $H$ is a flat subgroup of $G$, a natural candidate for a reduced envelope for $H$ is the group $\langle H, U \rangle$, where $U$ is tidy for $H$.  We confirm that $\langle H, U \rangle$ is indeed reduced provided that $H/H_{\us}$ is finitely generated.  In fact, we obtain a reduced envelope for $H \le G$ whenever $\overline{H}$ has a closed cocompact subgroup $K$ such that $K$ is flat and $K/K_{\us}$ is finitely generated (so in particular, every polycyclic subgroup of $G$ has a reduced envelope).

\begin{thm}[See \S\ref{sec:reduced_envelope}]\label{thmintro:reduced_envelope}Let $G$ be a \tdlc group and let $K$ be a closed flat subgroup of $G$ such that $K/K_{\us}$ is finitely generated.  Let $U$ be a compact open subgroup that is tidy for $K$ and let $U_0 = \bigcap_{k \in K}kUk\inv$.
\begin{enumerate}[(i)]
\item The product $G^\dagger_KU_0$ is the group generated by all $K$-conjugates of $U$.  Hence $\langle K,U \rangle$ is a reduced envelope for $K$ in $G$, and moreover
\[
\langle K, U \rangle = G^\dagger_KU_0K.
\]
\item Let $H \le G$ such that $K$ is cocompact in $\overline{H}$.  Then $H$ has a reduced envelope in $G$, and every reduced envelope for $H$ in $G$ is also a reduced envelope for $K$ in $G$.  Moreover, given any reduced envelope $E$ of $H$, then $\overline{G^\dagger_HH}$ is a cocompact subgroup of $E$.
\end{enumerate}
\end{thm}

We obtain further restrictions on reduced envelopes in the case that the almost flat subgroup is subnormal.

\begin{thm}[See \S\ref{sec:cocompact_envelope}]\label{thmintro:subnormal_envelope}
Let $G$ be a \tdlc group and let $H$ be a compactly generated closed subnormal subgroup of $G$.  Suppose that there is a cocompact subgroup of $H$ that is flat on $G$.  Let $E$ be a reduced envelope of $H$.  Then the following holds:
\begin{enumerate}[(i)]
\item $H$ is a cocompact subgroup of $E$.
\item We have
\[
E^\dagger = H^\dagger \text{ and } \Res(E) = \Res^\infty(E) = \Res(H) = \Res^\infty(H).
\]
In particular, both $E^\dagger$ and $\Res(E)$ are subgroups of $H$ characterized by the internal structure of $H$.
\end{enumerate}
\end{thm}

\subsection{Non-closed contraction groups}

Let $\mathscr{W}$ be the class of \tdlc groups that admit a non-degenerate faithful weakly decomposable action on a Boolean algebra.  As observed by P.-E. Caprace, G. Willis and the author in \cite{CRW-Part2}, $\mathscr{W}$ includes many of the known examples of groups in the class $\mathscr{S}$ of non-discrete, compactly generated, topologically simple \tdlc groups $G$.   The class $\mathscr{W}$ is considerably larger than just $\mathscr{S} \cap \mathscr{W}$: for instance, if $G$ is a \tdlc group with trivial quasi-center and $\mathscr{W}$ contains some open subgroup of $G$, then $\mathscr{W}$ contains every open subgroup of $G$, and also every closed normal subgroup of $G$.

By \cite[Corollary~K]{CRW-Part2}, given $G \in \mathscr{S} \cap \mathscr{W}$, then some $g \in G$ has non-closed contraction group.  We can use the structure of reduced envelopes to extend this result to all of $\mathscr{W}$: given $G \in \mathscr{W}$, either all contraction groups in $G$ are trivial or there exists a non-closed contraction group in $G$.

\begin{thm}[See \S\ref{sec:weakdecomp}]\label{thmintro:weakdecomp:nonclosed}Let $G$ be a non-trivial compactly generated \tdlc group.  Suppose that $G$ has a non-degenerate faithful weakly decomposable action on a Boolean algebra.  Then exactly one of the following holds:
\begin{enumerate}[(i)]
\item $G$ is anisotropic and has arbitrarily small non-trivial compact normal subgroups.
\item There exists $g \in G$ such that $\nub(g)$ is non-trivial, in other words, $\con(g)$ is not closed.
\end{enumerate}
\end{thm}

\begin{cor}Let $G$ be a non-trivial \tdlc group.  Suppose that $G$ has a non-degenerate faithful weakly decomposable action on a Boolean algebra, and suppose there exists $g \in G$ such that $\con(g) \neq \triv$.  Then there exists $h \in G$ such that $\con(h)$ is not closed.\end{cor}

\subsection{Example: Neretin groups}\label{sec:Neretin}

Let $q \ge 2$ and let $T_q$ be the locally finite tree in which every vertex has $q+1$ neighbours.  Given a set $A$ of vertices, write $T_q \setminus A$ for the subgraph of $T_q$ induced by the vertices $VT_q \setminus A$.  A \defbold{spheromorphism} of $T_q$ is an equivalence class of graph isomorphisms from $T_q \setminus A$ to $T_q \setminus B$, where $A$ and $B$ are finite, and two such maps are considered equivalent if they agree except on finitely many vertices.  Note that if $T_q \setminus A$ has no vertices of degree $\le 1$, then any two equivalent isomorphisms of $T_q \setminus A$ are actually equal as graph isomorphisms.  The set of all spheromorphisms of $T_q$ then forms a group under composition, the \defbold{Neretin group} $N_q$, which carries a \tdlc group topology generated as follows: a basic neighbourhood $U_A$ of the identity, where $A$ ranges over the finite subtrees of $T_q$, is given by all isomorphisms of the graph $T_q \setminus A$ that leave invariant each component of this graph.  This group was introduced in \cite{Neretin}.

By \cite{Kapoudjian}, $N_q$ is a compactly generated simple group.  There is a non-degenerate faithful weakly decomposable action of $N_q$, given by the action of $N_q$ on (the clopen subsets of) the space of ends of $T_q$, so in fact $N_q \in \mathscr{S} \cap \mathscr{W}$.  Moreover, $N_q$ contains a copy of $\Aut(T_q)$ as an open subgroup.  Unlike $\Aut(T_q)$, the group $N_q$ possesses a diverse collection of relative Tits cores and of flat subgroups of arbitrarily large finite rank, and thus provides a relatively straightforward illustration of some of the concepts in this article.

\

\paragraph{A family of open subgroups} Let $q \ge 3$ be odd and let $n$ be a positive integer, let $A_n$ be the set of vertices of $T_q$ of distance less than $n$ from some fixed vertex and let $S_n = A_{n+1} \setminus A_n$.  Then $T_q \setminus A_n$ is a forest of $(q+1)^n$ trees, with each tree having a unique vertex $v \in S_n$ of degree $q$, and all other vertices have degree $q+1$.  Form a graph $\Gamma_{n,q}$ by adding edges to $T_q \setminus A_n$ between the vertices of $S_n$, so that each vertex of $S_n$ is joined to exactly one other vertex in $S_n$ (this is possible as $|S_n|$ is even).  Then $\Gamma_{n,q}$ is a forest consisting of $(q+1)^n/2$ trees, each of which is isomorphic to $T_q$.  The group $U_{n,q} = \Aut(\Gamma_{n,q})$ has the following structure:
\[
U_{n,q} \cong \prod_{C \in \mathcal{C}} \Aut(C) \rtimes \Sym(\mathcal{C}),
\]
where $\mathcal{C}$ is the set of components of $\Gamma_{n,q}$ and each of the groups $\Aut(C)$ is isomorphic to $\Aut(T_q)$.  Note that $\Aut(T_q)$ has a simple open subgroup of index $2$, which we denote $\Aut(T_q)^+$, and correspondingly $\Aut(C)^+$ is the simple subgroup of $\Aut(C)$ of index $2$.  Moreover $U_{n,q}$ is an open subgroup of $N_q$ in a natural sense.  For each component $C$ of $\Gamma_{n,q}$, we regard $\Aut(C)$ as a direct factor of $U_{n,q}$ in the natural way, and we choose some fixed isomorphisms between the components in order to specify $\Sym(\mathcal{C})$ as a finite subgroup of $U_{n,q}$.  Now fix $n$ and $q$ and let $G = N_q$.

\

\paragraph{Relative Tits cores} Let $g \in U_{n,q}$.  Since $U_{n,q}$ is open, we have $G^\dagger_g = (U_{n,q})^\dagger_g$.  By raising $g$ to a suitable power we may assume $g \in  \prod_{C \in \mathcal{C}} \Aut(C)$.  In this case, by considering the situation of $\Aut(T_q)$ (see Example~\ref{ex:tree} below), it is straightforward to see that for each of the components $C$ of $\Gamma_{n,q}$, either $(U_{n,q})^\dagger_g$ contains $\Aut(C)^+$ (if the action of $g$ on $C$ is hyperbolic) or $(U_{n,q})^\dagger_g$ has trivial intersection with $\Aut(C)$ (if the action of $g$ on $C$ is not hyperbolic).  In particular, all of the direct products $\prod_{C \in \mathcal{C}'}\Aut(C)^+$ occur as relative Tits cores of $G$, where $\mathcal{C}'$ is any subset of $\mathcal{C}$.  In this situation, it is straightforward to show that $G^\dagger_g$ is closed and cocompact in $\overline{\langle G^\dagger_g, g \rangle}$, although $G^\dagger_g$ does not necessarily contain any non-zero power of $g$: for instance, $g$ could act as a hyperbolic element on the component $C_1$ and as an elliptic element of infinite order on another component $C_2$, and then $G^\dagger_g$ would only take account of the hyperbolic component for $g$.  As expected from Corollary~\ref{corintro:TitsCore:invariance}, the normalizer of $G^\dagger_g$ is an open subgroup of $G$; indeed, in this case we see that $\N_G(G^\dagger_g)$ contains a finite index subgroup of $U_{n,q}$, although $G^\dagger_g$ is not necessarily normal in $U_{n,q}$.  More generally, if $H$ is any subgroup of $U_{n,q}$, we see that
\[
G^\dagger_H = \prod_{C \in \mathcal{C}'} \Aut(C)^+ = \Res_G(H)
\]
where $\mathcal{C}'$ is some subset of $\mathcal{C}$.

\

\paragraph{Flat groups and nubs}Let $\mathcal{C'}$ be a subset of $\mathcal{C}$.  For each $C \in \mathcal{C}'$, choose some $g_C \in U_{n,q}$ that has hyperbolic action on $C$ (with displacement distance $1$) and trivial action on the other components.  Then $H = \langle g_C \mid C \in \mathcal{C}' \rangle$ is a finitely generated free abelian subgroup of $G$.  It is easily seen that $H$ is flat on $G$ and $H_{\us} = \triv$, so $H$ is flat of rank $|\mathcal{C'}|$.  The nubs of the elements $g_C$ acting on $G$ are non-trivial (again by considering standard properties of $\Aut(T_q)$), but nevertheless it is easily seen that
\[
\nub_G(H) = \nub_G(\prod_{C \in \mathcal{C'}}g_C) = \prod_{C \in \mathcal{C}'}\nub_G(g_C) = \prod_{C \in \mathcal{C}'}\nub_{\Aut(C)}(g_C),
\]
illustrating Corollary~\ref{corintro:nubdecomp:polycyclic}.

\

\paragraph{Reduced envelopes} Let $H$ be as before.  Then we can find a finite subgroup of $U_{n,q}$ that permutes $\mathcal{C}$ faithfully in a manner compatible with the actions of the elements $g_C$, in order to form a semidirect product 
\[
L = H \rtimes (\Sym(\mathcal{C}') \times \Sym(\mathcal{C} \setminus \mathcal{C}'))
\]
such that $\{g_C \mid C \in \mathcal{C}'\}$ is a conjugacy class of $L$.  Then $L$ is not flat, since its derived group is not uniscalar, giving an example of a virtually flat group that is not flat (see also Example~\ref{ex:virtually_flat}).  However, as in Theorem~\ref{thmintro:reduced_envelope}(ii), $L$ has a reduced envelope in $G$.  In fact, there is a reduced envelope $E$ of $L$ is of the following form:
\[
E = (\prod_{C \in \mathcal{C}'}\Aut(C) \times \prod_{C \in \mathcal{C} \setminus \mathcal{C}'}K_C) \rtimes (\Sym(\mathcal{C}') \times \Sym(\mathcal{C} \setminus \mathcal{C}'))
\]
where $K_C$ is a compact open subgroup of $\Aut(C)$ (for instance, for $K_C$ one could take the fixator in $\Aut(C)$ of $C \cap S_n$).  Indeed, every finitely generated subgroup of $U_{n,q}$ will have a reduced envelope of this form for a unique $\mathcal{C}' \subseteq \mathcal{C}$.  Moreover, in light of the structure of relative Tits cores, all such reduced envelopes can be realized as a reduced envelope of a cyclic subgroup.

\

In the above discussion, it is important to note that given an element $g \in U_{n,q}$, the concepts of relative Tits core, nub and reduced envelopes of $g$ in $N_q$ are all defined purely in terms of the structure of $N_q$ as a topological group and the choice of $g$ as an element of $N_q$, without any direct reference to $U_{n,q}$, nor to the nature of $N_q$ as a group of tree spheromorphisms.  So we can recover some relatively complicated subgroups of $N_q$, such as the groups $\prod_{C \in \mathcal{C}'}\Aut(C)^+$, as invariants of the pair $(N_q,g)$ where $g$ is a suitably chosen element of $N_q$.

\subsection{Open questions}

If a \tdlc group $G$ has dense Tits core, as in the hypothesis of Corollary~\ref{corintro:covolume:charsimple}, then clearly it has no non-trivial discrete quotient.  As far as the author is aware, it is possible that the converse holds for compactly generated \tdlc groups $G$.  By \cite[Theorem~3.8]{BaumgartnerWillis} and \cite[Theorem~A]{CM}, Question~\ref{que:dense_titscore} below reduces to the case where $G$ is topologically simple, so it also suffices to determine whether or not $G^\dagger$ can be trivial for $G \in \mathscr{S}$, where $\mathscr{S}$ is the class of non-discrete, compactly generated, topologically simple \tdlc groups.

\begin{que}\label{que:dense_titscore}
Let $G$ be a compactly generated \tdlc group such that $\Res(G) = G$.  Is $G^\dagger$ necessarily dense in $G$?
\end{que}

An affirmative answer to the following would answer the previous question, but also have important consequences for the structure of elementary groups.  (See \S\ref{sec:TitsCoreElementary} for further discussion.)

\begin{que}\label{que:TitsCoreElementary}
Let $G$ be a non-elementary (in the sense of Wesolek \cite{Wesolek}) second-countable \tdlc group.  Must there exist some non-trivial element $g$ such that $g \in G^\dagger_g$?
\end{que}

Corollary~\ref{corintro:ergodic_nub} and Corollary~\ref{corintro:nubdecomp:polycyclic} give partial affirmative answers to the following question (in particular, it is completely solved in the case that $H$ is polycyclic), but the full answer is not clear.  Example~\ref{ex:badnub} below shows that some restriction on the structure of the flat group is necessary.  An affirmative answer to Question~\ref{que:ergodic_nub}(ii) would imply an affirmative answer to Question~\ref{que:ergodic_nub}(i).

\begin{que}\label{que:ergodic_nub}Let $G$ be a \tdlc group and let $H$ be a finitely generated flat group of automorphisms of $G$.
\begin{enumerate}[(i)]
\item Is the action of $H$ on $\nub(H)$ ergodic?
\item Does there exist a finite subset $\{\alpha_1,\alpha_2,\dots,\alpha_n\}$ of $H$ such that 
$$ \nub(H) = \nub(\alpha_1) \nub(\alpha_2) \dots \nub(\alpha_n)?$$
\end{enumerate}
\end{que}

The proof of Theorem~\ref{thmintro:weakdecomp:nonclosed} and the known examples of groups in $\mathscr{W}$ suggest affirmative answers to the following questions.  Note that by Corollary~\ref{corintro:distal_SIN}, to answer Question~\ref{que:weakbranch_SIN}(ii) affirmatively it is enough to show that under the given hypotheses, $G$ is distal.

\begin{que}\label{que:weakbranch_SIN}Let $G$ be a \tdlc group.  Suppose that $G$ has a non-degenerate faithful weakly decomposable action on a Boolean algebra.
\begin{enumerate}[(i)]
\item Let $g \in G$ such that $\nub(g) = \triv$.  Does it follow that $\con(g) = \triv$?
\item Suppose that $G$ is compactly generated and anisotropic.  Does it follow that $G$ is a SIN group?
\end{enumerate}
\end{que}

\subsection*{Acknowledgement}I thank Pierre-Emmanuel Caprace, Riddhi Shah, Phillip Wesolek and George Willis for their very helpful comments and suggestions.  I also thank the anonymous referee for a very thorough and helpful report with many good suggestions for improvements.  In particular, Riddhi Shah alerted me to \cite[Theorem~3.1]{RajaShah}, which led to a proof of Theorem~\ref{thmintro:distal}, \S\ref{sec:TitsCoreElementary} is based on a suggestion of the referee and discussions with Phillip Wesolek, several examples are based on suggestions by George Willis, and the proof of Theorem~\ref{thmintro:TitsCore:cocompact} is partly based on an argument shown to me by George Willis.

\addtocontents{toc}{\protect\setcounter{tocdepth}{2}}

\section{Preliminaries}

For the purposes of this article, all homomorphisms are required to be continuous.  Given a topological group $G$, $\Aut(G)$ denotes the group of automorphisms of $G$, that is, permutations of $G$ that are both group automorphisms and homeomorphisms.  Given $X \subseteq G$ and $Y \subseteq \Aut(G)$, we say $X$ is \defbold{$Y$-invariant} if $\alpha(X) = X$ for all $\alpha \in Y$.

Throughout, we adopt the convention that any definition given for automorphisms of a group $G$ also applies to an element $g$ of the group, acting via the automorphism $x \mapsto gxg\inv$.  Similarly, definitions given for sets of automorphisms also apply to subsets of the group itself.  In fact, the distinction between subgroups and automorphisms will turn out to be largely inconsequential, since a \tdlc group $G$ with a group of automorphisms $H$ can be extended to a \tdlc group $G \rtimes H$ in which $G$ is open, and we are concerned with properties of the action of $H$ that are invariant on restricting the action to an open $H$-invariant subgroup.

The following classical result is a defining feature of the theory of \tdlc groups, and will be frequently used without comment.

\begin{thm}[Van Dantzig, \cite{vanDantzig}]Let $G$ be a \tdlc group.  Then the compact open subgroups of $G$ form a base of neighbourhoods of the identity.\end{thm}

\subsection{Tidy theory for cyclic actions}

Let $G$ be a \tdlc group, let $\alpha \in \Aut(G)$, and let $U$ be a compact open subgroup of $G$.  Define the subgroups
\[
U_+ = \bigcap_{n \ge 0}\alpha^n(U); \; U_- = \bigcap_{n \le 0}\alpha^n(U).
\]

Then $U$ is \defbold{tidy above}\index{tidy!tidy above} for $\alpha$ if $U = U_+U_-$; equivalently, there exist subgroups $V$ and $W$ of $U$ such that $U = VW$, $\alpha(V) \ge V$ and $\alpha(W) \le W$.  It is \defbold{tidy below}\index{tidy!tidy below} for $\alpha$ if the group $U_{++} := \{g \in G \mid \forall n \ll 0 : \; \alpha^n(g) \in U \}$ is closed.

A \defbold{tidy subgroup}\index{tidy} for $\alpha$ is a compact open subgroup that is both tidy above and tidy below.  More generally, a compact open subgroup $U$ is said to be tidy (above, below) for a set of automorphisms $A$ if it is tidy (above, below) for each element $\alpha \in A$.  Some caution is required here, as a compact open subgroup $U$ may be tidy for $A$ without being tidy for the group generated by $A$ (see \cite[Example 3.5]{WillisFlat}).

The \defbold{scale}\index{scale} $s(\alpha)$ is the minimum value of the (necessarily finite) index $|\alpha(U):\alpha(U) \cap U|$ as $U$ ranges over the compact open subgroups of $G$.  We say $\alpha$ is \defbold{uniscalar}\index{uniscalar} if $s(\alpha) = s(\alpha\inv) = 1$; equivalently, $\alpha$ is uniscalar if it leaves invariant a compact open subgroup of $G$.

These concepts originate in \cite{Willis94}, where it was shown that a tidy subgroup exists for every automorphism of a \tdlc group.

\begin{thm}[\cite{Willis94} Theorem~1 and \cite{WillisFurther} Theorem~3.1]\label{basic:tidy}
Let $G$ be a \tdlc group and let $\alpha \in \Aut(G)$.  Then there exists a tidy subgroup for $\alpha$.  Indeed, given a compact open subgroup $U$ of $G$, then $U$ is tidy for $\alpha$ if and only if $|\alpha(U):\alpha(U) \cap U| = s(\alpha)$.
\end{thm}

Some equivalent formulations of the tidy below property are effectively given in \cite{Willis94}.  We can thus take any of the equivalent statements in Lemma~\ref{tidybelow:asymptotic} below as the definition of tidiness below, without any danger of ambiguity.

\begin{lem}[See \cite{Willis94} Lemma~3 and its corollary]\label{tidybelow:asymptotic}Let $G$ be a \tdlc group and let $\alpha \in \Aut(G)$.  Define
\[ U_{++} := \{g \in G \mid \exists m \in \bZ : \forall n \le m : \alpha^n(g) \in U \}; \]
\[ U_{--} := \{g \in G \mid \exists m \in \bZ : \forall n \ge m : \alpha^n(g) \in U \}; \]
\[ \mc{L}_U := U_{++} \cap U_{--}.\]
Then the following are equivalent.
\begin{enumerate}[(i)]
\item $U_{++}$ is closed;
\item $U_{--}$ is closed;
\item $\mc{L}_U \le U$;
\item $U_{++} \cap U = U_+$;
\item $U_{--} \cap U = U_-$.
\end{enumerate}
\end{lem}

\begin{proof}In \cite{Willis94}, conditions (i)--(iv) are shown to be equivalent to the condition that $U$ is tidy, under the assumption that $U$ is tidy above.  However, we can bypass the assumption that $U$ is tidy above by noting (as in \cite{Willis94}) that any compact open subgroup $V$ can be replaced with the tidy above subgroup $U = \bigcap^n_{i=0}\alpha^i(V)$ for $n$ large enough.  We see that $U_+ = V_+$, $U_{++} = V_{++}$, $U_{--} = V_{--}$, $\mc{L}_U = \mc{L}_V$ and $\mc{L}_V \cap V \le U$.  So $V$ is tidy below if and only if $U$ is tidy below, and $U$ is tidy below if and only if any one of the equivalent statements (i)--(iv) is satisfied, which can all be translated to corresponding statements for $V$.  One can see the equivalence of (iv) and (v) by noting that replacing $\alpha$ with $\alpha\inv$ reverses the roles of (iv) and (v), but has no effect on (iii).
\end{proof}

There are strong restrictions on the dynamics of $\alpha$ on orbits that intersect a tidy subgroup.  In particular, an $\alpha$-orbit cannot leave the tidy subgroup $U$ and then return to it, and any forward or backward $\alpha$-orbit that escapes from $U$ is necessarily unbounded.

\begin{lem}[\cite{WillisFlat} Lemma~2.6]\label{tidy:orbit_dynamics}Let $G$ be a \tdlc group, let $\alpha \in \Aut(G)$, let $U$ be a compact open subgroup of $G$ that is tidy for $\alpha$ and let $u \in U$.
\begin{enumerate}[(i)]
\item The set $\{\alpha^n(u) \mid n \ge 0\}$ is bounded (that is, relatively compact in $G$) if and only if $u \in U_-$.
\item The set $\{ n \in \bZ \mid \alpha^n(u) \in U\}$ is an interval in $\bZ$.
\end{enumerate}
\end{lem}

For a fixed automorphism $\alpha$, the behaviours of the classes of tidy above and tidy below subgroups are somewhat divergent.  Tidy above subgroups can be thought of as `small enough'; in particular, they form a base of identity neighbourhoods, by the following result:

\begin{prop}[\cite{Willis94} Lemma~1]\label{basic:tidyabove}
Let $G$ be a \tdlc group, let $\alpha \in \Aut(G)$ and let $U$ be a compact open subgroup of $G$.  Then there exists $n$ (depending on $U$ and $\alpha$) such that for all intervals $I \in \bZ$ of length at least $n$, the intersection $\bigcap_{i \in I}\alpha^i(U)$ is tidy above for $\alpha$.
\end{prop}

Tidy below subgroups are instead `large enough', in a way that is characterized by the \defbold{nub}\index{nub!of an automorphism} $\nub(\alpha)$ of $\alpha$.  The nub is the intersection of all tidy subgroups for $\alpha$; it also admits several other equivalent definitions, as described in \cite{WillisNub}.

\begin{prop}[\cite{WillisNub} Corollary~4.2]\label{basic:tidybelow}
Let $G$ be a \tdlc group, let $\alpha \in \Aut(G)$ and let $U$ be a compact open subgroup of $G$.  Then $U$ is tidy below for $\alpha$ if and only if $\nub(\alpha) \le U$.  In particular, if $U$ is tidy below for $\alpha$ and $V$ is a compact subgroup of $G$ such that $V \ge U$, then $V$ is tidy below for $\alpha$.
\end{prop}

\begin{thm}[\cite{WillisNub} Theorem~4.1]\label{basic:nub_characterizations}
Let $G$ be a \tdlc group and let $\alpha \in \Aut(G)$.  Then $\nub(\alpha)$ is the largest closed subgroup of $G$ on which $\alpha$ acts ergodically and the largest compact subgroup of $G$ that has no proper open $\alpha$-invariant subgroups.  In addition, $\nub(\alpha) = \overline{\con(\alpha)} \cap \overline{\con(\alpha\inv)}$.
\end{thm}

We note that both the contraction group and the nub are invariant under replacing $\alpha$ with a positive power.

\begin{lem}\label{lem:contraction:powers}Let $G$ be a topological group and let $\alpha \in \Aut(G)$.
\begin{enumerate}[(i)]
\item Let $n$ be a positive integer.  Then $\con(\alpha) = \con(\alpha^n)$.
\item Let $n \in \bZ \setminus \{0\}$.  Then $\nub(\alpha) = \nub(\alpha^n)$.
\end{enumerate}
\end{lem}

\begin{proof}
Suppose $n$ is a positive integer.  We have $\con(\alpha^n) \ge \con(\alpha)$, since $(\alpha^{ni})_{i \in \bN}$ is a subsequence of $(\alpha^i)_{i \in \bN}$.  Let $x \in \con(\alpha^n)$ and set $x_i = \alpha^i(x)$.  Then the sequence $(x_{ni})_{i \in \bN}$ converges to the identity.  Since $\alpha^j$ is a continuous automorphism and $\alpha^j(x_k) = x_{j+k}$ for all $j,k \in \bZ$, it follows that the sequence $(x_{j+ni})_{i \in \bN}$ converges to $\alpha^j(1) = 1$.  Hence $(x_i)_{i \in \bN}$ converges to the identity, since it can be partitioned into finitely many subsequences $(x_{j+ni})_{i \in \bN}$ for $0 \le j < n$, each of which converges to the identity.  In other words, $x \in \con(\alpha)$, completing the proof of (i).

Part (ii) now follows immediately from part (i) and Theorem~\ref{basic:nub_characterizations}.
\end{proof}

We see that if $U$ is tidy (above, below) for $\alpha$, then it is also tidy (above,below) for $\alpha^n$, for any $n \in \bZ \setminus \{0\}$.  (For tidiness above, the converse is false: for example, if $\alpha$ acts on the group $\bZ_p \times \bZ_p$ by swapping the two copies of $\bZ_p$, then $\bZ_p \times p\bZ_p$ is tidy for $\alpha^2$, but it is not tidy above for $\alpha$.)

\begin{lem}\label{basic:tidy:powers}
Let $G$ be a \tdlc group, let $\alpha \in \Aut(G)$, let $n \in \bZ \setminus \{0\}$ and let $U$ be a compact open subgroup of $G$.
\begin{enumerate}[(i)]
\item If $U$ is tidy above for $\alpha$, then it is tidy above for $\alpha^n$.
\item $U$ is tidy below for $\alpha$ if and only if it is tidy below for $\alpha^n$.
\end{enumerate}
\end{lem}

\begin{proof}
Given Lemma~\ref{tidybelow:asymptotic}, observe that $\alpha$ and $\alpha\inv$ play symmetrical roles in the definitions of tidy above and tidy below.  Thus we may assume $n > 0$.

If $U$ is tidy above for $\alpha$, then $U = VW$ with $\alpha(V) \ge V$, so $\alpha^n(V) \ge V$, and $\alpha(W) \le W$, so $\alpha^n(W) \le W$.  Thus $U$ is tidy above for $\alpha^n$, proving (i).

Part (ii) follows immediately from Proposition~\ref{basic:tidybelow} and Lemma~\ref{lem:contraction:powers}.
\end{proof}

A characterization of when $\nub(\alpha)$ is trivial is given in \cite{BaumgartnerWillis}.\footnote{In \cite{BaumgartnerWillis}, the authors often assume that the \tdlc group $G$ is metrizable, but only do so in order to appeal to \cite[Theorem~3.8]{BaumgartnerWillis}.  The metrizability assumption was later shown to be superfluous by Jaworski (\cite[Theorem~1]{Jaw}), so the remaining results of \cite{BaumgartnerWillis} are also valid for \tdlc groups in general.}

\begin{thm}[See \cite{BaumgartnerWillis} Corollary~3.30 and Theorem~3.32]\label{basic:closed_contraction}
Let $G$ be a \tdlc group and let $\alpha \in \Aut(G)$.  Then $\overline{\con(\alpha)} = \con(\alpha)\nub(\alpha)$, and $\nub(\alpha) = 1$ if and only if $\con(\alpha)$ is a closed subgroup of $G$.
\end{thm}

Applying the scale function to inner automorphisms defines a function from $G$ to the positive integers.  This function is continuous (with respect to the discrete topology on $\bN$), due to the stability properties of the tidy subgroups.

\begin{thm}[Willis \cite{Willis94}]\label{tidy:stable}Let $G$ be a \tdlc group.
\begin{enumerate}[(i)]
\item Let $U$ be a compact open subgroup of $G$.  Let $X$ be the set of elements $x \in X$ such that $U$ is tidy for $x$.  Then $X$ is invariant under left and right translations by $U$, in other words, $X$ is a union of $(U,U)$-double cosets.  In particular, $X$ is a clopen subset of $G$.  In addition, for all $n \in \bZ$, if $x \in X$ then $x^n \in X$.
\item The function $s: G \rightarrow \bN$ is continuous when $\bN$ is equipped with the discrete topology.  Indeed, if $U$ is tidy for $x \in G$, then $s(x) = s(y)$ for all $y \in UxU$.
\item Let $\alpha$ be an automorphism of $G$.  Then the collection of tidy subgroups for $\alpha$ is invariant under the action of $\alpha$ and closed under finite intersections.
\end{enumerate}
\end{thm}

\begin{proof}(i)
$X$ is a union of $(U,U)$-double cosets by \cite[Theorem~3]{Willis94}, and any union of left cosets of a fixed open subgroup is clopen.  Given $x \in X$, then $x^n \in X$ for all $n \in \bZ$ by Lemma~\ref{basic:tidy:powers}.

(ii) is \cite[Theorem~3 and Corollary~4]{Willis94}.  

(iii)
It is clear that the collection of tidy subgroups for $\alpha$ is invariant under the action of $\alpha$.  The fact that this collection is closed under finite intersections is \cite[Lemma~10]{Willis94}.
\end{proof}

The scale function is well-behaved under positive powers.

\begin{lem}[\cite{Willis94} Corollary 3]\label{basic:scale:powers}
Let $G$ be a \tdlc group, let $\alpha \in \Aut(G)$ and let $n > 0$ be a natural number.  Then $s(\alpha^n) = s(\alpha)^n$; equivalently,
\[
|\alpha^n(U): \alpha^n(U) \cap U|^{1/n} = s(\alpha)
\]
for every compact open subgroup $U$ that is tidy for $\alpha$.
\end{lem}

Given $\alpha \in \Aut(G)$ and $n > 0$, then $|\alpha^n(U):\alpha^n(U) \cap U|^{1/n} = s(\alpha)$ if and only if $U$ is tidy for $\alpha$.  However, the same equation holds asymptotically as $n \rightarrow +\infty$ for any given compact open subgroup $U$.  Thus the $s(\alpha)$ can be thought of as a kind of spectral radius for $\alpha$.

\begin{thm}[\cite{Moller} Theorem~7.7]\label{scale:asymptotic}Let $G$ be a \tdlc group, let $\alpha$ be an automorphism of $G$, and let $U$ be a compact open subgroup of $G$.  Then
\[
|\alpha^n(U):\alpha^n(U) \cap U|^{1/n} \rightarrow s(\alpha) \text{ as } n \rightarrow +\infty.
\]
\end{thm}

We derive the following result from Theorem~\ref{scale:asymptotic}; it can also be derived easily from \cite[Proposition~4.3]{WillisFurther}.

\begin{cor}\label{openinvariant:tidy}Let $G$ be a \tdlc group, let $\alpha$ be an automorphism of $G$ and let $K$ be an open subgroup of $G$ such that $\alpha(K) = K$.  Then $s_G(\alpha) = s_K(\alpha)$, and every compact open subgroup of $K$ that is tidy for $\alpha$ on $K$ is also tidy for $\alpha$ on $G$.\end{cor}

\begin{proof}Let $V$ be a compact open subgroup of $K$.  Then $V$ is open in $G$, so by Theorem~\ref{scale:asymptotic}, we have
\[
s_G(\alpha) = \lim_{n \rightarrow \infty}|\alpha^n(V):\alpha^n(V) \cap V|^{1/n} = s_K(\alpha).
\]
The assertion about tidy subgroups follows from Theorem~\ref{basic:tidy}.
\end{proof}

An automorphism $\alpha$ is \defbold{anisotropic}\index{anisotropic} if the set of compact open $\alpha$-invariant subgroups of $G$ forms a base of identity neighbourhoods, and \defbold{isotropic}\index{isotropic} if it is not anisotropic.  Given a \tdlc group $G$ and a group $H$ acting on $G$ (or a subgroup $H$ of $G$), we say $H$ is uniscalar or anisotropic respectively on $G$ if all the automorphisms of $G$ induced by $H$ are so.  `Uniscalar/anisotropic subgroup' should be understood in this relative sense.

Anisotropic automorphisms are necessarily uniscalar.  In general, a uniscalar automorphism need not be anisotropic, however certain local structures of the group $G$ can force all uniscalar automorphisms to be anisotropic: for example, if some (equivalently, every) compact open subgroup $U$ of $G$ is topologically finitely generated and virtually pro-$p$, then $U$ admits a base of identity neighbourhoods consisting of characteristic subgroups, so any automorphism leaving $U$ invariant must be anisotropic.

Contraction groups and the nub can be used to characterize when an automorphism is uniscalar or anisotropic.

\begin{prop}\label{basic:anisotropic}Let $G$ be a \tdlc group and let $\alpha \in \Aut(G)$.
\begin{enumerate}[(i)]
\item We have $s(\alpha)=1$ if and only if $\con(\alpha\inv)$ is relatively compact.
\item Suppose that $\alpha$ is uniscalar.  Then $\alpha$ is anisotropic if and only if $\nub(\alpha)$ is trivial.
\item If $\con(\alpha) = \con(\alpha\inv) = \triv$, then $\alpha$ is anisotropic (and conversely).
\end{enumerate}
\end{prop}

\begin{proof}
For part (i), see \cite[Proposition~3.24]{BaumgartnerWillis}.

Suppose that $\alpha$ is uniscalar.  Then a compact open subgroup of $G$ is tidy for $\alpha$ if and only if it is $\alpha$-invariant.  If $\alpha$ is anisotropic, then evidently the intersection of all $\alpha$-invariant subgroups for $\alpha$ is trivial, so $\nub(\alpha)=\triv$.  Conversely if $\nub(\alpha)=\triv$, consider a compact open subgroup $U$ of $G$ and an $\alpha$-invariant compact open subgroup $V$ of $G$.  Then by the compactness of $V \setminus U$, there exists a finite set $\{V_1,V_2,\dots,V_n\}$ of $\alpha$-invariant compact open subgroups of $G$ such that $W= V \cap \bigcap^n_{i=1}V_i \le U$.  Now $W$ is an $\alpha$-invariant compact open subgroup; since $U$ was an arbitrary compact open subgroup of $G$, we conclude by Van Dantzig's theorem that there exist arbitrarily small compact open $\alpha$-invariant subgroups of $G$, that is, $\alpha$ is anisotropic, proving (ii).

If $\alpha$ is anisotropic, then clearly $\con(\alpha) = \con(\alpha\inv) = \triv$.  Conversely, if $\con(\alpha) = \con(\alpha\inv) = \triv$, then $\alpha$ is uniscalar by part (i) and $\nub(\alpha) = \triv$ by Theorem~\ref{basic:nub_characterizations}, so $\alpha$ is anisotropic, proving (iii).
\end{proof}

\subsection{Flat groups}

A group of automorphisms $H$ of $G$ is \defbold{flat}\index{flat} if there exists a compact open subgroup $U$ of $G$ such that $U$ is tidy for $H$, that is, for all $\alpha \in H$, $U$ is tidy for $\alpha$.  More generally, any group acting on $G$ (such as a subgroup of $G$ acting by conjugation) is said to be flat on $G$ if it induces a flat group of automorphisms, and `flat subgroup' should be understood in this relative sense.  (Note that if $H$ is a closed subgroup of $G$, then $H$ may be flat on itself without being flat on $G$.)

A class of groups that are evidently flat are groups $H \le \Aut(G)$ such that $H$ leaves invariant a compact open subgroup of $G$.  More generally, it is easily seen that in any flat group $H$, and given any tidy subgroup $U$ for $H$, the set of elements of $H$ that leave $U$ invariant form a normal subgroup, the \defbold{uniscalar part}\index{uniscalar!uniscalar part}\index{Hu@$H_{\us}$} $H_{\us}$ of $H$, which does not depend on the choice of $U$.

The uniscalar part itself could potentially be any group that acts by automorphisms on a compact open subgroup of $G$.  However, the quotient $H/H_{\us}$ has a special structure, as first described by Willis in \cite{WillisFlat}.  In particular, the following holds:

\begin{thm}[\cite{WillisFlat} Theorem~4.15]\label{flat:uniscalar_derived}Let $G$ be a \tdlc group and let $H$ be a flat group of automorphisms of $G$.  Then $H/H_{\us}$ is a torsion-free abelian group, and every non-identity element of $H/H_{\us}$ is a finite power of an indivisible element.\end{thm}

The \defbold{(flat) rank}\index{rank} of a flat group is the minimum number of generators of $H/H_{\us}$.

Some sufficient conditions for a group to be a finite-rank flat group were given in \cite{WillisFlat}, with further generalizations in \cite{ShalomWillis}.

\begin{thm}[\cite{ShalomWillis} Theorems~4.9 and 4.13]\label{thm:nilpotent:flat}Let $G$ be a \tdlc group, let $H$ be a group of automorphisms of $G$, and let $K$ be a normal subgroup of $H$ such that $K$ leaves invariant a compact open subgroup of $G$.
\begin{enumerate}[(i)]
\item If $H/K$ is finitely generated and nilpotent, then $H$ is flat.
\item If $H/K$ is polycyclic, then $H$ has a flat subgroup of finite index.
\end{enumerate}
\end{thm}

Example~\ref{ex:virtually_flat} below shows that finitely generated polycyclic groups need not be flat, and an example given after \cite[Theorem~4.13]{ShalomWillis} shows that finitely generated soluble groups need not be virtually flat.  

We see from Theorem~\ref{thm:nilpotent:flat} that flatness of finite rank persists on restricting the action to a closed invariant subgroup.

\begin{cor}\label{cor:finiterank:subgroup}Let $G$ be a \tdlc group and let $H$ be a flat group of automorphisms of $G$ of finite rank.  Let $K$ be a closed $H$-invariant subgroup of $G$.  Then $H$ is flat of finite rank on $K$.\end{cor}

\begin{proof}Let $U$ be a compact open subgroup of $G$ that is tidy for $H$ and let $L$ be the uniscalar part of $H$ acting on $G$.  Then $U$ is $L$-invariant, so $U \cap K$ is also $L$-invariant, and $H/L$ is finitely generated and abelian.  Hence $H$ is flat of finite rank on $K$ by Theorem~\ref{thm:nilpotent:flat}.
\end{proof}

In discussions of flat subgroups of \tdlc groups, it is convenient to work with closed subgroups.  We note that the flat property is well-behaved under closure.

\begin{lem}\label{flat:closure}Let $G$ be a \tdlc group and let $H$ be a flat subgroup of $G$.  Then $\overline{H}$ is a flat subgroup of $G$ and $H_{\us}$ is an open subgroup of $H$.\end{lem}

\begin{proof}We see that $H_{\us}$ is open by Theorem~\ref{tidy:stable}, since it consists of the uniscalar elements of $H$.  Also by Theorem~\ref{tidy:stable}, any compact open subgroup that is tidy for $H$ is also tidy for $\overline{H}$, so $\overline{H}$ is flat.\end{proof}

\begin{defn}A subgroup $H$ of a \tdlc group $G$ is \defbold{almost flat}\index{flat!almost flat} (on $G$) if $\overline{H}$ has a closed cocompact subgroup $K$ such that $K$ is flat on $G$.  Say $H$ is \defbold{almost finite-rank flat} if in addition $K$ can be chosen so that $K/K_{\us}$ is finitely generated.\end{defn}

It is not clear at present whether an almost finite-rank flat subgroup is necessarily \defbold{virtually flat}, that is, has a subgroup of finite index that is flat on $G$.  Virtually flat subgroups however need not be flat, as the next example shows.  In any case, almost (finite-rank) flat subgroups will be sufficiently well-behaved for most purposes in the present paper.

\begin{ex}\label{ex:virtually_flat}Let $K = \bQ_p \rtimes \langle t \rangle$, where $\bQ_p$ is open in $K$ and $t$ acts on $\bQ_p$ as multiplication by $p$, let $G = K \wr C$ where $C$ is a finite non-trivial cyclic group acting regularly, and let $H$ be the polycylic subgroup $\langle t \rangle \wr C = B \rtimes C$, where $B \cong \bZ^n$.  Observe that no non-trivial element of $B$ is uniscalar, so in particular the derived group of $H$ is not uniscalar.  Hence $H$ is not flat.  However, the finite index subgroup $B$ of $H$ is flat: indeed, there are arbitrarily small tidy subgroups for $B$ of the form $\bZ^n_p$.\end{ex}

Note that if $H$ is a closed compactly generated subgroup of $G$ that is almost flat, then it is almost finite-rank flat: any cocompact flat subgroup $K$ is compactly generated, so that $K/K_{\us}$ is finitely generated.

We also introduce a notion that is stronger than being flat, and is not satisfied in general even by cyclic groups.

\begin{defn}A group of automorphisms $H$ of a \tdlc group $G$ is \defbold{smooth}\index{smooth} (on $G$) if the tidy subgroups for $H$ on $G$ form a base of neighbourhoods of the identity.\end{defn}

Note that $H$ is uniscalar and smooth if and only if $H$ normalizes arbitrarily small compact open subgroups.  Given Van Dantzig's theorem, this situation is in turn equivalent to $H$ having \defbold{small invariant neighbourhoods (SIN)}\index{SIN} in its conjugation action on $G$: A SIN action on a topological group is one for which there exist arbitrarily small neighbourhoods of the identity left invariant by the action.

Although a subgroup can have virtually flat or virtually smooth action on $G$ without having flat action, the (relative) flat and smooth properties are inherited from cocompact \emph{uniscalar} subgroups.

\begin{lem}\label{cocompact:uniscalar_flat}
Let $G$ be a \tdlc group, let $H$ be a closed subgroup of $G$.  Suppose there is a closed subgroup $K$ of $H$ such that $K$ is cocompact in $H$ and such that $K$ is flat and uniscalar on $G$.  Then every open $K$-invariant subgroup of $G$ contains a compact open $K$-invariant subgroup of $H$.  In particular, $H$ is flat and uniscalar on $G$, and if $K$ is smooth on $G$, then so is $H$.
\end{lem}

\begin{proof}
Suppose that $K$ is flat on $G$, and let $O$ be an open $K$-invariant subgroup of $G$.  Then there is a compact open subgroup $U$ of $G$ that is tidy for $K$; by replacing $U$ with $U \cap O$, we may assume $U \le O$.  If $K$ is smooth, then $U$ can be made arbitrarily small.  Since $K$ is uniscalar, in fact $K$ normalizes $U$.  Now $H = XK$, where $X$ is a compact set, so
\[
V = \bigcap_{h \in H}hUh\inv = \bigcap_{x \in X}xUx\inv
\]
is a compact open subgroup normalized by $H$ such that $V \le U$.  In particular $H$ is uniscalar on $G$, and also $V$ is tidy for $H$, so $H$ is flat on $G$.  If $K$ is smooth, then $V$ can be made arbitrarily small, so $H$ is smooth.
\end{proof}

\subsection{Metrizability}

A topological space (or group) is \defbold{metrizable}\index{metrizable} if it is homeomorphic to a metric space.  Not all \tdlc groups are metrizable, and for the most part we do not need to restrict to the metrizable case, but occasionally it will be necessary to do so.  Here are some equivalent conditions.

\begin{lem}\label{basic:metrizable}Let $G$ be a \tdlc group.  Then the following are equivalent.
\begin{enumerate}[(i)]
\item $G$ is metrizable;
\item $G$ is first countable, that is, there is a countable base of neighbourhoods of the identity;
\item $G$ contains a Polish (that is, separable and completely metrizable) open subgroup;
\item Every compact subgroup of $G$ has only countably many open subgroups;
\item Every non-discrete compact subgroup of $G$ is homeomorphic to the Cantor set;
\item $G$ is either discrete or homeomorphic to a disjoint union of copies of the Cantor set.
\end{enumerate}
\end{lem}

\begin{proof}It is clear that if $U$ is a compact open subgroup of $G$, then $G$ is homeomorphic to a disjoint union of copies of $U$ (via the partition of $G$ into left cosets of $U$), so $G$ is metrizable if and only if $U$ is metrizable.  The other properties are also stable on passing between $G$ and $U$.  Hence we may assume $G$ is profinite.  It is also clear that each of the conditions (iii), (iv), (v) and (vi) implies metrizability.

By \cite[Proposition~4.1.3]{Wilson}, $G$ is metrizable if and only if it is an inverse limit of a countable sequence of finite groups.  An inverse limit of countably many finite groups is evidently first countable. Conversely, by Van Dantzig's Theorem any base of neighbourhoods of the identity in a \tdlc group can be replaced by one of the same size consisting of compact open subgroups, so a first countable profinite group $G$ has a base of neighbourhoods of the identity consisting of countably many open subgroups, from which we conclude that $G$ is an inverse limit of a countable sequence of finite groups, and that $G$ has only countably many open subgroups in total (since only finitely many open subgroups of a compact group can contain a given open subgroup).  Hence (i) and (ii) are equivalent and (i) implies (iv).

Under the assumption that $G$ is an inverse limit of a countable sequence of finite groups, it is easily verified that $G$ is either finite or homeomorphic to the Cantor set, and thus $G$ is Polish; moreover, every closed subgroup of a Polish group is Polish.  So (i) implies (iii), (v) and (vi), completing the proof that all six conditions are equivalent.
\end{proof}

\section{The relative Tits core}

Contraction groups in \tdlc groups have useful stability properties, which translate well to the context of relative Tits cores.  In particular, the group $G^\dagger_X$ is less sensitive to the choice of $X$ than one might expect, as will be shown in Theorem~\ref{titscore_closure} below.  First, we recall some prior work on stability properties of contraction groups.

\subsection{Prior results on stability of the contraction group}

The following result on contraction groups was proved by Baumgartner--Willis for metrizable \tdlc groups, then extended to the general \tdlc case by Jaworski.  (The analogous assertion does not hold in general for connected locally compact groups: see \cite[Example~4.1]{JawDistal}.)

\begin{thm}[\cite{BaumgartnerWillis} Theorem 3.8, \cite{Jaw} Theorem 1]\label{thm:bw:relative_contraction}Let $G$ be a \tdlc group, let $\alpha \in \Aut(G)$ and let $H$ be a closed subgroup of $G$ such that $\alpha(H) = H$.  Let $\mc{O}(G)$ be the set of all identity neighbourhoods in $G$.  Define
$$\con_{G/H}(\alpha) := \{ x \in G \mid \forall U \in \mc{O}(G) \; \exists n \; \forall n' \ge n : \alpha^{n'}(x) \in UH \}.$$
Then $\con_{G/H}(\alpha) = \con_G(\alpha)H$.\end{thm}

In particular, combining Theorem~\ref{thm:bw:relative_contraction} with Proposition~\ref{basic:anisotropic}, we have a criterion for an automorphism to have anisotropic action on a subquotient of $G$.

\begin{cor}\label{cor:anisotropic_quotient}Let $G$ be a \tdlc group, let $\alpha \in \Aut(G)$, and let $H$ and $K$ be closed $\alpha$-invariant subgroups of $G$ such that $K$ is normal in $H$.  A sufficient condition for $\alpha$ to have anisotropic action on $H/K$ is that $G^\dagger_{\alpha} \le K$.  If $H$ is open in $G$, this condition is also necessary.\end{cor}

The stability of contraction groups was also investigated in \cite{CRW-TitsCore}.

\begin{prop}[\cite{CRW-TitsCore}, Lemma~4.1 and Corollary~4.2]\label{contraction:conjugate}Let $G$ be a \tdlc group.  Let $g \in G$ and let $U$ be a compact open subgroup of $G$ that is tidy above for $g$.  Then for every $u \in U$, there exists $t \in U_+ \cap \con(g\inv)$ such that
\[
\con(gu) = t\con(g)t\inv.
\]
\end{prop}

\begin{prop}[\cite{CRW-TitsCore}, Proposition~5.1]\label{prop:ContractionGroup:NC}
Let $G$ be a \tdlc group and let $A$ be a (not necessarily closed) subgroup of $G$.  Given any $g \in A$, if $\overline{\con(g)}$ normalizes $A$, then $\overline{\con(g)} \le A$.  In particular, any normal subgroup of $G$ containing $g$ also contains $\overline{\con(g)}$.
\end{prop}

We note the following variant of Proposition~\ref{contraction:conjugate} for convenience.

\begin{cor}\label{contraction:conjugate:bis}
Let $G$ be a \tdlc group.  Let $g \in G$ and let $U$ be a compact open subgroup of $G$ that is tidy above for $g$.  Then for every $u \in U$, there exists $t \in U_+ \cap \con(g\inv)$ such that
\[
\con(ug) = t\con(g)t\inv.
\]
\end{cor}

\begin{proof}Let $V = g\inv U g$; note that $V$ is tidy above for $g$.  We have $ug = gv$ where $v = g\inv u g \in V$, so by Proposition~\ref{contraction:conjugate}, there exists $t \in V_+ \cap \con(g\inv)$ such that
\[
\con(ug) = \con(gv) = t\con(g)t\inv.
\]
Moreover, $V_+ = g\inv U_+g \le U_+$ by the definition of $U_+$, so $t \in U_+ \cap \con(g\inv)$.
\end{proof}

There is a straightforward condition for when the contraction group of an element is the same as its contraction group acting on a closed subgroup.

 \begin{lem}\label{lem:samecontraction}Let $G$ be a \tdlc group, let $g \in G$ and let $K$ be a closed $\langle g \rangle$-invariant subgroup of $G$.  Then $\con_K(g) = \con(g) \cap K$.  In particular, we have $K^\dagger = G^\dagger_K$ if and only if $G^\dagger_K \le K$.\end{lem}
 
 \begin{proof}
 Let $g \in G$.  Given $u \in \con_K(g)$, then for all open subgroups $U$ of $G$, we have $g^nug^{-n} \in K \cap U \le U$ for $n$ sufficiently large, since $K \cap U$ is an open subgroup of $K$.  Thus $u \in \con(g) \cap K$.  Conversely, given $u \in \con(g) \cap K$, then $g^nug^{-n} \in K$ for all $n \ge 0$ by hypothesis, so given an open subgroup $U$ of $G$, we have $g^nug^{-n} \in K \cap U$ for $n$ sufficiently large.  Since the subgroups $K \cap U$ form a base of identity neighbourhoods in $K$ as $U$ ranges over the open subgroups of $G$, it follows that $u \in \con_K(g)$.  Thus $\con_K(g) = \con(g) \cap K$.  The last conclusion is clear.
\end{proof}

\subsection{Invariance of contraction groups and the relative Tits core}\label{sec:TitsCore}

Let us consider the implications of Proposition~\ref{contraction:conjugate} for conjugacy classes of contraction groups, and hence of relative Tits cores.

\begin{prop}\label{prop:twosided:stable}Let $G$ be a \tdlc group and let $g \in G$ and define $L_g = \langle \con(g),\con(g\inv) \rangle$.  Let $U$ be an open subgroup of $G$ that is tidy for $g$, let $u,v \in U$ and let $n > 0$.  
\begin{enumerate}[(i)]
\item There exists $t \in U \cap L_g$ such that
\[
\con(ug^nv) = t\con(g)t\inv.
\]
\item We have $L_{ug^nv} = L_g$ and $G^\dagger_{ug^nv} = G^\dagger_g$.  In particular, the normalizers of $L_g$ and $G^\dagger_g$ both contain $U$, and are thus open subgroups of $G$.
\end{enumerate}
\end{prop}

\begin{proof}
By Lemma~\ref{lem:contraction:powers}, we have $\con(g) = \con(g^n)$, and by Lemma~\ref{basic:tidy:powers}, $U$ is tidy for $g^n$; thus we may assume $n=1$.  By Proposition~\ref{contraction:conjugate}, there is $t_1 \in U \cap \con(g\inv)$ such that $\con(gv) = t_1\con(g)t\inv_1$.  In particular, $\con(gv) \le L_g$.  Similarly, $\con(v\inv g\inv)$ is conjugate to $\con(g\inv)$ under the action of $\con(g)$, so $\con(v\inv g\inv) \le L_g$.  Now $U$ is tidy for $gv$ by Theorem~\ref{tidy:stable}, so by Corollary~\ref{contraction:conjugate:bis}, there is $t_2 \in U \cap \con(v\inv g\inv)$ such that $\con(ugv) = t_2\con(gv)t\inv_2$.  Now set $t = t_2t_1$, and observe that $t \in U \cap L_g$ and that $\con(ugv) = t\con(g)t\inv$.

Since $\con(g) \le L_g$ and $\con(ugv)$ is $L_g$-conjugate to $\con(g)$, we have $\con(ugv) \le L_g$.  Likewise $\con(v\inv g\inv u\inv) \le L_g$, so $L_{ugv} \le L_g$.  By the same argument $L_g \le L_{ugv}$, since $U$ is tidy for $ugv$, so $L_{ugv} = L_g$.  The proof that $G^\dagger_{ug^nu'} = G^\dagger_g$ is similar.
\end{proof}

We are now able to give several invariance properties of relative Tits cores.

\begin{thm}\label{titscore_closure}Let $G$ be a \tdlc group and let $X$ be a subset of $G$.  Let $Y$ be the set
\[
Y = \{g \in G \mid \overline{\con(g)}, \overline{\con(g\inv)} \le G^\dagger_X\}.
\]
Then the following properties hold.
\begin{enumerate}[(i)]
\item $Y$ is a clopen subset of $G$ that contains all anisotropic elements of $G$.
\item Let $g \in G$ and let $n$ be a non-zero integer.  Then $g \in Y$ if and only if $g^n \in Y$.
\item The normalizer of $Y$ is closed in $G$ and is equal to the normalizer of $G^\dagger_X$ in $G$.
\item Suppose there exists a compact open subgroup $U$ of $G$ such that for all $g \in X$, there exists $V \ge U$ such that $V$ is tidy for $g$.  (For example, $X$ could be a union of finitely many flat subgroups of $G$.)  Then $U \le \N_G(G^\dagger_X)$; in particular, $\N_G(G^\dagger_X)$ is open in $G$.
\item Let $R = \overline{G^\dagger_X}$.  Then 
$$R^\dagger = G^\dagger_R \le G^\dagger_X,$$
so $R \subseteq Y$.  In particular, if $G^\dagger_X$ is dense in $G$, then $G^\dagger_X = G^\dagger$.
\end{enumerate}
\end{thm}

\begin{proof}
Let us note first that $G^\dagger_X = G^\dagger_Y$, by the definition of $Y$.

(i)
We see from Proposition~\ref{prop:twosided:stable} that $Y$ is open.  Evidently $Y$ contains all anisotropic elements of $G$.

Let $k \in \overline{Y}$.  Then there is a compact open subgroup $U$ of $G$ that is tidy for $k$, and moreover we have $h \in kU$ for some $h \in Y$.  Hence $G^\dagger_k = G^\dagger_h$
by Proposition~\ref{prop:twosided:stable}, so $G^\dagger_k \le G^\dagger_Y = G^\dagger_X$, in other words $k \in Y$.  Hence $Y$ is closed.

(ii) follows immediately from Lemma~\ref{lem:contraction:powers}.

(iii)
We have $\N_G(Y) \le \N_G(G^\dagger_X)$, since $G^\dagger_X = G^\dagger_Y$ is determined by $G$ and $Y$.  Given $g \in \N_G(G^\dagger_X)$ and $y \in Y$, then 
$$G^\dagger_{gyg\inv} = gG^\dagger_y g\inv \le gG^\dagger_Yg\inv = G^\dagger_Y,$$
so $gxg\inv \in Y$.  Since $y \in Y$ was arbitrary we have $gYg\inv \subseteq Y$, and by symmetry in fact $gYg\inv = Y$.  So $g \in \N_G(Y)$ and hence $\N_G(Y) = \N_G(G^\dagger_X)$.

Let $N = \N_G(Y)$, let $r \in \overline{N}$ and let $y \in Y$.  Then $r$ can be approximated in $G$ by elements of $N$, so given a compact open subgroup $U$ of $G$ that is tidy for $ryr\inv$, there exists $s \in N$ such that $s \in Ur$.  By  Proposition~\ref{prop:twosided:stable} we have $G^\dagger_{ryr\inv} = G^\dagger_{sys\inv}$, and since $s \in \N_G(Y)$ we have $G^\dagger_{sys\inv} \le G^\dagger_Y$.  Since $y \in Y$ was arbitrary (in particular, independent of the choice of $r$), we have $rG^\dagger_Yr\inv \le G^\dagger_Y$, and by symmetry in fact $rG^\dagger_Yr\inv = G^\dagger_Y$, so $r \in N$.  Hence $N$ is closed.

(iv)
Let $U$ be as in the statement.  Then by Proposition~\ref{prop:twosided:stable}, we have $G^\dagger_{uxv} = G^\dagger_x$ for all $u,v \in U$ and $x \in X$.  Hence $G^\dagger_X = G^\dagger_{UXU}$, so $G^\dagger_X$ is normalized by $U$.

(v)
Let $g \in G^\dagger_X$.  Then $g = u_1u_2\dots u_n$, where $u_i \in G^\dagger_{x_i}$ for some $x_i \in X$.  Thus $g \in G^\dagger_Z$, where $Z$ is a finite subset of $X$.  By part (iv), $H = \N_G(G^\dagger_Z)$ is clopen.  We see that $\con(g) \le H$, since $H$ is a $g$-invariant neighbourhood of the identity, and hence $\overline{\con(g)} \le \N_G(G^\dagger_Z)$.  By Proposition~\ref{prop:ContractionGroup:NC}, it follows that $\overline{\con(g)} \le G^\dagger_Z$, and hence $\overline{\con(g)} \le G^\dagger_X$.  Since $g \in G^\dagger_X$ was arbitrary, we conclude that $G^\dagger_X \subseteq Y$, and since $Y$ is closed by part (i), we in fact have $R \subseteq Y$, that is, $G^\dagger_R \le G^\dagger_X$.

Given $r \in R$, we have seen that $\con_G(r) \le G^\dagger_X$, so $\con_G(r) \le R$.  It follows from Lemma~\ref{lem:samecontraction} that in fact $\con_G(r) = \con_R(r)$.  Hence $R^\dagger = R^\dagger_R = G^\dagger_R$.
\end{proof}

We now prove Theorem~\ref{thmintro:TitsCore:cocompact}, starting with a lemma.

\begin{lem}\label{lem:cocompact_power}
Let $G$ be a \tdlc group and let $G \ge H \ge K$ such that $H$ and $K$ are closed and $K$ is cocompact in $H$.  Let $h \in H$ and let $U$ be a compact open subgroup of $G$.  Then there exist $a,b \in \bZ$ and $v,w \in h^{-b}Uh^{b}$ such that $a > 0$ and $vh^aw \in K$.
\end{lem}

\begin{proof}
Let $R = \overline{\langle h \rangle K}$.  Then $R/K$ is a closed, hence compact, subspace of $H/K$.  The sequence $(h^iK)_{i \in \bN}$ thus has an accumulation point $rK$ say in the topology of $H/K$, where $r \in R$.

We see that there exist $i,j \in \bZ$ with $j > i$ such that $\{h^i,h^j\} \subseteq UrK$.  Moreover, since $r \in \overline{\langle h \rangle K}$, we can write $r$ as $r = uh^bk$ for some $u \in U$, $b \in \bZ$ and $k \in K$, so in fact $\{h^i,h^j\} \subseteq Uh^bK$, and hence $h^{a} \in Uh^bKh^{-b} U$, where $a = j-i > 0$.  After rearranging, we have $vh^{a}w \in K$ for $v,w \in h^{-b}Uh^{b}$.
\end{proof}

\begin{proof}[Proof of Theorem~\ref{thmintro:TitsCore:cocompact}]
Let us first consider the case when $K$ is dense in $H$.  Given $h \in H$ and a tidy subgroup $U$ for $h$, we have $h = ku$ for some $k \in K$.  Now $U$ is tidy for $k$ by Theorem~\ref{tidy:stable}, so by Proposition~\ref{contraction:conjugate}, there is $t \in \con(k\inv) \le G^\dagger_k$ such that $\con(h) = t\con(k)t\inv$, and we have $G^\dagger_h = G^\dagger_k$ by Proposition~\ref{prop:twosided:stable}.  So from now on, we may suppose that $K$ is closed in $H$.

Let $h \in H$ and let $U$ be a compact open subgroup of $G$ that is tidy for $h$.  We claim that there exist $a > 0$, $k \in K$ and $u$ and $v$ in a tidy subgroup for $h$ such that $k = uh^av$.

We first consider the case when $K$ has finite covolume in $H$ and let $\mu$ be an invariant probability measure on $H/K$.  Then $(H \cap U)K/K$ is a non-empty open subset of $H/K$, so $\mu((H \cap U)K/K) = \epsilon > 0$; by translation invariance, $\mu(h^n(H \cap U)K/K) = \epsilon$ for all integers $n$.  By finite additivity of the measure, there exist distinct integers $i < j$ such that
\[
h^i(H \cap U)K \cap h^j(H \cap U)K \neq \emptyset.
\]
In particular, $h^{a}(H \cap U)$ has non-empty intersection with $(H \cap U)K$ where $a = j-i > 0$, so there exists $u,v \in H \cap U$ and $k \in K$ such that $h^{a}v = u\inv k$.

Now suppose instead that $K$ is cocompact in $H$.  By Lemma~\ref{lem:cocompact_power} we obtain the equation $uh^{a}v = k$, where $a > 0$, $u,v \in h^{-b}Uh^{b}$ for some $b \in \bZ$ and $k \in K$.  Note that $h^{-b}Uh^{b}$ is tidy for $h$ by Theorem~\ref{tidy:stable}.

In either case, we see by Proposition~\ref{prop:twosided:stable} that $G^\dagger_k = G^\dagger_h$, and moreover there exists $t \in G^\dagger_h = G^\dagger_k$ such that $\con(h) = t\con(k)t\inv$.  The remaining assertions are now clear.
\end{proof}

The following corollary is now clear from Theorem~\ref{thmintro:TitsCore:cocompact} and Theorem~\ref{titscore_closure}(iv).

\begin{cor}\label{almostflat:titscore}Let $G$ be a \tdlc group and let $H$ be an almost flat subgroup of $G$.  Then $\N_G(G^\dagger_H)$ is open in $G$.\end{cor}

\subsection{Relative Tits cores and elementary groups}\label{sec:TitsCoreElementary}

We note some interesting features of the group $T = \overline{G^\dagger_H}$ in the case that $H$ is a compactly generated subgroup of $G$ (even just the case that $H$ is cyclic is interesting enough).

\begin{enumerate}[(1)]
\item It is clear that $H$ does not normalize any proper open subgroup of $T$.  From this, one can easily deduce that $L = \overline{TH}$ is compactly generated.  Indeed, $L = \langle H,U \rangle$ for any compact open subgroup $U$ of $L$.
\item In general, it is possible that $T$ contains $H$, intersects $H$ trivially, or some intermediate situation.  In general, there is little insight to be gained from the hypothesis that $T \cap H$ is trivial, since for instance this situation will occur whenever $G$ is a semidirect product $N \rtimes H$, where $H$ is any finitely generated group (equipped with the discrete topology) that acts by automorphisms on the \tdlc group $N$.  On the other hand, the existing literature suggests that an important special case is when $T$ is non-trivial and cocompact in $\overline{TH}$.  (Even when $H$ is cyclic, this is not the same thing as asking if $T \cap H$ has finite index in $H$; recall \S\ref{sec:Neretin}.)  It then follows from Theorem~\ref{titscore_closure} and Theorem~\ref{thmintro:TitsCore:cocompact} that
\[
T = \overline{G^\dagger_H} = \overline{G^\dagger_K} = \overline{T^\dagger}.
\]
Thus we have a compactly generated \tdlc group $T$ with dense Tits core; in particular, $T$ has no proper open normal subgroups.  One can then apply \cite[Proposition~5.4]{CM} to conclude that every proper closed normal subgroup of $T$ is contained in a maximal one, and that $T$ has $n$ topologically simple quotients for some positive integer $n$.  In particular, it follows that $T$ does not have any non-trivial elementary quotients in the sense of Wesolek (see \cite{Wesolek}), so that both $T$ and $G$ itself are non-elementary.
\end{enumerate}

If $G$ is an elementary \tdlc group and $g \in G \setminus \{1\}$, we conclude from the above observations that $g \not\in \overline{G^\dagger_g}$.  Conversely, as indicated by Question~\ref{que:TitsCoreElementary}, the author does not know of any counterexamples to the following statement:

$(*)$ Let $G$ be a non-elementary second-countable \tdlc group.  Then there is some non-trivial element $g$ such that $g \in G^\dagger_g$.

The statement $(*)$ is probably too ambitious and reflects a lack of knowledge of examples, but even weaker results of this kind could be highly significant for the general theory of \tdlc groups.  Proving $(*)$ to be true in general would prove all of the following statements:
\begin{enumerate}[(A)]
\item Given a compactly generated \tdlc group $G$, then $\Res(G)=G$ if and only if $G^\dagger$ is dense in $G$.
\item Given a compactly generated \tdlc group $G$ with no non-trivial discrete quotients, then $G^\dagger$ is the unique smallest dense subnormal subgroup of $G$.  If $G$ is topologically simple, then $G^\dagger$ is abstractly simple.
\item A second-countable \tdlc group $G$ is elementary if and only if there does not exist $K \unlhd H \le G$ such that $H/K$ is non-discrete, compactly generated and topologically simple.  (The `only if' follows from the fact that every closed subgroup of an elementary group is elementary; see \cite[Theorem~1.3]{Wesolek}.)
\item The set $E$ of closed subgroups of $\mathrm{Sym}(\bN)$ that are elementary second-countable \tdlc groups belongs to the Effros-Borel $\sigma$-algebra of $\mathrm{Sym}(\bN)$.
\item Letting $\mathscr{E}$ be the class of elementary second-countable \tdlc groups and writing $\xi(G)$ for the decomposition rank of $G$ (see \cite[\S4.3]{Wesolek}), then the supremum of $\{\xi(G) \mid G \in \mathscr{E}\}$ is a countable ordinal, which is achieved by $\xi(G)$ for some $G \in \mathscr{E}$.
\end{enumerate}

The derivation of (B) from (A) is given in \cite{CRW-TitsCore}.  The last two statements require some further explanation.

Let $X$ be a Polish space and let $\mathcal{F}$ be the set of all non-empty closed subsets of $X$. The \defbold{Effros-Borel $\sigma$-algebra} of $X$ is the smallest $\sigma$-algebra $\mathcal{E}(X)$ on $X$ containing the sets
\[
\{ F \in \mathcal{F}(X) \mid F \cap U \neq \emptyset\}, \; U \subseteq X \text{ open}.
\]
If $X$ is a locally compact space, $\mathcal{E}(X)$ coincides with the Borel $\sigma$-algebra of the Vietoris space.  In general, there is no standard topology on $\mathcal{F}(X)$, but nevertheless $\mathcal{E}(X)$ is isomorphic to a standard Borel $\sigma$-algebra; for instance, a suitable isomorphism is induced by identifying $X$ with a $G_{\delta}$-subset of the Hilbert cube.

The supremum $\alpha$ of $\{\xi(G) \mid G \in \mathscr{E}\}$ is achieved by $G \in \mathscr{E}$ if and only if $\alpha < \omega_1$: if the supremum is achieved, then $\alpha$ must be countable by the definition of the decomposition rank, and conversely if $\alpha$ is countable, then it is the supremum of $\{\xi(G_i) \mid i < \omega\}$ for a countable sequence $(G_i)_{i < \omega}$ of elementary groups; one can then construct a local direct product $G$ of the groups $(G_i)_{i < \omega}$ (see \cite{Wesolek}) so that $\xi(G) = \alpha$.

One can show (analogous to the situation with elementary amenable discrete groups; see \cite[\S6.4]{WesolekWilliams}) that $E$ is a ${\bf \Pi}^1_1$-set and the decomposition rank function $\xi$ is a ${\bf \Pi}^1_1$-rank on $E$.  In particular, we have $E \in \mathcal{E}(\mathrm{Sym}(\bN))$ if and only if the image of $E$ under $\xi$ is bounded below $\omega_1$.  Since every second-countable \tdlc group is isomorphic to a closed subgroup of $\mathrm{Sym}(\bN)$, we conclude that the statements (D) and (E) are equivalent.

It remains to deduce (D) from $(*)$; here we use the invariance properties of the relative Tits core to prove an unconditional result about the class of \tdlcsc groups that satisfy $(*)$.

\begin{lem}\label{lem:inner_type}
Let $G$ be a \tdlc group and let $P(G)$ be the set of elements $g$ of $G$ such that $\langle g \rangle$ is an infinite discrete group and every open subgroup normalized by $g$ contains a non-zero power of $g$.
\begin{enumerate}[(i)]
\item Every element $g \in P(G)$ is isotropic.
\item Given $g \in G$ isotropic, then $g \in P(G)$ if and only if $\overline{G^\dagger_g}$ is cocompact in $\overline{\langle G^\dagger_g, g\rangle}$.
\item $P(G)$ is open in $G$.
\item $P(G)$ is non-empty if and only if there exists $g \in G \setminus \{1\}$ such that $g \in G^\dagger_g$.
\end{enumerate}
\end{lem}

\begin{proof}
Let $g \in P(G)$.  If $g$ is anisotropic, then $g$ normalizes arbitrarily small compact open subgroups of $G$.  In particular, $g$ normalizes some compact open subgroup $U$ of $G$.  By the definition of $P(G)$, we also have $g^n \in U$ for some non-zero integer $n$.  Thus $\langle g \rangle U$ consists of only finitely many cosets of $U$, so it is compact.  This contradicts the condition that $\langle g \rangle$ should be an infinite discrete group.  Thus $g$ is isotropic, proving (i).

Now let $g \in G$ be isotropic and let $T = \overline{G^\dagger_g}$.

If $T$ is cocompact in $\overline{\langle T, g \rangle}$, then every open subgroup that is normalized by $g$ contains $T$, and every open subgroup containing $T$ contains a finite index subgroup of $\overline{\langle T, g\rangle}$, hence also contains a non-zero power of $g$.  So $g \in P(G)$.  Conversely, if $T$ is not cocompact in $\overline{\langle T, g\rangle}$, then $\langle gT \rangle$ is a non-compact subgroup of $\N_G(T)/T$, so $\langle gT \rangle$ is an infinite cyclic group that has trivial intersection with every compact open subgroup of $\N_G(T)/T$.  At the same time, $gT$ has anisotropic action by conjugation on $\N_G(T)/T$ by Corollary~\ref{cor:anisotropic_quotient}, so $gT$ normalizes a compact open subgroup $V/T$ of $\N_G(T)/T$.  Moreover $\N_G(T)$ is open in $G$ by Proposition~\ref{prop:twosided:stable}, so $V$ is open in $G$.  Thus $V$ is an open subgroup normalized by $g$ that does not contain any non-zero power of $g$, proving (ii). 

Now suppose $g \in P(G)$.  Then $g$ is contained in a compact open subgroup $V/T$ of $\N_G(T)/T$ by the characterization of $P(G)$ given in part (ii).  Since $\N_G(T)$ is open in $T$, in fact $V$ is open in $G$.  Given a compact open subgroup $U$ of $G$ that is tidy for $g$ and $u \in U$, then $u$ normalizes $G^\dagger_g$ and $G^\dagger_g = G^\dagger_{gu}$.  We now observe that for all $h \in gU \cap V$, then $h$ is isotropic on $G$ (so in particular, $\langle h \rangle$ is an infinite discrete group) and the group $\overline{G^\dagger_h} = \overline{G^\dagger_g}$ is cocompact in $\overline{\langle G^\dagger_h, h \rangle}$, so that $h \in P(G)$ by part (ii).  Hence $P(G)$ is open in $G$, proving (iii).

Given $g \in P(G)$ and $U$ tidy for $g$, we see by part (ii) and Proposition~\ref{prop:twosided:stable} that $g^nu \in G^\dagger_g = G^\dagger_{g^nu}$ for some $u \in U$ and non-zero integer $n$.  Conversely, if $g \in G \setminus \{1\}$ is such that $g \in G^\dagger_g$, then clearly $g$ is isotropic, and hence $g \in P(G)$ by part (ii).  Thus $P(G)$ is non-empty if and only if $g \in G^\dagger_g$ for some non-trivial $g \in G$, proving (iv).
\end{proof}

\begin{thm}\label{thm:inner_type:Borel}
Let $S = \Sym(\bN_{\ge 0})$ and let $E^*$ be the set of closed locally compact subgroups $G$ of $S$ such that for all $g \in G \setminus \{1\}$ we have $g \not\in G^\dagger_g$.

Then $E^* \in \mathcal{E}(S)$.
\end{thm}

\begin{proof}
By the Kuratowski--Ryll-Nardzewiski selector theorem, there is a sequence $(d_n)_{n \in \omega}$ of measurable functions $d_n: \mathcal{F}(S) \rightarrow S$ such that for each $F \in \mathcal{F}(S)$, the set $\{d_n(F)\}_{n \in \omega}$ is a dense subset of $F$.  Let us fix such a sequence $d_n$.  Fix also a countable descending chain $(U_n)_{n \ge 0}$ of clopen symmetric neighbourhoods of the identity in $S$, forming a base of identity neighbourhoods in $S$.

Let $F^*$ be the set of closed locally compact groups $G$ of $S$ such that there exists $g \in G \setminus \{1\}$ for which $g \in G^\dagger_g$.  Let $L_i$ be the set of subgroups $G$ of $S$ such that $G \cap U_i$ is compact.  The conditions of being a subgroup and having compact intersection with an open set are measurable conditions on non-empty closed subsets of $S$.  Thus $L_i \in \mathcal{E}(S)$.

Now let $G \in L_0$.  Then $G$ is locally compact, and by Lemma~\ref{lem:inner_type}, we have $G \in F^*$ if and only if the open set $P(G)$ is non-empty.  In turn, $P(G)$ is non-empty if and only if $d_n(G) \in P(G)$ for some $n$.  Write $g_n := d_n(G)$.  Then $g_n \in P(G)$ if and only if the following formulae (quantified over $\bN_{\ge 0}$) are both satisfied:

\begin{equation}\label{eq:inner_type:Borel1}
\forall a \forall b: g_n^{b+1} \not\in U_a;
\end{equation}
\begin{equation}\label{eq:inner_type:Borel2}
\forall c \exists d \forall e \exists f, k_1,\dots,k_f,l_1,\dots,l_f: d_{k_1}(V_{c,n,l_1}) \dots d_{k_f}(V_{c,n,l_f}) \in U_eg_n^{d+1};
\end{equation}
where
\[
V_{c,n,l} := g_n^{l}(U_c \cap G)g_n^{-l} \cup g_n^{-l}(U_c \cap G)g_n^{l}.
\]
Specifically, (\ref{eq:inner_type:Borel1}) is equivalent to stating that $\langle g_n \rangle$ is infinite and discrete, and (\ref{eq:inner_type:Borel2}) is equivalent to stating that for every identity neighbourhood $U$ in $G$, then there is a positive power $g_n^{d+1}$ of $g_n$ that can be approximated by words in the $\langle g_n\rangle$-conjugates of $U$, so that $g_n^{d+1}$ is in the (necessarily closed) group generated by the $\langle g_n\rangle$-conjugates of $U$.  These formulae impose a measurable condition on $g_n$, and so `there exists $n$ such that $g_n$ satisfies (\ref{eq:inner_type:Borel1}) and (\ref{eq:inner_type:Borel2})' is a measurable condition on $G$.  We therefore conclude that $F^* \cap L_0 \in \mathcal{E}(S)$, so $E^* \cap L_0 \in \mathcal{E}(S)$.

The same argument shows that the sets $E^* \cap L_i$ are in $\mathcal{E}(S)$.  Note that a subgroup $G$ of $S$ is closed and locally compact if and only if $G \cap U_i$ is compact for some $i$.  Thus $E^* = \bigcup_{i \ge 0}(E^* \cap L_i)$, so $E^* \in \mathcal{E}(S)$ as required. 
\end{proof}

The statement $(*)$ is then equivalent to asserting that $E^* = E$.  In particular, we see that $(*)$ implies (D) as claimed.

\subsection{Subgroups containing relative Tits cores}\label{sec:TitsCore:subgroups}

There is no reason for an arbitrary subgroup $D$ of $G$ to contain $G^\dagger_D$.  For example, if $G$ is the automorphism group of a locally finite tree, then $G^\dagger_g$ is open in $G$ for every hyperbolic element $g \in G$ (see Example~\ref{ex:tree}), so certainly $G^\dagger_g \not\le \langle g \rangle$.  However, we can ensure $G^\dagger_D \le D$ under certain circumstances, as stated in Theorem~\ref{thmintro:titscore:containment}.

\begin{proof}[Proof of Theorem~\ref{thmintro:titscore:containment}]We may assume that $X = X\inv$.  Let $U$ be an open subgroup of $G$ such that $U \cap G^\dagger_X \le \N_G(D)$.

Let $x \in X$.  By Proposition~\ref{prop:twosided:stable} we have $G^\dagger_x = G^\dagger_d$ for all $d \in VxV$, where $V$ is a compact open subgroup of $G$ that is tidy for $x$.  Since $X \subseteq \overline{D}$, there exists $d \in VxV \cap D$: for this $d$, we see that $U \cap G^\dagger_d = U \cap G^\dagger_x \le \N_G(D)$.

Let $u \in \con(d)$.  Then for $n \ge 0$ sufficiently large, we have $d^nud^{-n} \in U$, and thus $d^nud^{-n} \in \N_G(D)$.  But $\N_G(D)$ is $D$-invariant, so $u \in \N_G(D)$, and hence $\con(d) \le \N_G(D)$.  In addition, $\N_G(D)$ contains the open subgroup $U \cap \nub(d)$ of $\nub(d)$.  By Theorem~\ref{basic:nub_characterizations} there are no proper $d$-invariant open subgroups of $\nub(d)$ , so $\nub(d) \le \N_G(D)$.  Thus $\overline{\con(d)} \le \N_G(D)$ by Theorem~\ref{basic:closed_contraction}.  The same argument shows that $\overline{\con(d\inv)} \le \N_G(D)$, so in fact $G^\dagger_d \le \N_G(D)$.  Hence by Proposition~\ref{prop:ContractionGroup:NC}, we have $G^\dagger_d \le D$, so $G^\dagger_x \le D$.

As $x \in X$ was arbitrary, we conclude that $G^\dagger_X \le D$.\end{proof}

\begin{cor}\label{cor:subnormal:containment}Let $G$ be a \tdlc group, let $A$ be a subgroup of $G$, and let $B \subseteq A$.  Then the following are equivalent:
\begin{enumerate}[(i)]
\item $G^\dagger_{\overline{B}} \le A$;
\item There exists a subgroup $H$ of $G$ such that $A$ is subnormal in $H$ and $G^\dagger_B \le H$.
\end{enumerate}
\end{cor}

\begin{proof}Clearly (i) implies (ii), as we can take $A = H$ in this case.  Suppose (ii) holds, and let 
\[
A = A_0 \lhd A_1 \lhd \dots \lhd A_n = H
\]
be a subnormal series from $A$ to $H$.  Suppose $n > 0$.  Then by Theorem~\ref{thmintro:titscore:containment}, we have $G^\dagger_{\overline{B}} \le A_{n-1}$, since $G^\dagger_B$ normalizes $A_{n-1}$ and $B \subseteq A_{n-1}$.  The conclusion now follows by induction on the subnormal degree of $B$ in $H$.
\end{proof}

The following is now clear from Corollary~\ref{cor:subnormal:containment} and Theorem~\ref{thmintro:TitsCore:cocompact}.  

\begin{cor}\label{cor:subnormal:sametitscore}Let $G$ be a \tdlc group and let $H$ be a subnormal subgroup of $G$.  Then $G^\dagger_H = (\overline{H})^\dagger$.  If in addition $H$ is closed in $G$ and either cocompact or of finite covolume in $G$, then $G^\dagger = H^\dagger$.\end{cor}

In particular, we have the following strengthening of \cite[Corollary~1.2]{CRW-TitsCore}.

\begin{cor}Let $G$ be a \tdlc group and let $\mc{S}$ be the set of subnormal subgroups $S$ of $G$ such that $\overline{S}$ is cocompact or of finite covolume in $G$.  Then
\[
G^\dagger \le \bigcap_{S \in \mc{S}}S.
\]
\end{cor}

\subsection{Examples}\label{sec:TitsCore_examples}

We give two basic examples that illustrate how the relative Tits core $G^\dagger_g$ can depend very little on the choice of element $g$, even though the groups $\con(g)$ and $\con(g\inv)$ are sensitive to the choice of $g$.

\begin{ex}\label{ex:tree}
Let $T$ be a locally finite regular tree of degree at least $3$, let $G = \Aut(T)$, and let $g \in G$.  If $g$ is elliptic, that is, $g$ fixes a vertex or inverts an edge, then $\con(g) = \con(g\inv) = \triv$.  Otherwise $g$ is hyperbolic, and the set of vertices $v$ such that $d(v,gv)$ is minimised forms a bi-infinite path $L$ in $T$, the axis of $g$.  Identify $L$ with $\bZ$, so that $g(0) > 0$ and $d_T(i,j) = |j-i|$, and let $\pi$ be the nearest point projection from $T$ to $L = \bZ$.  Let $K^-_n$ be the fixator of the set $\pi\inv((-\infty,n))$ and let $K^+_n$ be the fixator of the set $\pi\inv((n,+\infty))$.  Then we see that
\[
\con(g) \ge \bigcup_{n \in \bZ}K^-_n \text{ and } \con(g\inv) \ge \bigcup_{n \in \bZ}K^+_n.
\]
In fact, $\con(g)$ is the set of all elements $h \in G$ such that there exists $k \in L$ and a function $f_h: (-\infty,k] \rightarrow \bN$, with $f_h(n) \rightarrow +\infty$ as $n \rightarrow -\infty$, such that $h$ fixes pointwise the ball of radius $f_h(n)$ about the vertex $n$.  It is easily seen that $\con(g)$ is not closed in $G$: in fact, the closure $\overline{\con(g)}$ consists of all elliptic elements fixing the end $-\infty$ of the axis of $g$.  The normalizer of $\con(g)$ (and also of $\overline{\con(g)}$) is the stabilizer of $-\infty$, so $\N_G(\con(g))$ is a closed but not open subgroup of $G$.  Similarly, $\nub(g)$ is the pointwise stabilizer of the axis of $g$, so $\N_G(\nub(g))$ is not open.  However the subgroups $K^-_n$ and $K^+_n$ are each compact, and the product $K^+_nK^-_{n+1}$ is a compact open subgroup of $G$, being the stabilizer in $G$ of the directed edge $(n,n+1)$.  In turn, it is easily seen that the group generated by the stabilizers of the directed edges $(n,n+1)$ as $n$ ranges over $L$ is in fact the subgroup $G^+$ of $G$ generated by all directed edge stabilizers, which is a simple open subgroup of $G$ of index $2$.  So in this case, every element $g \in G$ satisfies either $G^\dagger_g = G^+ = G^\dagger$ (if $g$ is hyperbolic) or $G^\dagger_g = \triv$ (if $g$ is not hyperbolic).

We will return to a more general class of examples including this one at the end of the paper (see \S\ref{sec:weakdecomp}).
\end{ex}

\begin{ex}\label{ex:linear}(See \cite{Prasad} for a more detailed treatment of a class of examples including this one.)

Let $G= \mathrm{SL}_n(\bQ_p)$ and let $g$ be the diagonal matrix $\mathrm{diag}(\lambda_1,\lambda_2,\dots,\lambda_n)$.  Suppose that
\[
|\lambda_1|_p \ge |\lambda_2|_p \ge \dots \ge |\lambda_n|_p.
\]
Then $\con(g)$ is closed in this case: it is the group of matrices of the form $1+u$, where $u_{ij} = 0$ whenever $|\lambda_i|_p \le |\lambda_j|_p$.  In other words, $\con(g)$ is a group of block upper unitriangular matrices, with the blocks corresponding to intervals of $(\lambda_1,\dots,\lambda_n)$ on which $|\lambda_i|_p$ is constant.  Thus $\con(g)$ is the unipotent radical of a parabolic subgroup $P$, where $P$ consists of all elements $a$ of $G$ such that $a_{ij} = 0$ whenever $|\lambda_i|_p > |\lambda_j|_p$.  In fact $P$ itself can be characterized directly in terms of the dynamics of $g$: it consists of those elements $a \in G$ such that $\{g^nag^{-n} \mid n \ge 0\}$ is relatively compact.  We see also that $\con(g\inv)$ is simply the image of $\con(g)$ under matrix transposition.

A similar description of contraction groups can be given for all elements of $G$ that have non-trivial contraction group.  So a typical non-trivial relative Tits core $G^\dagger_g$ in $G$ is of the form $\langle U, U'\rangle$, where $U$ is the unipotent radical of a parabolic subgroup $P$ and $U'$ is the unipotent radical of a parabolic subgroup opposite to $P$.  By \cite[Proposition 6.2(v)]{BorelTits}, the group $\langle U, U'\rangle$ does not depend on which parabolic $P$ we have, as long as it is proper (in other words, as long as $U$ is non-trivial).  So in fact $G^\dagger_g = G^\dagger$ whenever $\con(g) \not=\triv$; in the present example, $G^\dagger = \mathrm{SL}_n(\bQ_p) = G$.
\end{ex}

\section{The nub of a flat group}

\subsection{Introduction}

Let $G$ be a \tdlc group and let $H$ be a group of automorphisms of $G$.  If $H$ is flat, the \defbold{nub}\index{nub!of a subgroup} $\nub(H)$ is the intersection of all compact open subgroups of $U$ that are tidy for $H$.  More generally we define the \defbold{lower nub}\index{nub!lower nub}\index{lnub@$\nubl_G(H)$} $\nubl(H)$ to be the closure of the group generated by the nubs of the cyclic subgroups of $H$.   Recall that the action of $H$ is said to be smooth if the tidy subgroups for the action form a base of neighbourhoods of the identity; in other words, a flat group $H$ is smooth if and only if $\nub(H)$ is trivial.

It is clear from Proposition~\ref{basic:tidybelow} that $\nubl(H) \le \nub(H)$ whenever $H$ is flat.  The following example illustrates that $\nubl(H)$ need not be the same as $\nub(H)$, even for uniscalar flat groups, and that the action of $H$ on $\nub(H)$ does not in general have the dynamical properties observed in \cite{WillisNub} in the cyclic case.

\begin{ex}\label{ex:badnub}
Let $V = \bF_p[[t]]$, regarded as a profinite vector space over $\bF_p$, and let $W$ be a non-trivial (possibly finite) closed subspace of $V$.  Let $H$ be the group of continuous $\bF_p$-linear maps from $V$ to $W$ under pointwise addition, and define an action $\rho$ of $H$ on $G = V \oplus W$ by setting $\rho(h)(v,w) = (v,w+h(v))$.  Then $\rho(H)$ is a subgroup of $\Aut(G)$ (necessarily flat, since $G$ is compact), and $\nub_G(\rho(H)) = W$, since for every subspace $V'$ of $V$ of finite codimension, there exists $h \in H$ such that $h(V') = W$.  In particular, the action of $H$ on $G$ does not have SIN.  However, $\rho(H)$ acts trivially on $W$ and the group $G \rtimes_{\rho} H$ is nilpotent, so there is no non-trivial subgroup $K$ of $G$ such that $\N_{\rho(H)}(K)$ acts ergodically on $K$, and in particular $\nub(\rho(h))$ is trivial for every $h \in H$.
\end{ex}

Both the nub and the lower nub are well-behaved under closures.

\begin{lem}\label{lem:nub:closure}
Let $G$ be a \tdlc group and let $H$ be a subgroup of $G$.  Suppose $\nubl(H)$ is compact.  Then $\nubl(H) = \nubl(\overline{H})$.  If $H$ is flat, then $\overline{H}$ is flat and $\nub(H) = \nub(\overline{H})$.
\end{lem}

\begin{proof}Suppose $\nubl(H)$ is compact and let $U$ be a compact open subgroup such that $\nubl(H) \le U$, so that $\nub(h) \le U$ for all $h \in H$.  Let $a \in \overline{H}$.  Then by  \cite[Theorem~1.5]{CRW-TitsCore}, for some open subgroup $V$ of $U$, then $\nub(a)$ is $V$-conjugate to $\nub(h)$ for all $h \in aV$.  Since $aV \cap H$ is non-empty, we conclude that $\nub(a) \le U$.  So in fact 
\[\nubl(H) \le U \Rightarrow \nubl(\overline{H}) \le U\]
for all compact open subgroups $U$, and the converse implication also clearly holds.  Hence $\nubl(H) = \nubl(\overline{H})$.

Now suppose $H$ is flat on $G$.  Then any tidy subgroup $U$ for $H$ is also tidy for $\overline{H}$ by Theorem~\ref{tidy:stable}, and conversely.  Hence $\nub(H) = \nub(\overline{H})$.
\end{proof}

The nub provides a simple criterion for when a flat group of automorphisms of $G$ has flat action on an open subgroup of $G$. 

\begin{lem}\label{openinvariant:flat}
Let $G$ be a \tdlc group, let $H$ be a flat group of automorphisms of $G$, and let $K$ be an open $H$-invariant subgroup of $G$.  Then $H$ is flat on $K$ if and only if $\nub_G(H) \le K$.  If $\nub_G(H) \le K$, then $\nub_G(H) = \nub_K(H)$.
\end{lem}

\begin{proof}
By Corollary~\ref{openinvariant:tidy}, we have $s_G(\alpha) = s_K(\alpha)$ for all $\alpha \in H$, so the tidy subgroups for $H$ on $K$ are precisely the tidy subgroups $U$ for $H$ on $G$ such that $U \le K$.  If $\nub_G(H) \not\le K$, then no such tidy subgroup can exist, so $H$ is not flat on $K$.  So suppose that $\nub_G(H) \le K$.  Let $V$ be a compact open subgroup of $K$ such that $\nub_G(H) \le V$.  Let $U$ be a tidy subgroup for $H$ on $G$.  Then $\nub_G(H) \le U \cap V$, and the collection of tidy subgroups for $H$ on $G$ is closed under finite intersections, so by a compactness argument on the open subgroups of $U$, there exists $W \le U \cap V$ such that $W$ is tidy for $H$ on $G$, and hence also for $H$ on $K$.  Thus $H$ is flat on $K$.  We see that given tidy subgroups $A$ and $B$ for $H$ on $G$ and on $K$ respectively, then $A \cap B$ is tidy for $H$ both on $G$ and on $K$, so $\nub_G(H) = \nub_K(H)$.
\end{proof}

The following is an immediate consequence of Corollary~\ref{cor:finiterank:subgroup} and Lemma~\ref{openinvariant:flat}.

\begin{cor}\label{cor:finiterank:residual}Let $G$ be a \tdlc group and let $H$ be a flat group of automorphisms of $G$ of finite rank.  Let $K$ be an open $H$-invariant subgroup of $G$.  Then $\nub_G(H) = \nub_K(H)$.\end{cor}

\subsection{Invariant uniscalar subgroups}

We now prove a result on the effect of $H$-invariant subgroups on tidy subgroups for $H$, which will allow us to establish some properties of $\N_G(\nub(H))$.  This result is a variation on \cite[Theorem~3.3]{WillisFlat}.

\begin{thm}\label{thm:compact_invariant:tidy}Let $G$ be a \tdlc group and let $H$ be a flat group of automorphisms of $G$.  Let $K$ be an $H$-invariant subgroup of $G$ such that $|K:\N_K(U)|$ is finite, and let $V = \bigcap_{k \in K}kUk\inv$.
\begin{enumerate}[(i)]
\item If $U$ is tidy below for $H$, then $V$ is tidy below for $H$.
\item If $U$ is tidy for $H$, then $V$ is tidy for $H$.  If in addition $K$ is compact, then $VK$ is also tidy for $H$.
\end{enumerate}
\end{thm}

\begin{proof}It suffices to consider elements $\alpha \in H$ individually, so fix $\alpha \in H$.  Suppose for the time being that $U$ is tidy for $H$.

Since $|K:\N_K(U)|$ is finite, $V$ is the intersection of finitely many conjugates of $U$, and hence $V$ is a compact open subgroup of $G$.  Note also that the set $\bigcup_{k \in K}kUk\inv$ is compact.  Since $K$ is $\alpha$-invariant, for all $n \in \bZ$ we have
\[
\alpha^n(V) = \bigcap_{k \in K}\alpha^n(kUk\inv) =  \bigcap_{k \in K}k\alpha^n(U)k\inv.
\]

Define
\[
U_+ := \bigcap_{n \ge 0}\alpha^n(U); \; U_- := \bigcap_{n \le 0}\alpha^n(U); \; U_{++} = \{g \in G \mid \forall n \gg 0 : \; \alpha^{-n}(g) \in U\}.
\]
The subgroups $V_+$, $V_-$ and $V_{++}$ are defined similarly, with $V$ in place of $U$.  We see that $V_+ = \bigcap_{k \in K}kU_+k\inv$ and $V_- = \bigcap_{k \in K}kU_-k\inv$.  Let $v \in V \cap U_{+}$ and let $k \in K$.  Then for all $n \le 0$, we have
\[
\alpha^n(kvk\inv) = \alpha^n(k)\alpha^n(v)\alpha^n(k\inv) \in \bigcup_{k \in K}kU_+k\inv.
\]
So the backward $\alpha$-orbit $\{\alpha^n(kvk\inv) \mid n \le 0\}$ is confined to the relatively compact set $\bigcup_{k \in K}kU_+k\inv$.  Moreover, we have $kvk\inv \in U$ since $v \in V$.  Thus by Lemma~\ref{tidy:orbit_dynamics}, $kvk\inv \in U_+$, so in fact $kvk\inv \in V \cap U_+$.  Since $k \in K$ was arbitrary, we conclude that $V \cap U_+$ is normalized by $K$.  In particular, it follows that 
\[
V \cap U_+ = V \cap \bigcap_{k \in K}kU_+k\inv = V_+.
\]
Similarly, $V \cap U_- = V_-$.

Note that $V_{++} \le U_{++}$, so
\[
V \cap V_{++} \le V \cap U_{++} \le V \cap U \cap U_{++} = V \cap U_+ = V_+,
\]
and hence $V$ is tidy below for $\alpha$ by Lemma~\ref{tidybelow:asymptotic}.

Our next aim is to show that $V = V_+V_-$.  Given $v \in V$, then \emph{a fortiori} $v \in U$, so there exist $u_+ \in U_+$ and $u_- \in U_-$ such that $v = u_+u_-$.  Then the sequence $(\alpha^{-n}(u_+))_{n \ge 0}$ is confined to the compact set $U_+$, so has an accumulation point $x$ say.  For all $r \in \bZ$, we see that $\alpha^{-r}(x)$ is an accumulation point of the sequence $(\alpha^{-n}(u_+))_{n \ge r}$, and hence also of $(\alpha^{-n}(u_+))_{n \ge 0}$ (which accounts for all but finitely many of the terms of $(\alpha^{-n}(u_+))_{n \ge r}$).  Thus $\alpha^r(x) \in U$ for all $r \in \bZ$ and in fact $\alpha^r(x) \in U_+ \cap U_-$ for all $r \in \bZ$.  Let $p \ge 0$ be such that $\alpha^{-p}(u_+) \in Vx$, that is, $\alpha^{-p}(u_+) = v'x$ for some $v' \in V$.  Then $v'$ is also an element of $U_+$, so in fact $v' \in V \cap U_+ = V_+$.  Now $v = \alpha^{p}(v')\alpha^{p}(x)u_-$; we see that $v \in \alpha^{p}(v')U_-$ (since $\alpha^{p}(x) \in U_-$) and also that $\alpha^{p}(v') = u_+(\alpha^{p}(x))\inv \in U_+$ (since $\alpha^{p}(x) \in U_+$), and hence we could have chosen $u_+$ and $u_-$ so that $u_+ = \alpha^{p}(v')$ for some $p \ge 0$ and $v' \in V_+$.  Let us assume that we have done so, and let $k \in K$.  Then
\[
\alpha^{-p}(ku_+k\inv) = \alpha^{-p}(k)v'\alpha^{-p}(k\inv);
\]
since $\alpha^{-p}(k) \in K$ and $V_+ = V \cap U_+$ is normalized by $K$, we see that $\alpha^{-p}(ku_+k\inv) \in V_+$, in other words $ku_+k\inv = u'$ for some $u' \in \alpha^p(V_+)$, so $kvk\inv = u'ku_-k\inv$.  At the same time, $kvk\inv \in U$, so $kvk\inv = w_+w_-$ for $w_+ \in U_+$ and $w_- \in U_-$.  Consider the element $w\inv_+ u'$.  Since $u'ku_-k\inv = w_+w_-$, we have
\[
w\inv_+ u' = w_-(ku_-k\inv)\inv \in \bigcup_{k \in K}(U_{-}kU_-k\inv);
\]
since $K$ is $\alpha$-invariant and $\alpha(U_-) \le U_-$, it follows that $(\alpha^n(w\inv_+ u'))_{n \ge 0}$ is confined to a compact set.  In addition
\[
\alpha^{-p}(w\inv_+ u') \in \alpha^{-p}(U_+)V_+ \subseteq U_+,
\]
so $\alpha^{-p}(w\inv_+ u')$ is an element of $U_+$ whose forward $\alpha$-orbit is bounded; it follows from Lemma~\ref{tidy:orbit_dynamics} that $\alpha^{-p}(w\inv_+ u') \in U_+ \cap U_-$, so $w\inv_+ u' \in U_+ \cap U_-$.  In particular, $u' = w_+(w\inv_+ u') \in U_+$.  Since $u' = ku_+k\inv$ and the choice of $k \in K$ was arbitrary, we conclude that $u_+ \in V \cap U_+$, so $u_+ \in V_+$.  Hence $u_- \in V$ as well (since $u_- = u\inv_+ v$), so $u_- \in V \cap U_-$, and hence $u_- \in V_-$.  Thus we have expressed an arbitrary $v \in V$ as a product of an element of $V_+$ and an element of $V_-$, so $V$ is tidy above for $\alpha$, completing the proof that $V$ is tidy for $\alpha$.
  
Now suppose $K$ is compact.  Since $K$ is $\alpha$-invariant, we see that
$$|\alpha(VK): \alpha(VK) \cap VK| \le |\alpha(V):\alpha(V) \cap V|.$$
Since $V$ is tidy for $H$, the minimum value for $|\alpha(W):\alpha(W) \cap W|$ (for $W$ a compact open subgroup of $G$) is already attained by $V$, so $VK$ is tidy for $H$ by Theorem~\ref{basic:tidy}.  This completes the proof of (ii).

Finally, let us relax the assumption that $U$ is tidy, and instead assume that $U$ is tidy below for $\alpha$.  Then there exists $U' = \bigcap^n_{i=0}\alpha^i(U)$ that is tidy above for $\alpha$, by Proposition~\ref{basic:tidyabove}; in fact $U'$ is also tidy below for $\alpha$ by Proposition~\ref{basic:tidybelow}.  Each of the groups $\alpha^i(U)$ has only finitely many $K$-conjugates, because $K$ is $\alpha$-invariant, so $|K:\N_K(U')|$ is finite.  We now apply part (ii) to conclude that $V' = \bigcap_{k \in K}kU'k\inv$ is tidy for $\alpha$.  Now $V \ge V'$ since $U \ge U'$, so $V$ is tidy below for $\alpha$ by Proposition~\ref{basic:tidybelow}, proving (i).
\end{proof}

\begin{cor}\label{cor:nub:normalizer}Let $G$ be a \tdlc group and let $H$ be a flat group of automorphisms of $G$.
\begin{enumerate}[(i)]
\item Let $L$ be the closure of the group generated by all $H$-invariant compact subgroups of $G$.  Then $\nub(H)$ is a normal subgroup of $L$.
\item Let $H'$ be a subgroup of $H$.  Then $\nub(H')$ is a normal subgroup of $\nub(H)$.
\item Suppose $H$ is uniscalar.  Then $\nub(H)$ is normalized by every compact open subgroup of $G$ that is tidy for $H$.  In particular, $\N_G(\nub(H))$ is open in $G$.
\end{enumerate}
\end{cor}

\begin{proof}Theorem~\ref{thm:compact_invariant:tidy} implies that whenever $K$ is an $H$-invariant compact subgroup of $G$ and $U$ is a tidy subgroup for $H$, then $U$ contains a $K$-invariant tidy subgroup for $H$.  Thus for any $H$-invariant compact subgroup $K$ of $G$, then $\nub(H)$ can be expressed as an intersection of $K$-invariant compact open subgroups, so $K$ normalizes $\nub(H)$.  Hence the normalizer of $\nub(H)$ contains a dense subgroup of $L$.  Since $\nub(H)$ is compact and $H$-invariant, in fact $\nub(H) \le L$ and $\N_G(\nub(H))$ is closed, so $\nub(H)$ is a normal subgroup of $L$, proving (i).

Given a subgroup $H'$ of $H$, then every compact open subgroup of $G$ that is tidy for $H$ is also tidy for $H'$; hence $\nub(H') \le \nub(H)$.  Since $\nub(H)$ is an $H'$-invariant compact subgroup of $G$, it follows from part (i) that $\nub(H')$ is normalized by $\nub(H)$, proving (ii).

Suppose $H$ is uniscalar.  Then a compact open subgroup $V$ of $G$ is tidy for $H$ if and only if $V$ is $H$-invariant; in this case we have $V \le \N_G(\nub(H))$ by part (i), proving (iii).\end{proof}

The following lemma and corollary can be used to enlarge the uniscalar part of a flat subgroup.

\begin{lem}\label{tidy:uniscalar_extension}Let $G$ be a \tdlc group and let $H$ and $K$ be flat subgroups of $G$ such that $H \le \N_G(K)$ and $K$ is uniscalar.  Let $U$ be a compact open subgroup of $G$.  Then $U$ is tidy for $HK$ if and only if $U$ is tidy for both $H$ and $K$.  Moreover, if $HK$ is flat, then $s_G(hk) = s_G(h)$ for all $h \in H$ and $k \in K$.\end{lem}

\begin{proof}If $U$ is tidy for $HK$, then clearly it is tidy for both $H$ and $K$.  Conversely, suppose that $U$ is tidy for $H$ and for $K$, and let $h \in H$ and $k \in K$.  Then $U$ is normalized by $K$ since $K$ is uniscalar, and for all $n \in \bZ$ we have $(hk)^n \in h^nK = Kh^n$, since $K$ is normalized by $H$.  Thus $h^nUh^{-n} = (hk)^nU(hk)^{-n}$ for all $n \in \bZ$.  By Theorem~\ref{scale:asymptotic}, it follows that $s_G(h) = s_G(hk)$, and hence by Theorem~\ref{basic:tidy}, $U$ is tidy for $hk$.\end{proof}

\begin{cor}Let $G$ be a \tdlc group and let $H$ be a flat subgroup of $G$.  Let $K$ be a compact $H$-invariant subgroup of $G$.  Then $HK$ is flat, and for all $h \in H$ and $k \in K$ we have $s_G(hk) = s_G(h)$.\end{cor}

\begin{proof}By Theorem~\ref{thm:compact_invariant:tidy}, there is a compact open subgroup $V$ of $G$ that is normalized by $K$ and is tidy for $H$.  The conclusion now follows from Lemma~\ref{tidy:uniscalar_extension}.\end{proof}

\subsection{Tidy subgroups in quotients}

The (lower) nub is not in general preserved under passing to quotients, because a compact open subgroup that is tidy (below) does not necessarily remain tidy below on passing to a quotient.  Indeed, \cite[Example~6.5]{WillisFurther} gives an example of the following situation: there is a \tdlc group $G$, an automorphism $\alpha$ and a closed $\alpha$-invariant subgroup $K$ of $G$, such that $\alpha$ has arbitrarily small tidy subgroups (so $\nub_G(\alpha)$ is trivial), and yet for every tidy subgroup $U$ for $\alpha$ on $G$, the group $UK/K$ is not tidy for $\alpha$ on $G/K$, because $UK/K$ fails to be tidy below (in other words, $UK/K$ does not contain $\nub_{G/K}(\alpha)$).

However, under certain conditions there is good control over the tidy subgroups, and hence the nub, when passing to a quotient.  In particular, it suffices for the scale to be preserved, as the following lemma shows.

\begin{lem}\label{lem:tidy:quotients}Let $G$ be a \tdlc group, let $\alpha$ be an automorphism of $G$ and let $U$ be a compact open subgroup of $G$.  Let $K$ be a closed subgroup of $G$, such that $U \le \N_G(K)$ and $\alpha(K)=K$, and write $N := \N_G(K)$.
\begin{enumerate}[(i)]
\item We have $s_{N/K}(\alpha) \le s_{N}(\alpha)$.  Indeed, $s_{N/K}(\alpha)$ divides $s_{N}(\alpha)$.
\item If $U$ is tidy above for $\alpha$, then $UK/K$ is tidy above for the action of $\alpha$ on $N/K$.
\item Suppose $s_{N/K}(\alpha) = s_{N}(\alpha)$ and that $U$ is tidy for $\alpha$ on $N$.  Then $UK/K$ is tidy for $\alpha$ on $N/K$.
\item Suppose $K$ is compact.  Then
\[
s_{G}(\alpha) = s_{N}(\alpha) = s_{N/K}(\alpha);
\]
moreover, given any $V/K \in N/K$ such that $V/K$ is tidy for $\alpha$ on $N/K$, then $V$ is tidy for the action of $\alpha$ on $G$.
\end{enumerate}
\end{lem}

\begin{proof}
(i)
See \cite[Proposition~4.7]{WillisFurther}.
 
(ii) Let $U_+ = \bigcap_{n \ge 0}\alpha^n(U)$ and $U_- = \bigcap_{n \le 0}\alpha^n(U)$.  Suppose $U$ is tidy above for $\alpha$.  Then $UK/K = (U_+K/K)(U_-K/K)$, and we have $\alpha(U_+K/K) \ge U_+K/K$ and $\alpha(U_-K/K) \le U_-K/K$.  So $UK/K$ is tidy above for $\alpha$.

(iii) Let $k = |\alpha(UK/K): \alpha(UK/K) \cap UK/K|$.  It is clear that
$$k \le |\alpha(U):\alpha(U) \cap U| = s_{N}(\alpha) = s_{N/K}(\alpha).$$
Since $s_{N/K}(\alpha)$ is already the minimum possible value for $k$, in fact $s_{N/K}(\alpha) = k$, and it follows that $UK/K$ is tidy for $\alpha$ by Theorem~\ref{basic:tidy}.

(iv) Let $K \le V \le N$ such that $V/K$ is tidy for $\alpha$ on $N/K$.  Then $V$ is a compact open subgroup of $G$.  Since $K$ is $\alpha$-invariant, we have
$$|\alpha^n(V/K): \alpha^n(V/K) \cap V/K| = |\alpha^n(V):\alpha^n(V) \cap V|,$$
for all $n \in \bZ$.  Hence $s_G(\alpha) = s_{N}(\alpha) = s_{N/K}(\alpha)$, by Theorem~\ref{scale:asymptotic}, so $V$ is tidy for $\alpha$ on $G$.
\end{proof}

We also note as a general point that intersections of compact subgroups of \tdlc groups are well-behaved under homomorphisms.

\begin{lem}\label{lem:intersection:quotient}Let $G$ be a \tdlc group, let $\mc{C}$ be a collection of compact subgroups of $G$ that is closed under finite intersections and let $\phi: G \rightarrow H$ be a continuous homomorphism to some \tdlc group $H$.  Then
\[
\bigcap_{C \in \mc{C}}(\phi(C)) = \phi(\bigcap_{C \in \mc{C}}C).
\]
\end{lem}

\begin{proof}For each $C \in \mc{C}$, we see that $\phi(C)$ is profinite, so $\bigcap_{C \in \mc{C}}(\phi(C))$ is profinite.  Thus $\bigcap_{C \in \mc{C}}(\phi(C))$ is expressible as an intersection of compact open subgroups of $H$.  Writing $D = \bigcap_{C \in \mc{C}}C$, it is clear that $\phi(D)$ is compact, hence closed, and that
\[
\bigcap_{C \in \mc{C}}(\phi(C)) \ge \phi(D);
\]
to show the reverse inequality, it suffices to show that for every compact open subgroup $U$ of $H$, if $\phi(D) \le U$, then $\bigcap_{C \in \mc{C}}(\phi(C)) \le U$.  So suppose $U$ is a compact open subgroup of $H$ such that $\phi(D) \le U$.  Consider the set $\mc{C}' = \{C \setminus \phi\inv(U) \mid C \in \mc{C}\}$.  Then $\mc{C}'$ consists of compact sets; it is closed under finite intersections; and the intersection $\bigcap_{E \in \mc{C}'}E$ is empty, since $D \le \phi\inv(U)$.  Thus $\emptyset \in \mc{C}'$, in other words, $C \le \phi\inv(U)$ for some $C \in \mc{C}$.  But then $\phi(C) \le U$, and hence $\bigcap_{C \in \mc{C}}(\phi(C)) \le U$ as desired.\end{proof}

Combining the previous two lemmas, we obtain the following stability properties of tidy subgroups and the nub under passing to a quotient.

\begin{prop}\label{prop:tidy:quotients}Let $G$ be a \tdlc group, let $H$ be a flat group of automorphisms of $G$ and let $K$ be a closed normal $H$-invariant subgroup of $G$.  Let $R = G/K$.

If $K$ is compact, then we have $s_{R}(\alpha) = s_{G}(\alpha)$ for all $\alpha \in H$, and also $\nub_{R}(H) = \nub_G(H)K/K$.

If $s_{R}(\alpha) = s_{G}(\alpha)$ for all $\alpha \in H$, then the following holds:
\begin{enumerate}[(i)]
\item The action of $H$ on $R$ is flat.  Indeed, whenever $U$ is tidy for $H$ on $G$, then $UK/K$ is tidy for $H$ on $R$.
\item We have
$$ \nub_{R}(H) \le \nub_G(H)K/K.$$
In particular, if $\nub_G(H) \le K$, then the action of $H$ on $R$ is smooth.
\end{enumerate}
\end{prop}

\begin{proof}
If $K$ is compact, then $s_{R}(\alpha) = s_{G}(\alpha)$ for all $\alpha \in H$ by Lemma~\ref{lem:tidy:quotients}(iv).

Now assume that $s_{R}(\alpha) = s_{G}(\alpha)$ for all $\alpha \in H$.  Then by Lemma~\ref{lem:tidy:quotients}(iii), if $U$ is a compact open subgroup of $G$ that is tidy for every $\alpha \in H$, then $UK/K$ is also tidy for every $\alpha \in H$, proving (i).  Thus $\nub_R(H) \le UK/K$ for every tidy subgroup $U$ for $H$ on $G$, so by Lemma~\ref{lem:intersection:quotient}, we conclude that $\nub_R(H) \le \nub_G(H)K/K$, proving (ii).  If $K$ is compact, then by Lemma~\ref{lem:tidy:quotients}(iv), every tidy subgroup $V/K$ of $H$ on $G/K$ is the image of a tidy subgroup $V$ of $H$ on $G$, so $\nub_G(H) \le V$ for all such $V$, and hence $\nub_G(H)K/K = \nub_R(H)$.
\end{proof}

The following special case will be used later.

\begin{cor}\label{cor:tidy:quotients:nub}
Let $G$ be a \tdlc group, let $H$ be a flat group of automorphisms of $G$, let $L$ be a uniscalar normal subgroup of $H$ and let $K = \nub_G(L)$.  Let $R = \N_G(K)/K$.  Then the following holds:
\begin{enumerate}[(i)]
\item We have $s_{R}(\alpha) = s_{G}(\alpha)$ for all $\alpha \in H$.
\item The action of $H$ on $R$ is flat, and moreover $\nub_{R}(H) = \nub_G(H)K/K$.
\end{enumerate}
\end{cor}

\begin{proof}
Let $U$ be a a compact open subgroup of $G$ that is tidy for $H$.  The normalizer of $K$ contains $U$ by Corollary~\ref{cor:nub:normalizer}; in particular, $\N_G(K)$ is open and $\N_G(K) \ge \nub_G(H)$.  It now follows by Lemma~\ref{openinvariant:flat} that $H$ is flat on $\N_G(K)$ and $\nub_G(H) = \nub_{\N_G(K)}(H)$.  We also have $s_{G}(\alpha) = s_{\N_G(K)}(\alpha)$ for all $\alpha \in H$ by Corollary~\ref{openinvariant:tidy}.  Thus we may assume $G = \N_G(K)$.  Note also that $K$ is compact.  All the conclusions now follow from Proposition~\ref{prop:tidy:quotients}.
\end{proof}

So we obtain an action of $H$ on a quotient $R = \N_G(K)/K$ of an open subgroup of $G$, such that the action of $H$ on $R$ retains the important properties of the action of $H$ on $G$, such as flatness and the scale function, but now the uniscalar part of $H$ (which is the same subgroup, whether we define it with respect to $G$ or with respect to $R$) acts smoothly.

\subsection{Flatness below}\label{sec:flatbelow:uniscalar}

Say $H$ is \defbold{flat below}\index{flat!flat below} if there exists a compact open subgroup $U$ of $G$ such that $U$ is tidy below for all $\alpha \in H$; equivalently (given Propositions \ref{basic:tidyabove} and \ref{basic:tidybelow}), for all $\alpha \in H$ there exists $V \le U$ depending on $\alpha$ such that $V$ is tidy for $\alpha$.  By Proposition~\ref{basic:tidybelow}, $H$ is flat below if and only if $\nubl(H)$ is compact, and in this case $\nubl(H)$ is the intersection of all compact open subgroups of $G$ that are tidy below for $H$.

We note that unlike the flat property, flatness below is inherited from cocompact normal subgroups, so in particular any virtually flat below group is flat below.  Flatness below is also stable on restricting the action to a closed subgroup.

\begin{lem}\label{flatbelow:cocompact}Let $G$ be a \tdlc group, let $H$ be a closed subgroup of $G$ and let $K$ be a closed cocompact normal subgroup of $H$.  Then $H$ is flat below if and only if $K$ is flat below.  Indeed, if $K$ is flat below then $\nubl(H) = \nubl(K)$.\end{lem}

\begin{proof}
If $H$ is flat below, then clearly $K \le H$ is as well, so we may assume $K$ is flat below.  Let $U$ be a compact open subgroup of $G$ such that $\nubl(K) \le U$, and let $h \in H$.  Then the sequence $(h^nK)_{n \ge 0}$ accumulates at the identity in the compact group $H/K$.  Using Proposition~\ref{basic:tidyabove}, let $V$ be a compact open subgroup of $G$ that is tidy above for $h$ and let $W$ be a compact open subgroup of $G$ that is tidy above for $k$, such that $W \le V \le U$.  Then there exist distinct integers $i,j \in \bZ$ such that $\{h^i,h^j\} \subseteq WK$, so $h^{i-j} = rks$ for some $r,s \in W$ and $k \in K$.  By \cite[Lemma~4.3 and Corollary~4.4]{CRW-TitsCore}, there exists $v \in V$ such that $\nub(r\inv h^{i-j}) = v\nub(h^{i-j})v\inv$, and there exists $w \in W$ such that $\nub(ks) = w\nub(k)w\inv$.  Thus $\nub(h^{i-j})$ is $V$-conjugate to $\nub(k) \le U$, so $\nub(h^{i-j}) \le U$.  Moreover $\nub(h^{i-j}) = \nub(h)$ by Lemma~\ref{basic:tidy:powers}, so $\nub(h) \le U$.  Since $h \in H$ was arbitrary, we conclude that $\nubl(H) \le U$, so $H$ is flat below.  Since $U$ could be any compact open subgroup of $G$ such that $\nubl(K) \le U$, we have shown $\nubl(H) \le \nubl(K)$; from the definition, it is clear that $\nubl(H) \ge \nubl(K)$, so in fact $\nubl(H) = \nubl(K)$.
\end{proof}

\begin{lem}\label{flatbelow:subgroup}Let $G$ be a \tdlc group, let $H$ be a subgroup of $G$ and let $K$ be a closed $H$-invariant subgroup of $G$.  Then $\nubl_K(H) \le \nubl_G(H)$, and if $K$ is open then $\nubl_K(H) = \nubl_G(H)$.  In particular, if $H$ is flat below on $G$, then it is flat below on $K$.\end{lem}

\begin{proof}Let $h \in H$.  By \cite[Lemma~4.1]{WillisFurther}, any tidy subgroup for $h$ on $G$ contains a tidy subgroup for $h$ on $K$.  Thus $\nub_K(h) \le \nub_G(h)$.  If $K$ is open, then $\nub_G(h) \cap K$ is an open $h$-invariant subgroup of $\nub_G(h)$, so by Theorem~\ref{basic:nub_characterizations} we have $\nub_G(h) \le K$.  Since $\nub_K(h)$ is the largest $h$-invariant subgroup of $K$ on which $h$ acts ergodically, we must have $\nub_K(h) = \nub_G(h)$.  The remaining conclusions are clear.\end{proof}

We saw in Example~\ref{ex:virtually_flat} that a finitely generated non-flat group $H$ can potentially be virtually flat, and hence also flat below.

Here are some basic observations about the role of $[H,H]$-invariant compact open subgroups in the tidy theory of $H$.

\begin{lem}\label{lem:invariablytidy}Let $G$ be a \tdlc group, let $H$ be a group of automorphisms of $G$ and let $L$ be a subgroup of $H$ such that $[H,H] \le L$.  Let $U$ be a compact open subgroup of $G$, and suppose that $U$ is $L$-invariant.
\begin{enumerate}[(i)]
\item Let $X$ be a subset of $H$.  Then $\bigcap_{\alpha \in X}\alpha(U)$ is $L$-invariant.
\item Let $\alpha \in H$ and let $\beta \in L\alpha$.  Then $U$ is tidy above for $\beta$ if and only if it is tidy above for $\alpha$, and $U$ is tidy below for $\beta$ if and only if it is tidy below for $\alpha$.
\item Let $\alpha, \beta \in H$ and suppose $U$ is tidy (above, below) for $\alpha$.  Then $\beta(U)$ is tidy (above, below) for $\alpha$.
\end{enumerate}
\end{lem}

\begin{proof}(i) We see that $L$ is normal in $H$, since $[H,H] \le L$.  Hence the set of $L$-invariant compact open subgroups of $G$ is preserved by $H$.  Consequently, if $U$ is $L$-invariant, then so is $\bigcap_{\alpha \in X}\alpha(U)$.

(ii) Observe that by part (i), the sets $U_+ := \bigcap_{n \ge 0}\beta^n(U)$ and $U_- := \bigcap_{n \le 0}\beta^n(U)$ do not depend on the choice of $\beta$ inside $L\alpha$.  Thus the validity of the equation $U = U_+U_-$ does not depend on the choice of $\beta$ inside $L\alpha$.  Similarly, the set 
\[
U_{++} := \bigcup_{i \ge 0}\bigcap_{n \ge i}\beta^n(U),
\]
does not depend on the choice of $\beta$ inside $L \alpha$.  Consequently, $U$ is tidy (above, below) for $\beta$ if and only if it is tidy (above, below) for $\alpha$.  

(iii) We see that $\beta(U)$ is tidy (above, below) for $\beta\alpha\beta\inv \in L\alpha$.  Hence $\beta(U)$ is tidy (above, below) for $\alpha$ by part (ii).
\end{proof}

\begin{defn}Let $G$ be a \tdlc group and let $H$ be a group of automorphisms of $G$.  A \defbold{tidying set}\index{tidy!tidying set} for $H$ is a subset $X$ of $H$ with the following property:

$(*)$ Let $U$ be a compact open subgroup of $G$, and suppose $\alpha(U)$ is tidy for $\beta$, for all $\alpha \in H$ and $\beta \in X$.  Then $U$ is tidy for $H$.\end{defn}

Given a finitely generated flat group $H$ of automorphisms, not all finite generating sets for $H$ are tidying sets (see \cite[Example~3.5]{WillisFlat}).  However, the following is effectively established in the proof of \cite[Theorem~5.5]{WillisFlat}.

\begin{thm}[See \cite{WillisFlat}, \S 5]\label{flat:invariablytidy}Let $G$ be a \tdlc group and let $H$ be a finitely generated group of automorphisms of $G$.  Then there exists a finite subset $X$ of $H$ that is a tidying set for $H$ on $G$.\end{thm}

We thus obtain a `tidying above procedure' for actions of finitely generated groups under certain circumstances.

\begin{lem}\label{flat:invariablytidy:tidying}Let $G$ be a \tdlc group, let $H$ be a finitely generated group of automorphisms of $G$, and let $X$ be a finite tidying set for $H$ on $G$.  Let $U$ be a compact open subgroup of $G$ such that $U$ is tidy below for $\alpha$, for all $\alpha \in X$, and such that $U$ has only finitely many conjugates under the action of $[H,H]$.  Then there is a finite intersection of $H$-conjugates of $U$ that is tidy for $H$ on $G$.
\end{lem}

\begin{proof}By Theorem~\ref{thm:compact_invariant:tidy}, the intersection of all $[H,H]$-conjugates of $U$ is tidy below for all $\alpha \in X$ (since $[H,H]$ is invariant under the action of each $\alpha \in X$).  So we may assume that $U$ is $[H,H]$-invariant.  By Lemma~\ref{lem:invariablytidy}, if $U$ is $[H,H]$-invariant and tidy below for some $\alpha \in H$, then $U$ is also tidy below for every $H$-conjugate of $\alpha$, and so any $H$-conjugate of $U$ is tidy below for $\alpha$.  It is clear from Proposition~\ref{basic:tidybelow} that the property of being tidy below for $\alpha$ is closed under finite intersections; hence any finite intersection of $H$-conjugates of $U$ is tidy below for $\alpha$.

Fix a compact open subgroup $U$ of $G$ such that $U$ is tidy below for $\alpha$, for all $\alpha \in X$, and such that $U$ is $[H,H]$-invariant.  Let $X = \{\alpha_1,\alpha_2,\dots,\alpha_m\}$.  We define a sequence of subgroups $U_{(i)}$ as follows: $U_{(0)}= U$, and thereafter $U_{(i)} = \bigcap_{|n| \le k_i}\alpha^n_i(U_{(i-1)})$, where $k_i$ is large enough so that $U_{(i)}$ is tidy above for $\alpha_i$ (such a $k_i$ exists by Lemma~\ref{basic:tidyabove}).  Then $U_{(i)}$ is tidy below for $\alpha_i$ by Lemma~\ref{lem:invariablytidy}(iii), since it is a finite intersection of $H$-conjugates of $U$, so in fact $U_{(i)}$ is tidy for $\alpha_i$.  Given $j > i$, then $U_{(j)}$ is a finite intersection of $H$-conjugates of $U_{(i)}$; each $H$-conjugate of $U_{(i)}$ is tidy for $\alpha_i$ by Lemma~\ref{lem:invariablytidy}(iii).  Hence $U_{(j)}$ is tidy for $\alpha_i$.  In particular, $U_{(m)}$ is tidy for every element of $X$.  By Lemma~\ref{lem:invariablytidy}(iii), for all $\gamma \in M$, the conjugate $\gamma(U_{(m)})$ is tidy for every element of $X$.  Hence $V = U_{(m)}$ is tidy for $H$.
\end{proof}

\subsection{A decomposition theorem for the nub}\label{sec:nubdecomp}

We can now state and prove a more precise version of Theorem~\ref{thmintro:flat:nubdecomp}.

\begin{thm}\label{thm:flat:nubdecomp}
Let $G$ be a \tdlc group and let $H$ be a flat group of automorphisms of $G$.  Suppose that $L$ is a uniscalar normal subgroup of $H$ such that $H/L$ is polycyclic.  Then there is a finite subset $X$ of $H$ such that the following holds:

Let $U$ be a compact open subgroup of $G$ such that $\nub(L) \le U$ and $\nub(\alpha) \le U$ for all $\alpha \in X$.  Then there is a finite subset $Y$ of $H$ such that $V = \bigcap_{\alpha \in Y}\alpha(U)$ is tidy for $H$.

In particular, writing $X = \{\alpha_1,\alpha_2,\dots,\alpha_n\}$, we have
$$ \nub(H) = \nub(L) \nub(\alpha_1) \nub(\alpha_2) \dots \nub(\alpha_n).$$
\end{thm}

\begin{proof}
Let us suppose for the moment that $H/L$ is abelian.  In light of Lemma~\ref{lem:tidy:quotients} and Corollary~\ref{cor:tidy:quotients:nub}, we may assume that $\nub(L)$ is trivial, in other words, $L$ is smooth.

Let $M$ be a finitely generated subgroup of $H$ such that $H = LM$ and let $X = \{\alpha_1,\dots,\alpha_n\}$ be a finite tidying set for $M$, as given by Theorem~\ref{flat:invariablytidy}.  Let $U$ be a compact open subgroup of $G$ such that $\nub(\alpha) \le U$ for all $\alpha \in X$.  We wish to show that there is a finite intersection $V$ of $H$-conjugates $U$ such that $V$ is tidy for $H$.

Let $R$ be the group of automorphisms of $G$ generated by $L$ together with conjugation by elements of $\nubl(H)$.  Let $U'' = U'\nubl(H)$, where $U'$ is the intersection of all $R$-conjugates of $U$.  Since $L$ is smooth and uniscalar and $\nubl(H)$ is compact and $L$-invariant, we see that $U'$ is an open subgroup of $U$.  Thus $U''$ is a compact open $L$-invariant subgroup of $G$ and we have $\nub(\alpha) \le U \cap U''$ for all $\alpha \in X$.  Now let $W = U \cap U''$.  We see that $W$ has only finitely many $L$-conjugates (since there are only finitely many subgroups between $U'$ and $U''$).  Given Theorem~\ref{thm:compact_invariant:tidy}, the intersection $W'$ of all $L$-conjugates of $W$ still contains $\nub(\alpha)$ for all $\alpha \in X$.  (The topological structure of $L$ is irrelevant to this application of Theorem~\ref{thm:compact_invariant:tidy}, so there is no harm in treating $L$ as a subgroup of $G$ by replacing $G$ with $G \rtimes L$.)  Given Lemma~\ref{flat:invariablytidy:tidying}, there is a finite intersection $V$ of $M$-conjugates of $W'$ that is tidy for $M$ on $G$.  Notice that $V$ inherits $L$-invariance from $W'$, since $M$ normalizes $L$.  We conclude by Lemma~\ref{tidy:uniscalar_extension} that $V$ is tidy for $H$.

In particular, $\nubl(H) \le V$, so $\nubl(H) \le U$, from which we see that $U'' = U'$, so $U''$ is actually a finite intersection of $L$-conjugates of $U$.  Following through the construction of $V$, we see that $V$ is itself a finite intersection of $H$-conjugates of $U$, as desired.

By Corollary~\ref{cor:nub:normalizer}, the groups $\nub(L)$ and $\nub(\alpha)$ for $\alpha \in H$ are normal subgroups of $\nub(H)$.  Thus the product
$$ K = \nub(L)\nub(\alpha_1)\nub(\alpha_2) \dots \nub(\alpha_n)$$
is compact subgroup of $G$ that does not depend on the ordering of the factors.  We have seen that every compact open subgroup of $G$ that contains $K$ also contains $\nub(H)$; this in fact ensures $\nub(H) \le K$, since in a \tdlc group, every compact subgroup is expressible as the intersection of the compact open subgroups that contain it.  Clearly also
\[
K \subseteq \nub(L)\nubl(H) \subseteq \nub(H),
\]
so in fact
\[
K = \nub(L)\nubl(H) = \nub(H).
\]

The proof of the theorem is now complete in the case that $H/L$ is abelian.

\

Now suppose $H/L$ is polycyclic, but not abelian.  We will prove the theorem by induction on the derived length of $H/L$.  Let $H^*/L$ be the last non-trivial term of the derived series for $H$.  From the abelian case, we see that there is a finite subset $X_1$ of $H^*$ such that the following holds:

Given a compact open subgroup $U$ of $G$ such that $\nub(L) \le U$ and $\nub(\alpha) \le U$ for all $\alpha \in X_1$, then there is a finite subset $Y_1$ of $H^*$ such that $V = \bigcap_{\alpha \in Y_1}\alpha(U)$ is tidy for $H^*$.  In particular, $\nub(H^*) \le U$.

By the inductive hypothesis, there is a finite subset $X_2$ of $H$ such that the following holds:

Given a compact open subgroup $U$ of $G$ such that $\nub(H^*) \le U$ and $\nub(\alpha) \le U$ for all $\alpha \in X_2$, then there is a finite subset $Y_2$ of $H$ such that $V = \bigcap_{\alpha \in Y_2}\alpha(U)$ is tidy for $H$.

Combining the two statements, we obtain the following:

Given a compact open subgroup $U$ of $G$ such that $\nub(L) \le U$ and $\nub(\alpha) \le U$ for all $\alpha \in X_1 \cup X_2$, then $\nub(H^*) \le U$, and moreover there is a finite subset $Y_2$ of $H$ such that $V = \bigcap_{\alpha \in Y_2}\alpha(U)$ is tidy for $H$.  In particular, $\nub(H) \le U$.

The desired decomposition of the nub now follows in the same manner as in the abelian case.
\end{proof}

We highlight the following special cases of Theorem~\ref{thm:flat:nubdecomp}.  (Here the product should be understood as a permutable product of subsets of $G$, which is not necessarily a direct product.)

\begin{cor}\label{flat:nubdecomp:special}Let $G$ be a \tdlc group and let $H$ be a flat group of automorphisms of $G$.
\begin{enumerate}[(i)]
\item Suppose that $H/H_{\us}$ is finitely generated.  Then
$$ \nub(H) = \nub(H_{\us})\prod_{\alpha \in X}\nub(\alpha)$$
for some finite subset $X$ of $H$.
\item Suppose that $H/[H,H]$ is finitely generated.  Then
$$ \nub(H) = \nub([H,H])\prod_{\alpha \in X}\nub(\alpha)$$
for some finite subset $X$ of $H$.
\item Suppose that $H$ has a smooth uniscalar normal subgroup $K$, such that $H/K$ is polycyclic.  Then
$$ \nub(H) = \prod_{\alpha \in X}\nub(\alpha)$$
for some finite subset $X$ of $H$.
\end{enumerate}
\end{cor}

\section{Residuals}

\subsection{Preliminaries}

\begin{defn}Let $H$ be a group acting by homeomorphisms on a Hausdorff topological space $X$.  A pair $(x,y) \in X \times X$ is \defbold{proximal}\index{proximal} for the action if there is a diagonal point $(z,z) \in X \times X$ that is in the closure of the $H$-orbit $\{(\alpha(x),\alpha(y)) \mid \alpha \in H\}$; the action is \defbold{distal}\index{distal} if there does not exist a proximal pair of distinct points.
\end{defn}

In this article we will be interested in the case when $X = G/K$, where $G$ is a \tdlc group on which $H$ acts by automorphisms and $K$ is a closed $H$-invariant subgroup of $G$.  In this situation there are two natural notions of distality: we say the action is \defbold{distal} if it is distal in the usual sense, and \defbold{distal at $1$}\index{distal!distal at $1$} if no $H$-orbit accumulates at the trivial coset, in other words, there is no non-trivial proximal pair of the form $(xK,K)$.  \emph{A priori}, distality at $1$ is a weaker notion than distality; the two notions coincide in the special case that $K$ is a normal subgroup of $G$, so that $X$ is a group, since in this case if $(xK,yK)$ is a proximal pair, then so is $(y\inv xK,K)$.

Observe that distal action (at $1$) is a residual property.

\begin{lem}\label{lem:distal_residual}Let $G$ be a Hausdorff topological group and let $H$ be a group of automorphisms of $G$.  Let $\mc{K}$ be a collection of closed $H$-invariant subgroups of $G$ such that $H$ is distal (at $1$) on $G/K$ for all $K \in \mc{K}$.  Then $H$ is distal (at $1$) on $G/L$, where $L = \bigcap_{K \in \mc{K}}K$.
\end{lem}

\begin{proof}
Suppose that $H$ is distal on $G/K$ for all $K \in \mc{K}$.  Let $(x,y) \in G \times G$ and suppose $(xL,yL)$ is a proximal pair for the action of $H$ on $G/L$.  Then there is a convergent net $(\alpha_i(x)L,\alpha_i(y)L)_{i \in I}$ in $G/L \times G/L$ with limit $zL$ say.  Given $K \in \mc{K}$, then $(\alpha_i(x)K,\alpha_i(y)K)_{i \in I}$ converges to $zK$, owing to the natural continuous quotient map $G/L \times G/L \rightarrow G/K \times G/K$.  Since the action of $H$ on $G/K$ is distal, we must have $xK = yK$, in other words $y\inv x \in K$.  Since $K \in \mc{K}$ was arbitrary, we in fact have $y\inv x \in L$.  Thus the action of $H$ on $G/L$ is distal.

The proof that distality at $1$ is a residual property is similar.
\end{proof}

Distality at $1$ is closed under extensions.  Distality is closed under extensions under certain circumstances.

\begin{lem}\label{lem:distal_extension}
Let $G$ be a Hausdorff topological group and let $H$ be a group of automorphisms of $G$.  Let $(K_\alpha)_{\alpha < \lambda}$ be a well-ordered descending chain of closed $H$-invariant subgroups of $G$, such that $K_\alpha = \bigcap_{\beta < \alpha}K_\beta$ whenever $\alpha$ is a non-zero limit ordinal.
\begin{enumerate}[(i)]
\item Suppose that $H$ is distal on $K_0/K_1$, and that for all $\alpha$ such that $1 < \alpha+ 1 < \lambda$, we have $K_{\alpha} \le \N_G(K_{\alpha+1})$ and the action of $H$ on $K_{\alpha}/K_{\alpha+1}$ is distal at $1$.  Then $H$ is distal on $K_0/K_\alpha$ for all $\alpha < \lambda$.
\item Suppose that $H$ is distal at $1$ on the coset space $K_{\alpha}/K_{\alpha+1}$ for all $\alpha$ such that $\alpha+1 < \lambda$.  Then $H$ is distal at $1$ on $K_0/K_\alpha$ for all $\alpha < \lambda$.
\end{enumerate}
\end{lem}

\begin{proof}
(i)
Suppose there is some $\alpha < \lambda$ such that $H$ is not distal on $K_0/K_\alpha$; let $\alpha$ be the least ordinal for which this occurs.  Clearly $\alpha > 1$.

If $\alpha$ is a limit ordinal, then $H$ is distal on $K_0/K_\beta$ for all $\beta < \alpha$, so $H$ is distal on $K_0/K_\alpha$ by Lemma~\ref{lem:distal_residual}.

Suppose $\alpha = \beta+1$ for some ordinal $\beta$ and let $(xK_\alpha,yK_\alpha)$ be a proximal pair for the action of $H$ on $K_0/K_\alpha$.  Then $(xK_\beta,yK_\beta)$ is a proximal pair, so $xK_{\beta} = yK_{\beta}$; say $y = xk$ for some $k \in K_{\beta}$.  Let $(h_i)_{i \in I}$ be a net in $H$ such that $((h_i(x)K_\alpha,h_i(y)K_\alpha))_{i \in I}$ converges to $(zK_{\alpha},zK_{\alpha})$ for some $z \in K_0$.  Given an open neighbourhood $O$ of $1$ in $G$, we have
\[
h_i(x), h_i(y) \in zOK_{\alpha}
\]
for all $i$ sufficiently large.  Now $h_i(y) = h_i(x)h_i(k)$, so for $i$ sufficiently large we have
\[
h_i(x) \in zOK_{\alpha} \cap zOK_{\alpha}h_i(k\inv) = z(OK_{\alpha} \cap Oh_i(k\inv)K_{\alpha}),
\]
using the fact that $h_i(k\inv) \in K_{\beta}$, so $h_i(k\inv)$ normalizes $K_{\alpha}$ for all $i$.  In particular, for all $i$ sufficiently large the set $OK_{\alpha} \cap Oh_i(k\inv)K_{\alpha}$ is non-empty, so we have $h_i(k\inv) \in O\inv OK_{\alpha}$.  By the continuity of the group operations in $G$, the sets $O\inv O$ form a base of identity neighbourhoods as $O$ ranges over the identity neighbourhoods of $G$.  Thus $h_i(k\inv)K_{\alpha}$ converges to $K_{\alpha}$ in $K_{\beta}/K_{\alpha}$.  Since $H$ is distal at $1$ on $K_{\beta}/K_{\alpha}$ by hypothesis, it follows that $k\inv \in K_{\alpha}$, and consequently $xK_\alpha = yK_{\alpha}$.

In either case, we obtain a contradiction to the assumption that $H$ was not distal on $K_0/K_\alpha$ for some $\alpha < \lambda$.

(ii)
Again we suppose that $\alpha$ is minimal such that $K_0/K_{\alpha}$ is a counterexample.  As before, it is clear that $\alpha = \beta+1$ for some ordinal $\beta > 0$.

Let $(xK_\alpha,K_\alpha)$ be a proximal pair for the action of $H$ on $K_0/K_{\alpha}$.  Then $(xK_\beta,K_\beta)$ is a proximal pair, so $x \in K_{\beta}$.  But then $(xK_\alpha,K_\alpha)$ is a proximal pair for the action of $H$ on $K_{\beta}/K_{\alpha}$, so $x \in K_\alpha$, so in fact $H$ is distal at $1$ on $K_0/K_{\alpha}$, a contradiction.
\end{proof}

\begin{defn}
Let $G$ be a Hausdorff topological group and let $H$ be a group of automorphisms of $G$.

The \defbold{distal residual}\index{distal!distal residual} $\Dist_G(H)$ is the intersection of all $H$-invariant closed subgroups $K$ of $G$ such that $H$ acts distally on $G/K$, and $\Dist(G) := \Dist_G(\Inn(G))$.  Equivalently in light of Lemma~\ref{lem:distal_residual}, the subgroup $D = \Dist_G(H)$ is the smallest $H$-invariant closed subgroup of $G$ such that the conjugation action of $H$ on the coset space $G/D$ is distal.  We can also define the distal residual of a coset space: given an $H$-invariant subgroup $K$ of $G$, define $\Dist_{G/K}(H)$\index{Dist@$\Dist_{G/K}(H)$} to be the smallest closed $H$-invariant subgroup $D$ of $G$ such that $K \le D$ and the conjugation action of $H$ on $G/D$ is distal.  Similarly, we define the $1$-distal residual $\Dist^*_{G/K}(H)$\index{Dist@$\Dist^*_{G/K}(H)$} to be the smallest closed $H$-invariant subgroup $E$ of $G$ such that $K \le E$ and $H$ acts distally at $1$ on $G/E$.

The \defbold{discrete residual}\index{discrete residual} $\Res(G)$ of $G$ is the intersection of all open normal subgroups of $G$, and $G$ is \defbold{residually discrete}\index{residually discrete} if $\Res(G)=\triv$.  More generally, given an $H$-invariant subgroup $K \le G$, we define $\Res_{G/K}(H)$, the \defbold{discrete residual of $H$ on $G/K$}\index{Res@$\Res_{G/K}(H)$}, to be the intersection of all open $H$-invariant subgroups of $G$ that contain $K$, and the action of $H$ on $G/K$ is \defbold{residually discrete} if $\Res_{G/K}(H)=K$.

Define $\Res^{\alpha}_{G/K}(H)$ as $\alpha$ ranges over the ordinals as follows: $\Res^0_G(H) := G$; $\Res^{\alpha+1}_{G/K}(H) := \Res_{\Res^\alpha_{G/K}(H)/K}(H)$; and if $\alpha$ is a non-zero limit ordinal, $\Res^{\alpha}_{G/K}(H) := \bigcap_{\beta < \alpha}\Res^{\beta}_{G/K}(H)$.  Then the groups $\Res^{\alpha}_{G/K}(H)$ form a descending chain of closed $H$-invariant subgroups of $G$, eventually terminating at some group $\Res^\infty_{G/K}(H)$\index{Res@$\Res^\infty_{G/K}(H)$} that has no proper $H$-invariant open subgroups.\end{defn}

Indeed, $\Res^\infty_G(H)$ can be characterized as the unique largest closed $H$-invariant subgroup $K$ of $G$ such that $K$ has no proper open $H$-invariant subgroups.

\begin{lem}\label{lem:resinfty_properopen}Let $G$ be a Hausdorff topological group and let $H$ be a group of automorphisms of $G$.  Let $K$ be a closed $H$-invariant subgroup of $G$, and suppose that $K$ has no proper open $H$-invariant subgroups.  Then $K \le \Res^\infty_G(H)$.\end{lem}

\begin{proof}
It is enough to show that $K \le \Res^\alpha_G(H)$ for every ordinal $\alpha$.  We proceed by induction on $\alpha$.

The case $\alpha = 0$ is immediate.

If $\alpha$ is a non-zero limit ordinal, then $K \le \Res^\beta_G(H)$ for all $\beta < \alpha$ by the inductive hypothesis, so $K \le \Res^\alpha_G(H)$.

If $\alpha = \beta+1$ for some ordinal $\beta$, then $K \le \Res^\beta_G(H)$ by the inductive hypothesis.  Given an open $H$-invariant subgroup $U$ of $L:= \Res^\beta_G(H)$, then $K \cap U$ is an open $H$-invariant subgroup of $K$, so $K \cap U = K$, that is, $K \le U$.  Thus $K \le \Res_L(H) = \Res^\alpha_G(H)$.
\end{proof}

Example~\ref{ex:badnub} shows that $\Res^\infty_G(H)$ can be a proper subgroup of $\Res_G(H)$.

We also see that a residually discrete action is distal (and hence distal at $1$).  In particular, it follows that $\Dist_G(H) \le \Res_G(H)$ and $\Dist^*_G(H) \le \Res^{\infty}_G(H)$.

\begin{lem}\label{lem:discres_distal}
Let $G$ be a Hausdorff topological group, let $H$ be a group of automorphisms of $G$ and let $K$ be a closed $H$-invariant subgroup of $G$.  Suppose that $K$ is an intersection of open $H$-invariant subgroups of $G$.  Then $H$ acts distally on $G/K$.
\end{lem}

\begin{proof}
Let $\mathcal{K}$ be the set of all open $H$-invariant subgroups of $G$ that contain $K$ and let $O \in \mathcal{K}$.  Then clearly the action of $H$ on $G/O$ is distal, since $G/O$ is a discrete space on which $H$ acts by permutations.  The conclusion then follows by Lemma~\ref{lem:distal_residual}.
\end{proof}

Let us now assume that $G$ is a \tdlc group.  Since $\Dist^*_G(H)$ is a closed subgroup and contraction groups are well-behaved with respect to coset spaces by Theorem~\ref{thm:bw:relative_contraction}, we have $\overline{G^\dagger_H} \le \Dist^*_G(H)$.  So we have inclusions
\[
\overline{G^\dagger_H} \subseteq \Dist^*_G(H) \subseteq A_G(H) \subseteq \Res_G(H)
\]

of closed $H$-invariant subgroups of $G$, as stated in the introduction, where $A_G(H)$ is either $\Dist_G(H)$ or $\Res^\infty_G(H)$.

In the case of a group of automorphisms of a \emph{profinite} group, the existence or non-existence of proper invariant open subgroups is closely related to whether or not the action is ergodic.  The next series of results are based on results by Jaworski (\cite{JawDistal}) for actions on compact groups in general.

\begin{prop}[See \cite{JawDistal} Proposition 2.1]\label{ergodic_res}Let $G$ be a profinite group and let $H$ be a group of automorphisms of $G$.  Then the following are equivalent:
\begin{enumerate}[(i)]
\item $H$ does not act ergodically.
\item There exists a proper open normal $H$-invariant subgroup $N$ of $G$.
\item There exists a compact $H$-invariant identity neighbourhood $U$ in $G$ such that $U^2 \not= G$.
\end{enumerate}
\end{prop}

\begin{cor}\label{cor:distal_compact_extension}
Let $G$ be a \tdlc group, let $H$ be a group of automorphisms of $G$ and let $R$ be a compact $H$-invariant subgroup of $G$.  Suppose that $H$ acts distally on $G/R$.  Then $H$ acts distally on $G/\Res^\infty_R(H)$.
\end{cor}

\begin{proof}
Construct a descending sequence $(R_\alpha)_{\alpha < \lambda}$ of closed $H$-invariant subgroups of $R$, as far as it is possible to do so, in the following manner:

Set $R_0 = R$.

If $R_{\alpha}$ has been defined, let $R_{\alpha+1}$ be a proper open normal $H$-invariant subgroup of $R_{\alpha}$ if one exists; otherwise terminate.

If $\alpha$ is a limit ordinal and $R_{\beta}$ has been defined for all $\beta < \alpha$, set $R_{\alpha} = \bigcap_{\beta < \alpha}R_{\beta}$.

This sequence eventually terminates at some subgroup $T = R_{\alpha}$.  Given the construction of $T$ and Lemma~\ref{lem:distal_extension}, we see that $H$ acts distally on $G/T$.  It is also clear from the construction that $T \ge \Res^\infty_R(H)$.  The rule for terminating the series ensures that $T$ has no proper open normal $H$-invariant subgroups.  But then by Proposition~\ref{ergodic_res}, $H$ acts ergodically on $T$, so $T$ has no proper open $H$-invariant subgroups.  Hence $T = \Res^\infty_R(H)$ by Lemma~\ref{lem:resinfty_properopen}.
\end{proof}

\begin{prop}\label{prop:resinfty_ergodic}Let $G$ be a \tdlc group and let $H$ be a group of automorphisms of $G$.  Suppose that $\Res^\infty_G(H)$ is compact.  Then the following holds.
\begin{enumerate}[(i)]
\item $\Res^\infty_G(H)$ is the largest closed subgroup of $G$ on which $H$ acts ergodically.
\item $\Res^\infty_G(H)$ is normalized by every compact $H$-invariant subgroup of $G$.
\item If $\Res^\infty_G(H)$ is metrizable, then $\Res^\infty_G(H) = \Dist^*_G(H)$.
\end{enumerate}
\end{prop}

\begin{proof}
Suppose that $K$ is a closed $H$-invariant subgroup on which $H$ acts ergodically.  Then $K$ cannot have a proper open $H$-invariant subgroup, so $\Res_K(H) = K$, and hence $K \le \Res^\infty_G(H)$; in particular, $K$ is compact.  Observe that if $L$ is an $H$-invariant subgroup of $G$ and $\Res^\infty_G(H) \le L$, then $\Res^\infty_G(H) = \Res^\infty_L(H)$.  Part (i) now follows from Proposition~\ref{ergodic_res} and part (ii) follows from \cite[Theorem~2.6]{JawDistal}.

By Lemma~\ref{lem:discres_distal} and Lemma~\ref{lem:distal_extension}, we have $\Dist^*_G(H) \le \Res^\infty_G(H)$.  Suppose $R = \Res^\infty_G(H)$ is metrizable, in other words, $R$ has only countably many open subgroups, and let $T$ be the set of all $t \in R$ such that $(t,1)$ is a proximal pair for the action of $H$ on $R$.  Then $T$ can be expressed as the intersection of countably many open $H$-invariant subsets of $R$.  Since $H$ acts ergodically on $R$, every open $H$-invariant subset of $R$ is dense, so by the Baire Category Theorem, $T$ is dense in $R$.  Since the action of $H$ on $\Res^\infty_G(H)/\Dist^*_G(H)$ is distal at $1$, we see that $T \subseteq \Dist^*_G(H)$, and hence $\Dist^*_G(H) = \Res^\infty_G(H)$.
\end{proof}

\subsection{A sufficient condition for a non-distal action}\label{sec:nondistal}

In this subsection we shall obtain a sufficient condition for a compactly generated subgroup (not necessarily flat) to act non-distally.  The argument is to a large extent a combination of those used in the proofs of \cite[Corollary~4.1]{CM} and \cite[Theorem~3.1]{RajaShah}.

First of all, we prove a version of \cite[Theorem~3.1]{RajaShah}, using a similar argument.

\begin{prop}\label{prop:distal:cocompact}
Let $G$ be a \tdlc group and let $\Gamma$ be a group of automorphisms of $G$.  Let $L \le K$ be $\Gamma$-invariant closed subgroups of $G$ such that $\bigcap_{k \in K}kLk\inv$ is cocompact in $K$.  If $\Gamma$ acts distally on $G/L$, then $\Gamma$ acts distally on both $K/L$ and $G/K$.
\end{prop}

\begin{proof}
Suppose $\Gamma$ acts distally on $G/L$.  Then it is clear that $\Gamma$ acts distally on $X:= K/L$.  Let $E$ be the closure of $\Gamma$ in $X^X$, the set of all functions from $X$ to itself.  Since $\Gamma$ acts distally on the compact space $X$, $E$ is a compact group of homeomorphisms of $X$ (see \cite{Ellis} Theorem 1).

Let $(U_i)_{i \in I}$ be a descending net of open subsets of $G$ forming a base of neighbourhoods of identity in $G$.  Suppose $(aK,aK)$ is in the closure of $\{(\gamma(x)K,\gamma(y)K) \mid \gamma \in \Gamma\}$ (as a subset of $G/K \times G/K$) for some $a,x,y \in G$.  Then for each $U_i$, there exists $\gamma_i \in \Gamma$ such that
\[
\gamma_i(x)K,\gamma_i(y)K \in aU_iK.
\]
In other words, $\gamma_i(x) = au_ik_i$ and $\gamma_i(y) = au'_ik'_i$ for some $u_i,u'_i \in U_i$ and $k_i,k'_i \in K$.  The choice of the net $(U_i)_{i \in I}$ ensures that $u_i \rightarrow 1$ and $u'_i \rightarrow 1$.  Let $M = \bigcap_{k \in K}kLk\inv$.  Since $K/M$ is compact, by passing to a subnet we may assume there are $k,k' \in K$ such that $k_iM \rightarrow kM$ and also $k'_iM \rightarrow k'M$.  Let $\gamma \in E$ be a limit point of $(\gamma_i)_{i \in I}$ in $X^X$; by passing to a subnet we may assume $\gamma_i \rightarrow \gamma$.  Let $k_1L = \gamma\inv(k\inv L)$ and $k'_1L = \gamma\inv((k')\inv L)$.  Then $\gamma_i(k_1L)$ converges to $\gamma(k_1L) = k\inv L$ in $K/L$.  Thus we see that given any open neighbourhoods $O_G$ and $O_K$ of the identity in $G$ and $K$ respectively, then for $i$ sufficiently large we have
\[
\gamma_i(xk_1L) = au_ik_i\gamma_i(k_1L) \in aO_G(O_KkM)(k\inv O_K L)/L = aO_GO^2_K L/L,
\]
so $\gamma_i(xk_1L)$ converges to $aL$.  Similarly, $\gamma_i(yk'_1L)$ converges to $aL$.  Since $\Gamma$ is distal on $G/L$, it follows that $xk_1L = yk'_1L$, so $xK = yK$.  Thus $\Gamma$ is distal on $G/K$, completing the proof.
\end{proof}

We also note that actions of compact subgroups always have SIN.

\begin{lem}\label{lem:compact_anisotropic}
Let $G$ be a \tdlc group, let $H$ be a closed subgroup of $G$, let $K$ be an $H$-invariant closed subgroup of $G$ and let $M$ be a compact open subgroup of $H$.
\begin{enumerate}[(i)]
\item The action of $M$ on $K$ has SIN, in other words the open subgroups of $K$ normalized by $M$ form a base of identity neighbourhoods in $K$.
\item If $K$ is compact, then every open subgroup of $K$ has open normalizer in $H$.
\end{enumerate}
\end{lem}

\begin{proof}
(i)
Since $M$ is compact, there exists a compact open subgroup $U$ of $G$ such that $M \le U$.  In particular, $U \cap K$ is a compact open subgroup of $K$ that is normalized by $M$.  To show that the action of $M$ on $K$ has SIN, it suffices to show that the action of $M$ on $U \cap K$ has SIN.  So we may assume that $K \le U$, so $K$ is compact.

Let $\mc{U}$ be a base of identity neighbourhoods in $U$; since $U$ is profinite, we can choose $\mc{U}$ to consist of open normal subgroups of $U$, which are then in particular $M$-invariant.  Then the set $\mc{U}_K:= \{V \cap K \mid V \in \mc{U}\}$ is a collection of $M$-invariant open normal subgroups of $\mc{U}_K$ with trivial intersection.  A standard compactness argument now shows that $\mc{U}_K$ is in fact a base of identity neighbourhoods in $K$.  Hence the action of $M$ on $K$ has SIN.

(ii)
Let $A$ be an open subgroup of $K$.  Since the action of $M$ on $K$ has SIN, there is a compact open $M$-invariant subgroup $B$ of $K$ such that $B \le A$.  Since $K$ is compact, $B$ has finite index in $K$, so there are only finitely many subgroups between $B$ and $K$, and hence $|M:\N_M(A)| < \infty$.  It follows that $A$ has open normalizer in $M$, and hence in $H$.
\end{proof}

We now obtain a non-distal action under certain circumstances.

\begin{thm}\label{thm:trivial_core:nondistal}
Let $G$ be a \tdlc group and let $H$ be a closed subgroup of $G$.  Suppose that there is a compact subgroup $K$ of $G$ such that the following conditions hold:
\begin{enumerate}[(i)]
\item $K$ is commensurated by $H$.
\item $\N_H(K)$ is open in $H$.
\item Letting $L = \bigcap_{h \in H}hKh\inv$, then $L$ is not open in $K$;
\item There is a closed normal subgroup $N$ of $H$ such that $N \le \N_H(K)$ and $H/N$ is compactly generated.
\end{enumerate}
Then there is a non-trivial coset $kL$ of $L$ such that the $H$-orbit of $kL$ accumulates at the trivial coset in $G/L$.  In particular, $H$ does not act distally on $G/L$, and hence $H$ does not act distally on $G$.
\end{thm}

\begin{proof}
Let $X$ be a generating set for $H$ such that $N \subseteq X$, $X$ is a union of left cosets of $N$, $x \in X \Leftrightarrow x\inv \in X$ and $X/N$ is compact.  Write $X^n$ for the set of elements of $H$ of the form $x_1x_2 \dots x_n$ for $x_i \in X$; note that $X^n/N$ is compact for all natural numbers $n$ and that $X^n \subseteq X^{n'}$ whenever $n' \ge n$.  Set
\[
K(n) = \bigcap_{y \in X^n}y\inv Ky.
\]
Since $N \le \N_H(K)$ and $X^n/N$ is compact, we see that $X^n$ is contained in the union of finitely many (right) cosets of the open subgroup $\N_H(K)$ of $H$, and hence $K(n)$ is the intersection of finitely many conjugates of $K$.  Since $K$ is commensurated by $H$, in fact $K(n)$ is an open subgroup of $K$ for all $n$.  By Lemma~\ref{lem:compact_anisotropic}, the normalizer of $K(n)$ in $H$ is also open; by construction $K(n)$ is also normalized by $N$.  Thus the same argument as for $K$ itself shows that given $m,n \in \bN$, there are only finitely many distinct subgroups of the form $y\inv K(n) y$ for $y \in X^m$.

We now have a descending chain $K = K(0) \ge K(1) \ge \dots$ of open subgroups of $K$ with intersection $L$.  A standard compactness argument then shows that the set $\{K(n)/L \mid n \ge 0\}$  is a base of neighbourhoods of the trivial coset in the coset space $K/L$.

For $n \ge 0$, define
\[
P(n) := \bigcup_{y \in X^{n}}yK(n)y\inv.
\]
We observe that $P(n)$ is a subset of $K$, for all $n$; moreover, $P(n)$ is compact, since it is a union of finitely many closed subsets of $K(n)$.  Given $m \le n$ and $g \in P(n)$, there exists $y \in X^n$ such that $y\inv g y \in K(n) = \bigcap_{z \in X^{n-m}}z\inv K(m) z$, so in particular, there exists $w = yz\inv \in X^m$, with $z \in X^{n-m}$, such that $w\inv g w \in K(m)$.  Thus $g \in P(m)$, showing that $P(m) \supseteq P(n)$.  So we have a descending chain
\[
P(0) \supseteq P(1) \supseteq \dots
\]
of compact subsets of $K$.

Suppose that $P(m) \subseteq K(1)$ for some $m$, in other words, for all $y \in X^{m}$ we have $yK(m)y\inv \subseteq K(1)$.  Then $K(m) \subseteq y\inv K(1) y$ for all $y \in X^m$, so 
\[
K(m) \subseteq \bigcap_{y \in X^{m}}y\inv K(1) y =  K(m+1).
\]
By the same argument, $K(n) = K(n+1)$ for all $n \ge m$, so in fact $K(m) = \bigcap_{n \ge m}K(n) = L$.  This is absurd as $K(m)$ is an open subgroup $K$, whereas $L$ is not open in $K$.  Hence for all $n$, $P(n) \cap (K \setminus K(1))$ is non-empty.  Now $K \setminus K(1)$ is compact, since $K$ is compact and $K(1)$ is an open subgroup of $K$.  So by compactness, there exists
\[
x \in \bigcap_{n \ge 0}P(n) \cap (K \setminus K(1)).
\]
Since $x \not\in K(1)$, we see that $xL$ is a non-trivial element of $K/L$.  On the other hand, since $x \in P(n)$ for all $n \ge 0$, we see that $x$ is $H$-conjugate to an element of $K(n)$ for all $m \ge 0$.  Since $\{K(n)/L \mid n \ge 0\}$ is a base of neighbourhoods of the trivial coset in $K/L$, it follows that the $H$-orbit of $xL$ accumulates at the trivial coset $L$, so $(xL,L)$ is a non-trivial proximal pair for the action of $H$ on $G/L$.  In particular, the action of $H$ on $G/L$ is not distal.  It follows from Proposition~\ref{prop:distal:cocompact} that the action of $H$ on $G$ is not distal.
\end{proof}

The hypotheses of Theorem~\ref{thm:trivial_core:nondistal} are general enough that the (relative) Tits core $G^\dagger_H$ can be trivial, even if $H = G$ and $L = \triv$: see Example~\ref{ex:nondistal} below.  On the other hand, the sufficient condition for non-distal action does provide several equivalent characterizations of when the action of $H$ is flat and uniscalar.  We can now state and prove a more general version of Theorem~\ref{thmintro:distal}.

\begin{thm}\label{thm:distal}
Let $G$ be a \tdlc group, let $H$ be a closed subgroup of $G$, acting by conjugation, and let $K$ be a closed $H$-invariant subgroup of $G$.  Suppose that there exists a closed normal subgroup $N$ of $H$ (possibly trivial) such that $N$ has SIN action on $K$ and $H/N$ is compactly generated.

Then the following are equivalent:
\begin{enumerate}[(i)]
\item $\Dist_K(H)$ is compact;
\item $\Res_K(H)$ is compact;
\item $H$ normalizes a compact open subgroup of $K$.
\end{enumerate}
Moreover, if any of the above conditions is satisfied, then
\[
\nub_K(H) = \Res_K(H) =  \Res^\infty_K(H) = \Dist_K(H)
\]
and $H$ acts ergodically on $\nub_K(H)$, with $\nub_K(H) = \Dist^*_K(H)$ in the case that $\nub_K(H)$ is metrizable.
\end{thm}

\begin{proof}
Recall that $\Res_K(H) \ge \Dist_K(H) \ge \Dist^*_K(H)$.  Fix a compact $H$-invariant subgroup $R$ of $K$.  Let us consider whether or not the following statement is true:

$(**)$ For every compact open subgroup $U$ of $K$ such that $U \ge R$, then $V = \bigcap_{h \in H}hUh\inv$ is $H$-invariant and open in $K$.

If $(**)$ is true, then $H$ is uniscalar and $V$ is tidy for $H$, so $H$ is flat and $V \ge \nub_K(H)$.  In particular, since $U/R$ can be made arbitrarily small, we see that $R \ge \nub_K(H)$ and also that $R \ge \Res_K(H)$.

Suppose instead that $(**)$ is false, with the open subgroup $U$ of $K$ as a counterexample.  Let $U'$ be the intersection of all $N$-conjugates of $U$; since $N$ has SIN action, $U'$ is open in $U$.  Fix a compact open subgroup $M$ of $H$.  Then the action of $M$ on $K$ has SIN by Lemma~\ref{lem:compact_anisotropic}, so the intersection $W$ of all $M$-conjugates of $U'$ is open in $K$; moreover, $W$ is $N$-invariant, using the fact that $M$ normalizes $N$.  Let $L$ be the intersection of all $H$-conjugates of $W$, and note that $L \ge R$.  Then the following are all easily verified:

$W$ is commensurated by $\Aut(K)$, so in particular by $H$; $\N_H(V)$ contains $MN$ and so is open in $H$; $L$ is a closed but not open subgroup of $W$; and $H/N$ is compactly generated.

It now follows by Theorem~\ref{thm:trivial_core:nondistal} that there is a non-trivial $H$-orbit on $K/L$ that accumulates at the trivial coset.  Since $L$ is compact, it follows by Proposition~\ref{prop:distal:cocompact} that $H$ does not act distally on $K/D$, for any closed $H$-invariant subgroup $D$ of $L$.  In particular, $L \not\ge \Dist_K(H)$, so $R \not\ge \Dist_K(H)$.

If $H$ is uniscalar and flat on $K$, then $\nub_K(H)$ is the intersection of all compact open $H$-invariant subgroups of $K$, in other words $\nub_K(H)= \Res_K(H)$; since $\Res_K(H) \ge \Dist_K(H)$, it follows that $R \not\ge \nub_K(H)$ in this case.

In particular, we see from the arguments above that if $R = \Dist_K(H)$ is compact, then there exist arbitrarily small open neighbourhoods $V/R$ of the trivial coset in $G/R$ such that $V$ is a compact open subgroup of $G$ that is normalized by $H$.  In particular, in this case $H$ is uniscalar and flat on $K$ and
$$\nub_K(H) = \Dist_K(H) = \Res_K(H).$$
Moreover, we see by Corollary~\ref{cor:distal_compact_extension} that $\Dist_K(H)$ has no proper open $H$-invariant subgroups, so that $\Dist_K(H) = \Res^\infty_K(H)$.

So (i) $\Rightarrow$ (iii).  The implications (iii) $\Rightarrow$ (ii) and (ii) $\Rightarrow$ (i) are immediate, so (i)--(iii) are equivalent.  The fact that $H$ acts ergodically on $\Dist_K(H)$ now follows from Proposition~\ref{prop:resinfty_ergodic}(i), and if $\Res^\infty_K(H)$ is metrizable then $\Res^\infty_K(H) = \Dist^*_K(H)$ by Proposition~\ref{prop:resinfty_ergodic}(iii), finishing the proof.
\end{proof}

\begin{proof}[Proof of Corollary~\ref{corintro:ergodic_nub}]
By replacing $G$ with $G \rtimes H$ (where $G$ is embedded in the semidirect product as an open subgroup), we may assume $H$ is a subgroup of $G$.

The group $\nub_G(H)$ is compact, so the hypotheses of Theorem~\ref{thm:distal} apply with $K = \nub_G(H)$.  In particular, $\nub^2_G(H) = \Res_{\nub_G(H)}(H)$ and $H$ acts ergodically on $\nub^2_G(H)$.

Now suppose that $H_{\us}$ is finitely generated.  Applying Theorem~\ref{thm:distal} again, this time with $K = G$, we see that $H_{\us}$ normalizes a compact open subgroup, and hence $H_{\us}$ acts ergodically on $\nub_G(H_{\us})$.

Since $H$ is finitely generated, the quotient $H/H_{\us}$ is finitely generated, so $H$ is flat of finite rank.  By Theorem~\ref{thm:flat:nubdecomp}, we have
\[
\nub_G(H) = \nub_G(H_{\us})\prod^n_{i=1}\nub_G(\alpha_i),
\]
where $X = \{\alpha_1,\dots,\alpha_n\}$ is a finite subset of $X$.  By Theorem~\ref{basic:nub_characterizations}, $\nub_G(\alpha_i)$ has no proper open $\alpha$-invariant subgroups; similarly $\nub_G(H_{\us})$ has no proper open $H_{\us}$-invariant subgroups.

Now let $U$ be an $H$-invariant subgroup of $\nub_G(H)$.  Then $U \cap \nub_G(H_{\us})$ is an open $H_{\us}$-invariant subgroup of $\nub_G(H_{\us})$, so $U \ge \nub_G(H_{\us})$, and similarly $U \ge \nub_G(\alpha_i)$ for all $i$.  Hence $U \ge \nub_G(H)$.  Since $\nub_G(H)$ is compact, it follows that $H$ has no proper tidy subgroups for its action on $\nub_G(H)$, so $\nub^2_G(H) = \nub(H)$ as required. 
\end{proof}

\begin{ex}\label{ex:nondistal}This example is due to Kepert--Willis and Bhattacharjee--MacPherson (\cite{KepertWillis}, \cite{BhattacharjeeMacPherson}).

Let $F$ be a non-abelian finite simple group and let $R = \prod_{\bZ}F$.  Then we can form a semidirect product $R \rtimes \Sym(\bZ)$, where $\Sym(\bZ)$ has the discrete topology and acts by permuting the copies of $F$.  Let $A$ be a subgroup of $\Sym(\bZ)$ with the following properties:
\begin{enumerate}[(a)]
\item $A$ is transitive on $\bZ$;
\item $A$ is finitely generated as an abstract group;
\item For all $a \in A$, every orbit of $\langle a \rangle$ on $\bZ$ is finite;
\item For all $a \in A$, the symmetric difference of $\bN$ and $a\bN$ is finite.
\end{enumerate}
Such a permutation group was obtained by Bhattacharjee and MacPherson: they show (\cite[Theorem~1.2]{BhattacharjeeMacPherson}) that the free group on $2$ generators has a faithful transitive action on $\bZ$ with the required properties.

For each $i \in \bZ$ let $S_i$ be the subgroup $\prod_{j \ge i}F$ of $R$; and let $S$ be the ascending union $\bigcup_{i \le 0}S_i$, equipped with the topology extending the natural topology of $S_0$ (so $S_0$ is embedded in $S$ as a compact open subgroup).  Condition (d) ensures that $A$ normalizes $S$, preserving the topology of $S$, so that there is a subgroup $G = S \rtimes A$ of $R \rtimes \Sym(\bZ)$, and moreover the subgroups $S_i$ generate a group topology on $G$, under which $G$ is a \tdlc group.  Indeed, given conditions (c) and (d), we see that each element $a \in A$ preserves intervals $[j,+\infty)$ in $\bZ$ where $j$ can be made arbitrarily large, and consequently $a$ normalizes subgroups $S_{j}$ such that $j \rightarrow +\infty$.  Such a collection of subgroups forms a base of neighbourhoods of the identity in $G$, so each $a \in A$ is anisotropic, and indeed $G$ as a whole is anisotropic, that is, $G^\dagger = \triv$.  Conditions (a) and (b) ensure that $G$ is compactly generated (it is generated by $S_0$ and $A$) and also that $G$ does not have any non-trivial compact normal subgroups: indeed, using the transitivity of $A$, it can be seen that every non-trivial normal subgroup of $G$ contains $T = \bigoplus_{\bZ}[F,F] = \bigoplus_{\bZ}F$, which already fails to be relatively compact in $G$.

Although $G^\dagger$ is trivial, we can easily see that this example does not contradict Theorem~\ref{thm:trivial_core:nondistal}: $A$ does not act distally on $S$, because given any element $t \in S$ such that $t(i) = 1$ for all $i \neq 0$ and $t(0) \in F \setminus \{1\}$, then the pair $(t,1)$ is proximal for the action of $A$ on $S$.
\end{ex}

\subsection{Eigenfactors}

We recall some of the theory of eigenfactors as set out in \cite{WillisFlat}.

Let $H$ be a flat subgroup of the \tdlc group $G$, and let $U$ be a compact open subgroup of $G$ that is tidy for $U$.  A \defbold{$U$-eigenfactor}\index{eigenfactor} for $H$ is a closed subgroup $K$ of $U$ with the following properties:

\begin{enumerate}[(a)]
\item $K$ is commensurated by $H$.
\item The set $\{hKh\inv \mid h \in H\}$ is totally ordered by inclusion.
\item $K$ is the intersection of the set $\{hUh\inv \mid h \in H, hKh\inv \ge K\}$.
\end{enumerate}

\begin{thm}[\cite{WillisFlat} Lemma~6.2 and Theorem~6.8]\label{willis:eigenfactor_decomposition}Let $G$ be a \tdlc group, let $H$ be a flat subgroup of $G$ such that $H/H_{\us}$ is finitely generated and let $U$ be a compact open subgroup of $G$ that is tidy for $H$.  Then there are only finitely many $U$-eigenfactors for $H$, and $U$ can be expressed as a product of the distinct $U$-eigenfactors (in some order).\end{thm}

Theorem~\ref{willis:eigenfactor_decomposition} gives some insight into the structure of the group
\[
\langle hUh\inv \mid h \in H \rangle.
\]

\begin{cor}\label{cor:titscore:uzero}Let $G$ be a \tdlc group, let $H$ be a flat subgroup of $G$ such that $H/H_{\us}$ is finitely generated and let $U$ be a compact open subgroup of $G$ that is tidy for $H$.  Let $U_0 = \bigcap_{h \in H}hUh\inv$.
\begin{enumerate}[(i)]
\item Let $K$ be a $U$-eigenfactor of $H$.  Then there exists $h \in H$ such that $K = (\con(h) \cap K)U_0$.
\item $G^\dagger_HU_0$ is the group generated by all $H$-conjugates of $U$.  In particular, $G^\dagger_HU_0$ is an open subgroup of $G$, and $UG^\dagger_H/\overline{G^\dagger_H}$ is normalized by $H$.
\end{enumerate}
\end{cor}

\begin{proof}
(i)
It is clear that $U_0$ is a $U$-eigenfactor of $H$; moreover it is the only $U$-eigenfactor that is normalized by $H$.  For any other $U$-eigenfactor $K$, we see that the total order on $\{hKh\inv \mid h \in H\}$ under inclusion is discrete and has no minimal or maximal elements, so it is order-isomorphic to $\bZ$.  Thus given $h \in H$ such that $hKh\inv < K$, then $\bigcap_{n \ge 0}h^nKh^{-n}$ is the intersection of all $H$-conjugates of $K$, so that $\bigcap_{n \ge 0}h^nKh^{-n} = U_0$.

In other words, $h$ induces a contracting self-map on the coset space $K/U_0$.  By Theorem~\ref{thm:bw:relative_contraction}, it follows that $K \subseteq \con(h)U_0$, so $K = (\con(h) \cap K)U_0$.

(ii)
By Proposition~\ref{prop:twosided:stable}, $G^\dagger_H$ is normalized by $U$, so $G^\dagger_HU_0$ is a group.  We see from part (i) that $G^\dagger_HU_0$ contains every $U$-eigenfactor, so by Theorem~\ref{willis:eigenfactor_decomposition}, $U \le G^\dagger_HU_0$.  Thus $G^\dagger_HU_0 = G^\dagger_HU$.  Clearly, $G^\dagger_HU_0$ is $H$-invariant, so the quotient $UG^\dagger_H/\overline{G^\dagger_H}$ is normalized by $H$.

Let $R$ be the group generated by all $H$-conjugates of $U$.  Then $R \le G^\dagger_HU$ since $G^\dagger_HU$ is $H$-invariant.  On the other hand $G^\dagger_H$ is a subgroup of $R$, since $R$ is open and $H$-invariant, and also $U \le R$, so in fact we must have $R = G^\dagger_HU$.
\end{proof}

\subsection{Almost flat actions}\label{sec:relative_discrete}

Let $G$ be a \tdlc group and let $H$ be an almost finite-rank flat subgroup of $G$.  Then $\Res_G(H)$ is expressible in terms of nubs and contraction groups, as stated in Theorem~\ref{thmintro:relative_discrete}.  In fact we will prove a result with slightly more general hypotheses.

\begin{thm}\label{thm:relative_discrete:general}Let $G$ be a \tdlc group, let $H$ be a subgroup of $G$, and suppose there is a cocompact closed subgroup $K$ of $\overline{H}$ such that $K$ is flat on $G$ and such that $K/K_{\us}$ is finitely generated.
\begin{enumerate}[(i)]
\item The following subgroups of $G$ are all equal to $\Res_G(H)$:
\[
\Res_G(K), \; \overline{G^\dagger_H}\nub_G(K), \; \overline{G^\dagger_H}\nub_G(K_{\us}).
\]
\item The normalizers of $\Res_G(H)$ and $\Res^\infty_G(H)$ in $G$ are open.  Indeed, $\Res_G(H)$ is normalized by every tidy subgroup for the action of $K$ on $G$.
\item $H$ is anisotropic and flat on $\N_G(G^\dagger_H)/\overline{G^\dagger_H}$.
\item $\overline{G^\dagger_H}$ is a cocompact normal subgroup of $\Res_G(H)$.  Indeed, $\Res_G(H)/\overline{G^\dagger_H}$ is the nub of the action of $H$ on $\N_G(G^\dagger_H)/\overline{G^\dagger_H}$.
\item The action of $H$ on $\N_G(\Res_G(H))/\Res_G(H)$ has SIN.
\item Suppose that $G$ is metrizable.  Then
\[
\Dist^*_G(H) = \Dist_G(H) = \Res^\infty_G(H).
\]
\item Suppose that $\overline{H}$ is compactly generated.  Then
\[
\Dist_G(H) = \Res^\infty_G(H) =\Res_G(H).
\]
\end{enumerate}
\end{thm}

We begin the proof with the case where $H$ is flat and uniscalar, a situation which has several equivalent characterizations.

\begin{lem}\label{discrete_residual:uniscalar}Let $G$ be a \tdlc group and let $H$ be a flat subgroup of $G$.  Then the following are equivalent:
\begin{enumerate}[(i)]
\item $H$ is uniscalar;
\item $\Res_G(H)$ is compact;
\item $\Res_G(H) = \nub_G(H)$;
\item $\overline{G^\dagger_H} = \nubl_G(H)$.
\end{enumerate}
\end{lem}

\begin{proof}
Suppose that $H$ is uniscalar.  Then there exists an $H$-invariant compact open subgroup $U$.  Moreover, every $H$-invariant open subgroup $O$ contains a compact open $H$-invariant subgroup $O \cap U$, and since $H$ is uniscalar, the nub of $H$ is precisely the intersection of all $H$-invariant compact open subgroups.  Thus (i) implies (ii) and (iii).  For each $h \in H$, we have $\nub(h) = \overline{\con(h)}$ by \cite[Proposition~5.4]{WillisNub}, so (i) implies (iv).

Conversely, suppose that at least one of (ii), (iii) and (iv) holds.  Then $\con(\alpha)$ is relatively compact for all $\alpha \in H$, since we have $\con(\alpha) \le \overline{G^\dagger_H}$ and $\con(\alpha) \le \Res_G(H)$, and both $\nubl_G(H)$ and $\nub_G(H)$ are compact.  By Proposition~\ref{basic:anisotropic}, it follows that $H$ is uniscalar, so each of (ii), (iii) and (iv) implies (i).  Hence (i), (ii), (iii) and (iv) are all equivalent as required.
\end{proof}

\begin{proof}[Proof of Theorem~\ref{thm:relative_discrete:general}]
Since the relative Tits cores, discrete residual, ($1$-)distal residual and nub defined with respect to the action of a subgroup $H$ on a \tdlc group $G$ are all unaffected by replacing $H$ with $\overline{H}$, we may assume that $H$ is closed.

Let $K$ be a cocompact subgroup of $H$ such that $K$ is flat on $G$ and $K/K_{\us}$ is finitely generated, and recall that $G^\dagger_H = G^\dagger_K$ by Theorem~\ref{thmintro:TitsCore:cocompact}.  Consider the set
\[
\mc{N} = \{ U \le G \mid U \text{ is tidy for $K$ on } G\}.
\]
Let $T = \overline{G^\dagger_K}$ and let $O = \N_G(G^\dagger_K)$; note that $O$ is $H$-invariant.  By Proposition~\ref{prop:twosided:stable}, we have $O \ge \langle \mc{N} \rangle$.  In particular, $O$ is open in $G$, so $\Res_G(H) \le O$ and $\Res_G(K) \le O$.

By Corollary~\ref{cor:titscore:uzero}, the group $TU$ is a $K$-invariant open subgroup of $G$, for all $U \in \mc{N}$.  By definition, $\nub_G(K) =\bigcap_{U \in \mc{N}}U$, so by applying Lemma~\ref{lem:intersection:quotient} to the quotient map $O \rightarrow O/T$, we have
\[
\bigcap_{U \in \mc{N}}(TU) = T\nub_G(K).
\]
In particular, we see that
\[
\Res_G(K) \le T\nub_G(K).
\]
By Theorem~\ref{thm:flat:nubdecomp}, we have $\nub_G(K) = \nubl_G(K)\nub_G(K_{\us})$, and since $\nub(k) \le \overline{\con(k)}$ for each $k \in K$, we see that $\nubl_G(K) \le T$.  Thus
\[
T\nub_G(K) = T\nub_G(K_{\us}).
\]
Let us now fix some $U \in \mc{N}$.  Let $Y$ be a $K$-invariant open subgroup of $G$.  Then $U$ is $K_{\us}$-invariant, so $Y \cap U$ is also $K_{\us}$-invariant, and hence $Y \cap U$ is tidy for $K_{\us}$.  Thus $\nub_G(K_{\us}) \le Y$.  In addition, $\con(k) \le Y$ for all $k \in K$, since $Y$ is an open $K$-invariant identity neighbourhood.  Hence $Y \ge T\nub_G(K_{\us})$.  We conclude that
\[
\Res_G(K) = T\nub_G(K) = T\nub_G(K_{\us}).
\]

The image $UT/T$ is normalized by $K$ by Corollary~\ref{cor:titscore:uzero}, so $K$ is uniscalar and flat on $O/T$.  By Lemma~\ref{cocompact:uniscalar_flat}, $H$ is also flat on $O/T$, and $H$ is anisotropic on $O/T$ by Corollary~\ref{cor:anisotropic_quotient}, proving (iii).

By Lemma~\ref{cocompact:uniscalar_flat}, we see that $\Res_{O/T}(H)/T = \Res_{O/T}(K)/T$.  Moreover, every $K$-invariant open subgroup of $G$ contains $T$, so in fact
\[
\Res_{O/T}(K) = \Res_G(K) \text{ and } \Res_{O/T}(H) = \Res_G(H).
\]
Thus $\Res_G(H) = \Res_G(K)$, completing the proof of (i).

We have seen that $\Res_G(H) = T\nub_G(K_{\us})$.  Given $U \in \mc{N}$, then $U$ normalizes $T$ by Proposition~\ref{prop:twosided:stable} and $U$ normalizes $\nub_G(K_{\us})$ by Corollary~\ref{cor:nub:normalizer}.  Thus $U$ normalizes $\Res_G(H)$; in particular, $\N_G(\Res_G(H))$ is open.  We see that $O \ge \Res^\infty_G(H) \ge T$, so in fact $\Res^\infty_G(H) = \Res^\infty_{O/T}(H)$.  By Proposition~\ref{prop:resinfty_ergodic}(ii), the group $\Res^\infty_{O/T}(H)/T$ is normalized by every compact $H$-invariant subgroup of $O/T$; since $H$ is flat and uniscalar on $O/T$, it follows that $\N_G(\Res^\infty_{G}(H))$ is open, completing the proof of (ii).

Since $H$ is flat and uniscalar on $O/T$, by Lemma~\ref{discrete_residual:uniscalar} it follows that $\Res_G(H)/T = \nub_{O/T}(H)$, proving (iv).  A compactness argument then shows that the $H$-invariant compact open subgroups of $\N_O(\Res_G(H))/\Res_G(H)$ form a base of identity neighbourhoods, from which (v) follows.

Now consider the action of $H$ on quotients of $O/T$.  Certainly the action of $H$ on $O/\Res_G(H)$ is distal and $\Res_G(H)/T$ is compact.  By applying Corollary~\ref{cor:distal_compact_extension}, we see that $\Dist_{O/T}(H)/T \le \Res^\infty_{O/T}(H)/T$, so we have the inequalities
\[
\Res^\infty_{O/T}(H)/T \ge \Dist_{O/T}(H)/T \ge \Dist^*_{O/T}(H)/T.
\]
Moreover $\Dist_G(H) \ge T$, so in fact $\Dist_G(H) = \Dist_{O/T}(H)$, and similarly $\Dist^*_G(H) = \Dist^*_{O/T}(H)$.

If $\Res^\infty_{O/T}(H)/T$ is metrizable then in fact
\[
\Res^\infty_{O/T}(H)/T = \Dist_{O/T}(H)/T = \Dist^*_{O/T}(H)/T
\]
by Proposition~\ref{prop:resinfty_ergodic}(iii), implying that $\Res^\infty_{G}(H) = \Dist_{G}(H) = \Dist^*_{G}(H)$, which proves (vi).

Let us now suppose that $H$ is compactly generated.  By (iv) we have $\Res_{O/T}(H)/T = \nub_{O/T}(H)$, so in particular $\Res_{O/T}(H)/T$ is compact.  Hence by Theorem~\ref{thm:distal}, we have 
\[
\Dist_{O/T}(H) =  \Res^\infty_{O/T}(H) = \Res_{O/T}(H).
\]
Let $A$ represent $\Dist$ or $\Res^\infty$.  In each case it is clear that $A_O(H)$ contains $T$, so that $A_{O/T}(H) = A_O(H)$, and hence $A_O(H) = \Res_O(H) = \Res_G(H)$.  Furthermore, we have $A_O(H) \le A_G(H)$, since if the action of $H$ on $G/A_G(H)$ is distal or admits a descending series of residually discrete sections, then the same is true of $OA_G(H)/A_G(H)$, and hence of $O/(A_G(H) \cap O)$ (here we use the continuity of the natural map from $O/(A_G(H) \cap O)$ to $OA_G(H)/A_G(H)$).  Since in each case we also have $A_G(H) \le \Res_G(H)$, we complete a cycle of inequalities and conclude that $A_O(H) = A_G(H) = \Res_G(H)$, so in particular
\[
\Dist_{G}(H) = \Res^\infty_{G}(H) = \Res_{G}(H),
\]
proving (vii).
\end{proof}

\begin{proof}[Proof of Corollary~\ref{corintro:relative_discrete:polycyclic}]By Theorem~\ref{thm:nilpotent:flat}, there is a polycyclic subgroup $L$ of $K$ such that $L$ is flat on $G$ and $\overline{L}$ has finite index in $K$.  We therefore have $\Res_G(H) = \overline{G^\dagger_L}\nub_G(L)$ by Theorem~\ref{thm:relative_discrete:general}(i).  We have $\overline{G^\dagger_L} = \overline{G^\dagger_H}$ by Theorem~\ref{thmintro:TitsCore:cocompact}.  Since $L$ is polycyclic, Theorem~\ref{thm:flat:nubdecomp} ensures that $\overline{G^\dagger_L} \ge \nub_G(L)$.  Hence $\Res_G(H) = \overline{G^\dagger_H}$.  Parts (iii) and (iv) of Theorem~\ref{thm:relative_discrete:general} then ensure that the action of $H$ on $\N_G(G^\dagger_H)/\overline{G^\dagger_H}$ is uniscalar with trivial nub, so this action is uniscalar and smooth, in other words, the open subgroups of $\N_G(G^\dagger_H)/\overline{G^\dagger_H}$ normalized by $H$ are a base of identity neighbourhoods.  The desired conclusion for the action on $G/\overline{G^\dagger_H}$ follows from the fact that $\N_G(G^\dagger_H)$ is open in $G$.
\end{proof}

\begin{proof}[Proof of Corollary~\ref{corintro:relative_discrete:hji}]
Suppose that $H$ normalizes a compact open subgroup $U$ of $G$.  It is a general fact that a just infinite profinite group only has finitely many subgroups of any given index; see for instance \cite[Corollary~2.5]{ReidJI}.  In particular, for each natural number $n$, the intersection $U_n$ of all open subgroups of $U$ of index at most $n$ is a characteristic open subgroup of $U$, and the set $\{U_n \mid n \in \bN\}$ is a base of identity neighbourhoods in $U$, and hence also in $G$.  We see that $U_n$ is normalized by $H$ for each $n \in \bN$, so case (a) of the dichotomy is satisfied.

Suppose that $H$ does not normalize any compact open subgroup of $G$.  It then follows from Lemma~\ref{cocompact:uniscalar_flat} that the cocompact subgroup $K$ of $H$ also does not normalize any compact open subgroup of $G$.  Since $K$ is flat, we conclude that $K$ is not uniscalar and hence $\con(k)$ is non-trivial for some $k \in K$.  Consequently $\overline{G^\dagger_H}$ is non-discrete, so $\Res_G(H)$ is non-discrete.  Let $V$ be a compact open subgroup of $G$ that is tidy for $K$.  By Theorem~\ref{thm:relative_discrete:general}(ii), $V$ normalizes $\Res_G(H)$, so the intersection $N = V \cap \Res_G(H)$ is a non-discrete, in particular non-trivial, closed normal subgroup of $V$.  By hypothesis, $V$ is just infinite, so $N$ is open in $V$ and hence also in $G$.  Thus $\Res_G(H)$ is an open subgroup of $G$ normalized by $H$; by definition, $\Res_G(H)$ is contained in any other open subgroup of $G$ normalized by $H$.  Thus case (b) of the dichotomy is satisfied.
\end{proof}

Using Theorem~\ref{thm:relative_discrete:general}, we obtain another characterization of the discrete residual of an almost finite-rank flat subgroup.

\begin{cor}\label{discrete_smooth}Let $G$ be a \tdlc group and let $H$ be an almost finite-rank flat subgroup of $G$.  Then $\Res_G(H) = R$ is the smallest closed subgroup of $G$ with both of the following properties:
\begin{enumerate}[(a)]
\item $\N_G(R)$ is open in $G$ and contains $H$;
\item The action of $H$ on $\N_G(R)/R$ has SIN.
\end{enumerate}
\end{cor}

\begin{proof}Let $R$ be a closed subgroup of $G$ satisfying (a) and (b).  Then $H$ normalizes arbitrarily small (compact) open subgroups $U/R$ of $\N_G(R)/R$, and if $U/R$ is such a subgroup, then $U$ is $H$-invariant and open in $G$, so $\Res_G(H) \le U$.  Hence $\Res_G(H) \le R$.

Now consider $\Res_G(H) = R$ itself.  Then (a) is satisfied by Theorem~\ref{thm:relative_discrete:general}(ii) and (b) is satisfied by Theorem~\ref{thm:relative_discrete:general}(v).
\end{proof}

The fact that $\overline{G^\dagger_H}$ is cocompact in $\Res_G(H)$ allows us to prove a stability result for discrete residuals on quotients.

\begin{prop}Let $G$ be a \tdlc group, let $H$ be an almost finite-rank flat subgroup of $G$ and let $N$ be a closed normal $H$-invariant subgroup of $G$.  Then
\[
\Res_{G/N}(H)/N = \overline{\Res_G(H)N}/N.
\]
\end{prop}

\begin{proof}
Let $O = \N_G(G^\dagger_H)$, and note that $O$ is open in $G$ by Corollary~\ref{almostflat:titscore}.  We have $\Res_{G/N}(H) \le UN$ for any open $H$-invariant subgroup $U$ of $G$, since $UN/N$ is also open and $H$-invariant.  In particular, $\Res_{G/N}(H)$ is contained in $ON$; similarly, $\Res_G(H) \le O$.  Moreover, $H$ is an almost finite-rank flat subgroup of $ON$.  Thus we may assume $G = ON$.

Let $T = \overline{G^\dagger_H}$.  Then $R = \overline{TN}$ is normal in $G$.  By Theorem~\ref{thm:relative_discrete:general}(iv), $T$ is cocompact and normal in $\Res_G(H)$, and hence $R$ is cocompact in $\Res_G(H)R$; in particular, $\Res_G(H)R$ is closed in $G$.  We also see that
\[
\overline{\Res_G(H)N} = \Res_G(H)R \text{ and } \Res_{G/N}(H) = \Res_{G/R}(H).
\]
Certainly $\Res_{G/R}(H) \ge \Res_G(H)R$, since any $H$-invariant open subgroup of $G/R$ is the image of an $H$-invariant open subgroup of $G$.  To finish the proof, it remains to prove that $\Res_G(H)R \ge \Res_{G/R}(H)$.
  
By Theorem~\ref{thm:relative_discrete:general}, $H$ is uniscalar and flat on $O/T$ and $\Res_G(H)/T = \nub_{O/T}(H)$.  In particular, $\Res_G(H)/T$ is an intersection of compact open subgroups of $O/T$.

Considering the natural homomorphism from $O/T$ to $OR/R$, we see by Lemma~\ref{lem:intersection:quotient} that
\[
\Res_G(H)R/R = \bigcap_{U \in \mc{U}}(UR/R),
\]
where $\mc{U}$ is the set of $H$-invariant open subgroups $U$ of $O$ such that $U/T$ is compact.  Given $U \in \mc{U}$, then $UR/R$ is an $H$-invariant open subgroup of $G/R$, so $UR \ge \Res_{G/R}(H)$.  Hence
\[
\Res_G(H)R \ge \Res_{G/R}(H),
\]
as required.
\end{proof}

\subsection{The Mautner phenomenon}\label{sec:Mautner}

The \emph{Mautner phenomenon} is a collection of related results of the following form: given a suitable action of a group $G$ on a set $X$, and a point ${x \in X}$ that is fixed by some subgroup $H \le G$, then the stabilizer of $x$ in $G$ necessarily contains not just $H$, but a much larger subgroup (often $G$ itself) that depends on the dynamics of the conjugation action of $H$ on $G$.  The concept originates in the ergodic theory of flows on manifolds, and also plays an important role in the representation theory of locally compact groups: see for instance \cite{Mautner}, \cite{Moore} and \cite{Wang}.  We can define the phenomenon for topological groups in general terms as follows.

\begin{defn}Let $G$ be a group acting on a Hausdorff topological space $X$, and let $x \in X$ be a fixed point of the action.  Then $x \in X$ is an \defbold{isolated point} of the action of $G$ if for all $y \in X \setminus \{x\}$, the closure of the $G$-orbit of $y$ does not contain $x$; in other words, no orbit of the action of $G$ on $X$ accumulates at $x$.

Let $G$ be a topological group and let $H \le G$.  We say that $H$ \defbold{exhibits the Mautner phenomenon in $G$}, or more briefly $H$ is an \defbold{MP-subgroup}\index{MP-subgroup} of $G$, if the following condition holds:

$(*)$ Let $X$ be a Hausdorff topological space admitting a $G$-action by homeomorphisms such that the map $G \rightarrow X; g \mapsto gx$ is continuous for all $x \in X$.  Suppose $x \in X$ is an isolated point of the action of $H$.  Then $x$ is a fixed point of the action of $G$.\end{defn}

We can extract some more familiar versions of the Mautner phenomenon from this definition. 
 
\begin{prop}\label{prop:Mautner:specific}
Let $G$ be a topological group and let $H$ be an MP-subgroup of $G$.  Then the following holds.
\begin{enumerate}[(i)]
\item Let $X$ be a metrizable space on which $G$ acts continuously, and suppose $H$ acts distally with respect to some metric $d$ for $X$ (that is, $\mathrm{\inf}\{d(hx,hy) \mid h \in H\} > 0$ for any pair $(x,y)$ of distinct points).  Then every point fixed by $H$ is fixed by $G$.
\item Let $X$ be a topological space admitting a Borel probability measure, such that $G$ acts continuously and ergodically by measure-preserving maps.  Then the action of $H$ on $X$ is ergodic.
\end{enumerate}
\end{prop}

\begin{proof}
(i)
Let $x \in X$ be a fixed point of $H$ and let $y \in X \setminus \{x\}$.  Then since $(x,y)$ is not a proximal pair for $H$, the $H$-orbit of $y$ does not accumulate at $x$.  Thus $x$ is an isolated fixed point of $H$, so $x$ is fixed by $G$.

(ii)
Assume for a contradiction that there exists a measurable subset $Y$ of $X$ such that $0 < \mu(Y) < 1$ and $\mu(hY \setminus Y) = 0$ for all $h \in H$, and consider the space $L^2(X)$ of square-integrable functions from $X$ to $\bC$ modulo essentially zero functions.  Then the indicator function $\phi_Y$ of $Y$ is (a representative of) a non-zero element of $L^2(X)$ that is fixed by $H$.  Now $L^2(X)$ is a normed vector space, so in particular a metric space, on which $G$ acts continuously by isometries, so $\phi_Y$ is fixed by $G$ by part (i).  But then $\mu(gY \setminus Y) = 0$ for all $g \in G$, so the action of $G$ on $X$ is not ergodic, a contradiction.
\end{proof}

A natural criterion for the Mautner phenomenon can be expressed in terms of a $1$-distal residual.  (Note that if $H \le D \le G$, then the translation action of $H$ on $G/D$ is the same as the conjugation action of $H$ on $G/D$.)

 \begin{thm}\label{thm:Mautner}Let $G$ be a topological group and let $H$ be a subgroup of $G$.  Let $D = \Dist^*_{G/H}(H)$.  Then $H$ is an MP-subgroup of $D$.  Moreover, $H$ is an MP-subgroup of $G$ if and only if $D = G$.\end{thm}

\begin{proof}
Suppose $H$ is an MP-subgroup of $G$, and let $R$ be a closed subgroup of $G$ such that $H \le R \le G$ and $H$ acts distally at $1$ on $G/R$ by translation.  Then the map $G \rightarrow G/R; g \mapsto gxR$ is continuous for all $x \in G$, and $R$ is an isolated point of the action of $H$ on $G/R$.  Hence $R$ is a fixed point of the action of $G$ on $G/R$ by translation, in other words $R = G$, and hence $\Dist^*_{G/H}(H) = G$.

Let $D = \Dist^*_{G/H}(H)$.  It remains to show that $H$ is an MP-subgroup of $D$.

Let $X$ be a Hausdorff topological space admitting a $D$-action by homeomorphisms such that the map $D \rightarrow X; g \mapsto gx$ is continuous for all $x \in X$.  Suppose $x \in X$ is an isolated point of the action of $H$.  Then the stabilizer $D_x$ is a closed subgroup of $D$ such that $H \le D_x$.

Suppose $D_x < D$.  Then by hypothesis, $H$ does not act distally at $1$ on $G/D_x$.  Since $H$ acts distally at $1$ on $G/D$, it follows by Lemma~\ref{lem:distal_extension} that $H$ does not act distally at $1$ on $D/D_x$, that is, there exists $g \in D \setminus D_x$ such that the set $\{hgD_x \mid h \in H\}$ accumulates at $D_x$.  Then there are nets $(h_i)_{i \in I}$ and $(k_i)_{i \in I}$ in $D_x$ such that $(h_igk_i)_{i \in I}$ converges to the identity, and thus $(h_igk_ix)_{i \in I}$ converges to $x$.  Since $x$ is fixed by $D_x$, in fact $(h_iy)_{i \in I}$ converges to $x$, where $y = gx$.  But $x$ is an isolated point of $H$, so we must have $y = x$.  Thus $g \in D_x$, a contradiction.  Hence our assumption that $D_x < D$ was false, in other words, $x$ is fixed by $D$, proving that $H$ is an MP-subgroup of $D$.
\end{proof}

Theorem~\ref{thmintro:Mautner} now follows.  We also have the following sufficient conditions for $H$ to be an MP-subgroup.

\begin{cor}Let $G$ be a \tdlc group and let $H$ be a subgroup of $G$.  Then $H$ is an MP-subgroup of $\overline{\Dist^*_G(H)H}$ and of $\overline{G^\dagger_HH}$.\end{cor}

We recall the basic examples Example~\ref{ex:tree} and Example~\ref{ex:linear}, where the relative Tits cores were quite large.

\begin{cor}
\begin{enumerate}[(i)]
\item Let $G$ be the automorphism group of a locally finite regular tree of degree at least $3$, and let $g \in G$ be hyperbolic.  Then $\langle g \rangle$ is an MP-subgroup of $G^+$ if $g \in G^+$, and $\langle g \rangle$ is an MP-subgroup of $G$ if $g \not\in G^+$.
\item Let $G = \mathrm{SL}_n(\bQ_p)$ and let $g \in G$ such that $\con_G(g) \neq \triv$.  Then $\langle g \rangle$ is an MP-subgroup of $\mathrm{SL}_n(\bQ_p)$.
\end{enumerate}
\end{cor}

Under similar hypotheses to Theorem~\ref{thm:relative_discrete:general}, we can show that the Mautner phenomenon is controlled by the subgroup $\overline{\Res_G(H)H}$.  We first prove a lemma.

\begin{lem}\label{lem:distal_coset:SIN}Let $G$ be a \tdlc group and let $H$ be a closed subgroup of $G$.  Suppose that the action of $H$ on $G$ by conjugation has SIN.  Then the translation action of $H$ on $G/H$ is distal.\end{lem}

\begin{proof}
Let $x,y,z \in G$ and suppose there is a net $(h_i)_{i \in I}$ in $H$ such that $(h_ixH,h_iyH)$ converges to $(zH,zH)$.  Let $U$ be a compact open subgroup of $G$ normalized by $H$.  For $i$ large enough we have $h_ix,h_iy \in zUH$, so $x\inv y \in HUH = UH$.  Since $U$ can be made arbitrarily small and $H$ is closed, in fact $x\inv y \in H$, that is, $xH = yH$.
\end{proof}

\begin{prop}Let $G$ be a \tdlc group and let $H$ be a closed subgroup of $G$ that is compactly generated and almost flat on $G$.  Suppose either that $G$ is metrizable, or there is a polycyclic subgroup $K$ of $H$ such that $\overline{K}$ is cocompact in $H$.

Then $\Dist^*_{G/H}(H) = \overline{\Res_G(H)H}$.  In particular, $H$ is an MP-subgroup of $G$ if and only if $G = \overline{\Res_G(H)H}$.
\end{prop}

\begin{proof}
Let $R = \Res_G(H)$.  We have $\overline{G^\dagger_H} \le \Dist^*_G(H) \le R$.  If $G$ is metrizable then $\Dist^*_G(H) = R$ by Theorem~\ref{thm:relative_discrete:general}.  If instead $H$ has a polycyclic subgroup with cocompact closure, then $R = \overline{G^\dagger_H}$ by Corollary~\ref{corintro:relative_discrete:polycyclic}, also ensuring that $\Dist^*_G(H) = R$.  So certainly our hypotheses ensure that $\Dist^*_{G/H}(H) \ge \overline{RH}$.

On the other hand, by Corollary~\ref{discrete_smooth}, $\N_G(R)$ is open and the action of $H$ on $\N_G(R)/R$ is uniscalar and smooth, in other words, the action of $H$ on $\N_G(R)/R$ has SIN.  It follows via Lemma~\ref{lem:distal_coset:SIN} that $H$ acts distally on $\N_G(R)/\overline{RH}$ by translation.  Since $\N_G(R)$ is open in $G$, we see that the action of $H$ on $G/\overline{RH}$ is distal at $1$.  Thus $\Dist^*_{G/H}(H) = \overline{RH}$ as required.  The final conclusion follows from Theorem~\ref{thm:Mautner}.
\end{proof}

We note the following special case for clarity.

\begin{cor}
Let $G$ be a \tdlc group and let $H$ be a polycyclic subgroup of $G$.  Then $H$ is an MP-subgroup of $G$ if and only if $G = \overline{G^\dagger_HH}$.
\end{cor}

\subsection{Subgroups of finite covolume}\label{sec:covolume}

We now derive Theorem~\ref{thmintro:covolume} and its corollary, starting with two lemmas.

\begin{lem}\label{lem:Mautner:covolume}
Let $G$ be a \tdlc group and let $H$ and $K$ be closed subgroups of $G$ such that $K$ has finite covolume in $G$.
\begin{enumerate}[(i)]
\item If $H$ is an MP-subgroup of $G$, then $H$ acts ergodically on $G/K$ by left translation.
\item If $H$ acts ergodically on $G/K$, then given any non-empty open subset $U$ of $G$, we have
\[
G = \overline{HUK}.
\]
If in addition $G$ is metrizable, then the set
\[
\bigcap_{U \in \mc{U}}HUK
\]
is dense in $G$, where $\mc{U}$ is the set of compact open subgroups of $G$.
\end{enumerate}
\end{lem}

\begin{proof}
(i)
We observe that the coset space $G/K$ is a Borel probability space, on which $G$ acts continuously and ergodically (indeed, transitively) by measure-preserving maps.  Thus by Proposition~\ref{prop:Mautner:specific}, $H$ acts ergodically on $G/K$.

(ii)
Suppose that $H$ acts ergodically on $G/K$, and let $U$ be a non-empty open subset of $G$.  Then $UK/K$ is a subspace of $G/K$ of positive measure, so the $H$-invariant subspace $HUK/K$ has a complement of zero measure.  Let $V$ be a compact open subgroup of $G$, and suppose there is $g \in G$ such that $gV \cap HUK = \emptyset$.  Then $gVK \cap HUK = \emptyset$, since $HUK$ is invariant under right translation by $K$.  In other words, $gVK/K$ is disjoint from $HUK/K$ in the coset space $G/K$.  But then $gVK/K$ has zero measure, which is absurd.  This contradiction implies that $HUK$ is dense in $G$.

Now suppose in addition that $G$ is metrizable.  Then $\bigcap_{U \in \mc{U}}HUK = \bigcap_{U \in \mc{U}'}HUK$ where $\mc{U}'$ is a countable set of compact open subgroups of $G$, since $G$ has a base of identity neighbourhoods consisting of countably many compact open subgroups (this follows from Van Dantzig's Theorem together with Lemma~\ref{basic:metrizable}).  The last conclusion follows by the Baire Category Theorem.
\end{proof}

\begin{lem}[See \cite{Raghunathan} Lemma~1.6]\label{lem:covolume:monotone}
Let $G$ be a locally compact group, let $H$ be a closed subgroup of $G$ and let $K$ be a closed subgroup of $H$.  Then $K$ has finite covolume in $G$ if and only if $K$ has finite covolume in $H$ and $H$ has finite covolume in $G$.
\end{lem}

\begin{proof}[Proof of Theorem~\ref{thmintro:covolume}]
Let $D = \Dist^*_{G/H}(H)$.  Since $H$ has finite covolume in $G$, it also has finite covolume in $D$ by Lemma~\ref{lem:covolume:monotone}.  By Theorem~\ref{thm:Mautner}, $H$ is an MP-subgroup of $D$, and hence by Lemma~\ref{lem:Mautner:covolume}, the set $\bigcap_{V \in \mc{V}}HVH$ is dense in $D$, where $\mc{V}$ is the set of compact open subgroups of $D$.  So certainly $K(H)$ contains a dense subset of $D$.  On the other hand, the action of $H$ on $G/D$ is distal at $1$, in other words, the set of $H$-invariant open neighbourhoods of the trivial coset in $G/D$ has trivial intersection.  Let $O/D$ be such a neighbourhood; in other words, $O$ is an $H$-invariant open subset of $G$ that is a union of left cosets of $D$, such that $D \subseteq U$.  Then $O = HOD$ and $O$ is an identity neighbourhood in $G$, so $K(H) \subseteq O$.  Since the intersection of all such sets $O$ is just $D$, we conclude that $K(H) \subseteq D$, and thus $\overline{K(H)} = D$.

By Theorem~\ref{thmintro:TitsCore:cocompact}, we have $G^\dagger = G^\dagger_H$, and it is clear that $G^\dagger_H \le \Dist^*_G(H) \le D$.  Thus $G^\dagger \le D$, completing the proof of (i).

For (ii), we see from Theorem~\ref{thm:Mautner} and Lemma~\ref{lem:Mautner:covolume} that $H$ acts ergodically on $D/H$.  It remains to show that $D$ is the unique largest subgroup of $G$ such that $H \le D$ and $H$ acts ergodically on $D/H$.  So suppose $E$ is another closed subgroup of $G$ such that $H \le E$ and $H$ acts ergodically on $E/H$.  

By Lemma~\ref{lem:covolume:monotone}, $H$ has finite covolume in $E$, so by Lemma~\ref{lem:Mautner:covolume}, the set $L = \bigcap_{W \in \mc{W}}HWH$ is dense in $E$, where $\mc{W}$ is the set of compact open subgroups of $E$.  Suppose that $H$ is not an MP-subgroup of $E$.  Then by Theorem~\ref{thm:Mautner}, there exists a proper closed subgroup $F$ of $E$ such that $H \le F$ and $H$ acts distally at $1$ on $E/F$.  Since $G$ is metrizable, we see that $E/F$ is a locally compact Hausdorff metrizable space, so the fact that no $H$-orbit accumulates at $F$ ensures that there is a proper $H$-invariant neighbourhood $O/F$ of $F$ in $E/F$ that is not dense in $E/F$: for instance, if we specify a metric on $E/F$ compatible with the topology, and $B_{n}$ is the open ball of radius $1/n$ around $F$ with respect to this metric, then by the Baire Category Theorem, there exists $n \in \bN$ such that the set $\bigcup_{h \in H}hB_{n}$ is not dense.  Now $O$ is a neighbourhood of the identity in $E$, so there is $W \in \mc{W}$ such that $W \subseteq O$.  Since $O$ is invariant under left translation by $H$ and right translation by $F \ge H$,  we have $O = HOH$.  Hence $L \subseteq HWH \subseteq O$; in particular, $L$ is not dense in $E$, a contradiction.  Thus in fact $H$ must be an MP-subgroup of $E$.

It now follows by Theorem~\ref{thm:Mautner} that $E = \Dist^*_{E/H}(H)$.  In particular, if $D \not\ge E$, then $H$ does not act distally on $E/(D \cap E)$, that is, there exists $x \in E \setminus D$ such that the $H$-orbit of $x(D \cap E)$ accumulates at the trivial coset.  But then the $H$-orbit of $xD$ accumulates at the trivial coset and $xD$ is a non-trivial element of $G/D$, so $H$ does not act distally on $G/D$, contradicting the definition of $D$.  Thus $D \ge E$, proving (ii).
\end{proof}

\begin{proof}[Proof of Corollary~\ref{corintro:covolume:charsimple}]
By Theorem~\ref{thmintro:covolume}(i), we have $G^\dagger \le \Dist^*_{G/H}(H)$.  Since $G^\dagger$ is dense in $G$ and $\Dist^*_{G/H}(H)$ is closed, it follows that $\Dist^*_{G/H}(H) = G$.  Hence $H$ is an MP-subgroup of $G$ by Theorem~\ref{thm:Mautner}, so $H$ acts ergodically on $G/H$ by Lemma~\ref{lem:Mautner:covolume}.\end{proof}

\begin{rem}
The role of the set $K(H)$ has been previously studied by H. Keynes (\cite{Keynes}) under somewhat different assumptions: $H$ is not necessarily closed and $G$ is not necessarily a \tdlc group, but there exists a compact subset $X$ of $G$ such that $G = XH$.  Keynes shows in this case (\cite[Theorem~2.3]{Keynes}) that the pair $(xH,yH)$ is proximal under the action of $G$ on $G/H$ by left translation if and only if $x^{-1}y \in K(H)$.
\end{rem}

\section{Open envelopes}

\subsection{Reduced envelopes of an almost flat subgroup}\label{sec:reduced_envelope}

Within the class of subgroups normalized by the almost finite-rank flat subgroup $H$, we can consider the open subgroups of $G$ that actually contain $H$.

\begin{defn}Let $G$ be a \tdlc group and let $X \subseteq G$.  An \defbold{envelope}\index{envelope} of $X$ in $G$ is an open subgroup of $G$ that contains $X$.  Say an envelope $E$ of $X$ is \defbold{reduced}\index{envelope!reduced envelope} if, whenever $E_2$ is an envelope of $X$ in $G$, then $|E:E \cap E_2|$ is finite.\end{defn}

The following observations are immediate from the definitions, together with Van Dantzig's theorem.

\begin{lem}Let $G$ be a \tdlc group, let $X \subseteq G$ and let $H = \langle X \rangle$.
\begin{enumerate}[(i)]
\item Let $E$ be an envelope for $X$.  Then $\overline{\Res_G(H)H}$ is a subgroup of $E$.
\item Suppose that $X$ has a reduced envelope $E$ in $G$.  Then all reduced envelopes of $X$ in $G$ are commensurate to $E$, and there is a reduced envelope $E_2 \le E$ of the form $E_2 = \langle U, X \rangle$, where $U$ is a compact open subgroup of $G$.
\end{enumerate}
\end{lem}

We now prove the theorem on reduced envelopes from the introduction.

\begin{proof}[Proof of Theorem~\ref{thmintro:reduced_envelope}]
By Corollary~\ref{cor:titscore:uzero}, the product $G^\dagger_KU_0$ is the group generated by all $K$-conjugates of $U$.  Hence
\[
\langle K, U \rangle = G^\dagger_KU_0K = \overline{(G^\dagger_KK)}U_0.
\]
Clearly $\langle K, U \rangle$ is an envelope for $K$ in $G$.  Since every envelope for $K$ contains $\overline{G^\dagger_KK}$, which is a cocompact subgroup of $\langle K,U \rangle$, we see that $\langle K,U \rangle$ is reduced, proving (i).

Now consider $H \le G$ such that $K$ is cocompact in $\overline{H}$.  Then $G^\dagger_H = G^\dagger_K$ by Theorem~\ref{thmintro:TitsCore:cocompact}.  Let $O = \N_G(G^\dagger_H)$.  Since the action of $H$ on $O/\overline{G^\dagger_H}$ is uniscalar and flat, there exists an $H$-invariant subgroup $A$ of $G$ such that $\overline{G^\dagger_H} \le A$ and $A/\overline{G^\dagger_H}$ is a compact open subgroup of $O/\overline{G^\dagger_H}$.  Thus $E = AH$ is an envelope for $H$ in $G$.  Since $\overline{G^\dagger_HH}$ is cocompact in $E$, in fact $E$ is a reduced envelope for $H$.  Since $K$ is cocompact in $\overline{H}$, the open subgroup $AK$ has finite index in $E$; since every envelope of $K$ contains $\overline{G^\dagger_KK} = \overline{G^\dagger_HK}$, we see that $AK$ is a reduced envelope for $K$, so $E$ is a reduced envelope for $K$.  Finally, observe that any reduced envelope $E$ for $H$ must contain $\overline{G^\dagger_HH}$, and in turn $\overline{G^\dagger_HH}$ contains the cocompact subgroup $\overline{G^\dagger_KK}$, so $\overline{G^\dagger_HH}$ is a cocompact subgroup of $E$.  This completes the proof of (ii).
\end{proof}

We observe that up to taking closures, the relative Tits core of an almost finite-rank flat subgroup is realized as the Tits core of any reduced envelope.

\begin{prop}\label{reduced_envelope:titscore}Let $G$ be a \tdlc group and let $H$ be an almost finite-rank flat subgroup of $G$.  Let $E$ be a reduced envelope for $H$ in $G$.  Then
\[
\overline{G^\dagger_H} = \overline{E^\dagger_H} = \overline{E^\dagger}.
\]
\end{prop}

\begin{proof}
Let $K$ be a cocompact subgroup of $\overline{H}$ such that $K$ is flat on $G$ and $K/K_{\us}$ is finitely generated.  By Theorem~\ref{thmintro:reduced_envelope}, $E$ is a reduced envelope for $K$.  By Theorem~\ref{thmintro:TitsCore:cocompact} we have $G^\dagger_H = G^\dagger_K$, so we may assume that $H=K$.  Clearly $G^\dagger_H = E^\dagger_H$, since $E$ is an open $H$-invariant subgroup of $G$.  Since the Tits core is invariant on passing to an open subgroup of finite index, the choice of $E$ is inconsequential, and we can arrange for $E$ to normalize $G^\dagger_H$ and contain a tidy subgroup for $H$, so that $H$ is flat on $E$.  Thus we may assume $G=E$ and that $G^\dagger_H$ is normal in $G$.

Let $S = G^\dagger_HH$.  Then $G^\dagger_S \le \overline{G^\dagger_H}$, since the action of $S$ on $G/\overline{G^\dagger_H}$ is anisotropic, and also $G^\dagger_H \le G^\dagger_S$ since $H \le S$.  So $\overline{G^\dagger_S} = \overline{G^\dagger_H}$, and in fact $\overline{G^\dagger_{\overline{S}}} = \overline{G^\dagger_H}$.

By Theorem~\ref{thmintro:reduced_envelope}, $\overline{S}$ is cocompact in $G$.  Hence $G/\overline{G^\dagger_H}$ is anisotropic by Theorem~\ref{thmintro:TitsCore:cocompact}, so $G^\dagger \le \overline{G^\dagger_H}$, and hence $\overline{G^\dagger} = \overline{G^\dagger_H}$.
\end{proof}

\subsection{Compact normal subgroups of reduced envelopes}\label{sec:compactnormal}

We can use discrete residuals of the $H$-action to restrict the action of the relative Tits core of $H$ on compact subgroups.

\begin{prop}\label{commutator_res}Let $G$ be a \tdlc group, let $H \le G$, and let $K$ be a compact subgroup of $G$ that is normalized by $G^\dagger_H$ and $H$.
\begin{enumerate}[(i)]
\item Every open subgroup of $K$ that is normalized by $H$ is also normalized by $G^\dagger_H$, so $\Res_K(H) = \Res_K(\overline{G^\dagger_HH})$.
\item The commutator group $[\overline{G^\dagger_H},K]$ is contained in $\Res_K(H)$.
\item Suppose that $H$ is almost finite-rank flat and let $E$ be a reduced envelope for $H$.  If $K$ is normalized by $E$, then $\Res_K(H) = \Res_K(E)$.
\end{enumerate}
\end{prop}

\begin{proof}We note that $G^\dagger_H = (\N_G(K))^\dagger_H$.  Thus without loss of generality, we may assume that $K$ is normal in $G$.

Let $L$ be an open subgroup of $K$.  Then $\N_{G}(L)$ is open in $G$ by Lemma~\ref{lem:compact_anisotropic}.  Thus if $L$ is normalized by $H$, then $\N_{G}(L)$ is an open $H$-invariant subgroup of $G$, so $\Res_G(H) \le \N_G(L)$ and in particular $G^\dagger_H \le \N_G(L)$.  Thus every open subgroup of $K$ that is normalized by $H$ is also normalized by $G^\dagger_H$.  Since the normalizer of any closed subgroup is closed, we have $\Res_K(\overline{G^\dagger_HH}) = \Res_K(G^\dagger_HH)$.  Hence $\Res_K(H) = \Res_K(\overline{G^\dagger_HH})$, completing the proof of (i).

Let $L$ be an open subgroup of $K$ that is normalized by $H$.  Since $\N_G(L)$ is open, there is a compact open subgroup $V$ of $G$ such that $[V,K] \le L$.

Let $h \in H$, let $u \in \con(h)$ and let $k \in K$.  Then for $n$ sufficiently large we have $h^nuh^{-n} \in V$, so
\[
h^n[k,u]h^{-n} = [h^nkh^{-n},h^nuh^{-n}] \in [K,V] \le L.
\]
Since $L$ is normalized by $h$, in fact $[k,u] \le L$, so $[K,\con(h)] \le L$.  In other words, $\con(h) \le \CC_G(K/L)$.  Since $h \in H$ was arbitrary and $\CC_G(K/L)$ is closed, it follows that $\overline{G^\dagger_H} \le \CC_G(K/L)$, in other words $[\overline{G^\dagger_H},K] \le L$.  Applying this argument to all open $H$-invariant subgroups $L$ of $K$, we conclude that $[\overline{G^\dagger_H},K] \le\Res_K(H)$, proving (ii).

Now suppose that $H$ is almost finite-rank flat.  Let $E$ be a reduced envelope for $H$; suppose $K$ is normal in $E$ and let $L$ be an open subgroup of $K$ that is normalized by $H$.  By part (i), the group $\N_E(L)$ contains $\overline{G^\dagger_HH}$ and thus is cocompact in $E$ by Theorem~\ref{thmintro:reduced_envelope}; moreover, $\N_E(L)$ is open in $E$.  Thus $L$ has only finitely many $E$-conjugates.  Since $L$ has finite index in $K$, we conclude that the intersection of all $E$-conjugates of $L$ is open in $K$, so $L \ge \Res_K(E)$.  Thus $\Res_K(H) \ge \Res_K(E)$; clearly also $\Res_K(E) \ge \Res_K(H)$, so in fact $\Res_K(E) = \Res_K(H)$, proving (iii).
\end{proof}

For the rest of this subsection, assume that $H$ is flat and $H/H_{\us}$ is finitely generated.  

By Theorem~\ref{thm:compact_invariant:tidy}, every $H$-invariant compact subgroup $K$ of $G$ is contained in a tidy subgroup $U$ for $H$, so in fact $K \le U_0$, where $U_0$ is the intersection of all $H$-conjugates of $U$.  We have good control over $\Res_{U_0}(H)$ thanks to the following:

\begin{lem}[See \cite{WillisFlat} Lemma~4.11 and Lemma~6.2]\label{lem:uzero:tidy}Let $G$ be a \tdlc group, let $H$ be a flat subgroup of $G$ such that $H/H_{\us}$ is finitely generated, and let $U$ and $V$ be compact open subgroups of $G$ that are tidy for $H$.  Let $U_0 = \bigcap_{h \in H}hUh\inv$ and let $V_0 = \bigcap_{h \in H}hVh\inv$.  Then
\[ U \cap V_0 = V \cap U_0.\]
In particular, $U_0 \cap V_0$ is an open $H$-invariant subgroup of $U_0$.
\end{lem}

Given a flat group $H$ of finite rank, define the \defbold{residual nub}\index{nub!residual nub}\index{rnub@$\snub_G(H)$} $\snub_G(H)$ to be the group
\[
\bigcap_{r \in G^\dagger_H}r\Res_{U_0}(H)r\inv,
\]
where $U_0$ is as in Lemma~\ref{lem:uzero:tidy}.  By Lemma~\ref{lem:uzero:tidy}, $U_0$ only depends up to an open subgroup on the choice of $U$, and hence $\snub_G(H)$ does not depend on the choice of $U$; in particular, $\snub_G(H) \le U$ for every tidy subgroup $U$ for $H$, and thus $\snub_G(H) \le \nub_G(H)$.  The residual nub has some further properties with regard to compact normal subgroups of reduced envelopes.

\begin{prop}\label{uzero_res:nub}Let $G$ be a \tdlc group, let $H$ be a flat subgroup of $G$ such that $H/H_{\us}$ is finitely generated and let $K$ be a compact $H$-invariant subgroup of $G$.  \begin{enumerate}[(i)]
\item We have $\Res_K(H) \le \Res_{U_0}(H) \le \nub_G(H)$, where $U_0$ is the intersection of $H$-conjugates of any tidy subgroup for $H$ on $G$.
\item If $K$ is normalized by $G^\dagger_HH$, then $\Res_K(\overline{G^\dagger_HH}) \le \snub_G(H)$.
\item Let $U$ be a tidy subgroup for $H$, let $E = \langle H,U \rangle$, and suppose $K$ is normalized by $E$.  Then $\snub_G(H)$ is normal in $E$ and the action of $E$ on $K\snub_G(H)/\snub_G(H)$ by conjugation has SIN.
\item Let $L$ be a uniscalar normal subgroup of $H$.  Then $\snub_G(L) = \nub_G(L)$, and $\snub_G(L)$ is normal in $E$.
\end{enumerate}
\end{prop}

\begin{proof}
By Theorem~\ref{thm:compact_invariant:tidy}, there is a tidy subgroup $U$ for $H$ such that $K \le U$, and hence $K \le U_0$ where $U_0 = \bigcap_{h \in H}hUh\inv$.  Certainly $\Res_{K}(H) \le \Res_{U_0}(H)$ in this case.  In turn, we see from Lemma~\ref{lem:uzero:tidy} that $\Res_{U_0}(H) = \Res_{V_0}(H)$, where $V_0$ is the intersection of $H$-conjugates of any given tidy subgroup $V$ for $H$ on $G$.  In particular, $\Res_{U_0}(H) \le V$ for all tidy subgroups $V$ for $H$ on $G$, and hence $\Res_{U_0}(H) \le \nub_G(H)$, proving (i).

Now suppose $K$ is normalized by $G^\dagger_HH$.  Then $\Res_K(H) = \Res_K(\overline{G^\dagger_HH})$ by Proposition~\ref{commutator_res}, so $\Res_K(\overline{G^\dagger_HH}) \le \Res_{U_0}(H)$ by part (i).  Moreover, clearly $\Res_K(\overline{G^\dagger_HH})$ is also $G^\dagger_H$-invariant.  Hence $\Res_K(\overline{G^\dagger_HH}) \le \snub_G(H)$, proving (ii).

Let $E = \langle H,U \rangle$ where $U$ is tidy for $H$.  We have $E = G^\dagger_HU_0H$ by Theorem~\ref{thmintro:reduced_envelope}.  In particular, it follows that
\[
\bigcap_{r \in G^\dagger_H}r\Res_{U_0}(H)r\inv = \bigcap_{r \in E}r\Res_{U_0}(H)r\inv,
\]
and hence $\snub_G(H)$ is normal in $E$.

Suppose that $K$ is normalized by $E$.  Then $K\snub_G(H)$ is also a compact subgroup of $G$ normalized by $E$, so we may assume that $K \ge \snub_G(H)$.  We have $\Res_K(H) = \Res_K(E)$ by Proposition~\ref{commutator_res}.  The same argument as for part (ii) shows that $\Res_K(E) \le \snub_G(H)$.  The open $E$-invariant subgroups of $K/\Res_K(E)$ have trivial intersection; via Lemma~\ref{lem:intersection:quotient}, we conclude that the open $E$-invariant subgroups of $K/\snub_G(H)$ have trivial intersection.  By a compactness argument, the open $E$-invariant subgroups of $K/\snub_G(H)$ form a base of identity neighbourhoods, so the action of $E$ on $K/\snub_G(H)$ has SIN, completing the proof of (iii).

Finally, let $L$ be a uniscalar normal subgroup of $H$.  Given a tidy subgroup $U$ for $H$, then $U$ is normalized by $L$.  We have $\Res_U(L) = \Res_G(L)$, and in turn $\Res_G(L)$ is just the intersection of all compact open $L$-invariant subgroups, so $\Res_G(L) = \nub_G(L)$.  Now $\nub_G(L)$ is normalized by $H$, since $L$ is normalized by $H$, and also $\nub_G(L)$ is normalized by $U$ by Corollary~\ref{cor:nub:normalizer}.  Hence $\nub_G(L)$ is normal in $E$; in particular, $\nub_G(L)$ is normalized by $G^\dagger_H$.  It now follows from the definition of the residual nub that $\snub_G(L) = \nub_G(L)$, completing the proof of (iv).
\end{proof}

Combining Proposition~\ref{commutator_res} with Proposition~\ref{uzero_res:nub}, we obtain a restriction on compact $(G^\dagger_HH)$-invariant subgroups of $G$ as follows.

\begin{prop}\label{prop:commutator_nub}Let $G$ be a \tdlc group and let $H$ be a flat subgroup of $G$ such that $H/H_{\us}$ is finitely generated.
\begin{enumerate}[(i)]
\item Let $K$ be a compact $(G^\dagger_HH)$-invariant subgroup of $G$.  Then
\[
[\overline{G^\dagger_H},K] \le \snub_G(H).
\]
\item Let $R = \overline{G^\dagger_H\snub_G(H)}$.  Then every compact normal $H$-invariant subgroup of $R/\snub_G(H)$ is central in $R/\snub_G(H)$.
\end{enumerate}
\end{prop}

\begin{proof}By Proposition~\ref{commutator_res}, we have $[\overline{G^\dagger_H},K] \le \Res_K(H)$, and by Proposition~\ref{uzero_res:nub} we have $\Res_K(H) \le \snub_G(H)$.  Hence $[\overline{G^\dagger_H},K] \le \snub_G(H)$, proving (i).

(ii) follows immediately from (i), noting that if $K/\snub_G(H)$ is compact, then $K$ is compact.
\end{proof}

The possibilities for $\snub_G(H)$ are mysterious at present, although we note that in some situations, the fact that $\snub_G(H)$ has open normalizer is a useful restriction.

\begin{prop}Let $G$ be a \tdlc group and let $H$ be a flat subgroup of $G$ such that $H/H_{\us}$ is finitely generated.  Suppose that $G^\dagger_H \neq \triv$ and that there is a tidy subgroup $U$ for $H$ such that $U$ is just infinite and not virtually abelian.
\begin{enumerate}[(i)]
\item We have $\snub_G(H) = \triv$.
\item Let $K$ be a compact subgroup of $G$ with open normalizer, such that $H \le \N_G(K)$.  Then $K \cap U = \triv$.
\end{enumerate}
\end{prop}

\begin{proof}Let $U_0 = \bigcap_{h \in H}hUh\inv$.  We divide into two cases: either $U_0$ is open in $U$, or $U_0$ is not open in $U$.

If $U_0$ is open in $U$, then $U_0$ is itself a tidy subgroup for $H$ and $H$ is uniscalar.  As in the proof of Corollary~\ref{corintro:relative_discrete:hji}, we deduce that in fact $H$ has SIN action on $G$.  In particular $\nub_G(H) = \triv$, so $\snub_G(H) = \triv$.

If instead $U_0$ is not open in $U$, then $\snub_G(H)$ is a subgroup of $U$ that is closed but not open; moreover, $\snub_G(H)$ is normal in $U$ by Proposition~\ref{uzero_res:nub}(iii).  Since $U$ is just infinite, we conclude that $\snub_G(H)$ must be trivial.  This completes the proof of (i).

Let $K$ be a compact subgroup of $G$, such that $\N_G(K)$ is open in $G$ and $\N_G(K) \ge H$.  Then certainly $\N_G(K) \ge \overline{G^\dagger_H}$.  By Proposition~\ref{prop:commutator_nub} and part (i), we see that in fact $K$ commutes with $\overline{G^\dagger_H}$.  Now $\overline{G^\dagger_H}$ is a non-trivial, hence non-discrete, subgroup of $G$, and hence $\overline{G^\dagger_H} \cap U$ is infinite; moreover, $\overline{G^\dagger_H} \cap U$ is normal in $U$ by Proposition~\ref{prop:twosided:stable}, so $\overline{G^\dagger_H} \cap U$ is open in $U$.  Thus $K$ is centralized by an open subgroup of $U$.  Since $U$ is just infinite and not virtually abelian, it is easily verified that $U$ does not have any non-trivial finite conjugacy classes, that is, no element of $U \setminus \triv$ has open centralizer in $U$.  Thus $K \cap U = \triv$, proving (ii).
\end{proof}

The group $\snub(H)$ is also relevant for describing the structure of a compactly generated group that has flat action on itself.  (Note that if $G$ is any \tdlc group and $H$ is a compactly generated flat subgroup of $G$, then $H$ is flat on itself by Corollary~\ref{cor:finiterank:subgroup}.)

\begin{prop}\label{prop:flatitself}Let $G$ be a compactly generated \tdlc group such that $G$ is flat on itself.  Let $U$ be a compact open subgroup of $G$ that is tidy for $G$ and let $U_0$ be the core of $U$ in $G$.  Then the following holds:
\begin{enumerate}[(i)]
\item $\snub(G) = \Res_{U_0}(G)$ is the largest compact normal subgroup of $G$ on which $G$ acts ergodically.
\item The factor $G_{\us}/\nub(G)$ is a SIN group.
\item $G$ has SIN action on $\nub(G)/\snub(G)$.
\item We have $\nub(G) = \nub(G_{\us})\snub(G)$.
\end{enumerate}
\end{prop}

\begin{proof}
It is clear that $\Res_{U_0}(G)$ is normal in $G$, so $\snub(G) = \Res_{U_0}(G)$.  Let $R$ be a compact normal subgroup on which $G$ acts ergodically.  Then $R = \Res_R(G)$; we have $\Res_R(G) \le \Res_{U_0}(G)$ by Proposition~\ref{uzero_res:nub}, so $R \le \snub(G)$.  On the other hand, $G$ acts ergodically on $\Res_{U_0}(G)$ by Theorem~\ref{thm:distal}.  Thus $\snub(G)$ is characterized as in (i).

Given a tidy subgroup $V$ for $G$, then $V$ is normalized by $G_{\us}$.  In fact $V$ is itself uniscalar (since it is compact), so $V \unlhd G_{\us}$.  By a compactness argument, we see that the open normal subgroups of $G_{\us}/\nub(G)$ form a base of identity neighbourhoods, so $G_{\us}/\nub(G)$ is a SIN group, proving (ii).  The action of $G$ on $U_0/\snub(G)$ is residually discrete, hence a SIN action by compactness, so the action on $\nub(G)/\snub(G)$ also has SIN, proving (iii).

Given $g \in G$, then $g$ acts ergodically on $\nub(g)$.  Part (i) then ensures that $\nub(g) \le \snub(G)$ for all $g \in G$.  Applying Theorem~\ref{thmintro:flat:nubdecomp}, we see that $\nub(G) = \nub(G_{\us})\snub(G)$, proving (iv).
\end{proof}

\subsection{Cocompact envelopes and subnormal subgroups}\label{sec:cocompact_envelope}

\begin{defn}Let $G$ be a \tdlc group and let $H$ be a subgroup of $G$.  Say $H$ is \defbold{almost open}\index{almost open} if there exists an open subgroup $L$ of $G$ such that $H \le L$ and $\overline{H}$ is cocompact in $L$; call such an $L$ a \defbold{cocompact envelope}\index{envelope!cocompact envelope} of $H$.\end{defn}

Here are some easy observations on this definition, given the results we have so far.

\begin{lem}\label{lem:cocompact_envelope}
Let $G$ be a \tdlc group and let $H \le G$.
\begin{enumerate}[(i)]
\item If $H$ is almost open in $G$, then every cocompact envelope of $H$ is reduced and \emph{vice versa}.
\item If $H$ is almost open in $G$ and $K \le G$ is such that $\overline{K}$ is a cocompact subgroup of $\overline{H}$, then $K$ is almost open in $G$.
\item Suppose $K \le H$ such that $K$ is closed and normal in $G$.  Then $H$ is almost open in $G$ if and only if $H/K$ is almost open in $G/K$.
\item If $H$ is almost finite-rank flat, then $H$ has a cocompact envelope if and only if $\overline{H}$ is cocompact in $\overline{G^\dagger_HH}$.
\end{enumerate}
\end{lem}

\begin{proof}
Parts (i), (ii) and (iii) are clear from the definitions.  Part (iv) follows immediately from Theorem~\ref{thmintro:reduced_envelope}.
\end{proof}

Now consider the situation that the compactly generated almost flat subgroup $H$ of $G$ is subnormal in some open subgroup.  This can only occur under special circumstances, and in particular we find that $H$ is almost open in $G$, as stated in Theorem~\ref{thmintro:subnormal_envelope}.

\begin{proof}[Proof of Theorem~\ref{thmintro:subnormal_envelope}]
Let $K$ be a cocompact closed flat subgroup of $H$.  Note that since $H$ is compactly generated, its cocompact subgroup $K$ is also compactly generated, and hence $K/K_{\us}$ is finitely generated.

Let $E$ be a reduced envelope of $H$.  Then $H$ is subnormal in $E$, so by Corollary~\ref{cor:subnormal:sametitscore}, we have $E^\dagger_H = H^\dagger$.

By Theorem~\ref{thmintro:reduced_envelope}, the group $\overline{E^\dagger_H H}$ is cocompact in some, hence every, reduced envelope for $H$, so $\overline{E^\dagger_H H}$ is cocompact in $E$.  Since $E^\dagger_H \le H$, it follows that $H$ is cocompact in $E$, proving (i).  Hence $E^\dagger_H = E^\dagger$ by Theorem~\ref{thmintro:TitsCore:cocompact}.  We have now proved that $E^\dagger = H^\dagger$.  Note also that $K$ is cocompact in $E$, so $E$ is itself compactly generated and almost flat.

Let
\[
E = H_0 \unrhd H_1 \unrhd \dots \unrhd H_n = H
\]
be a descending subnormal series from $E$ to $H$.  Then it is clear that $\Res_{E}(H) \le H_1$; since $E$ is open in $G$, in fact $\Res_G(H) \le H_1$.  We then see that $\Res^i_G(H) \le H_{i+1}$ for all $i < n$, so in fact $\Res^\infty_G(H) \le \Res(H)$.  Indeed, $\Res^\infty_G(H) = \Res^\infty(H)$.  Since $H$ is cocompact in $E$, Theorem~\ref{thm:relative_discrete:general} ensures that in fact
\[
\Res^\infty_G(H) = \Res_G(H) = \Res_G(K) \text{ and } \Res^\infty_G(E) = \Res_G(E) = \Res_G(K)
\]
since $E$ is open, $\Res_G(E) = \Res(E)$ and $\Res^\infty_G(E) = \Res^\infty(E)$.  So we have
\[
\Res^\infty(E) = \Res(E) = \Res(H).
\]
By Corollary~\ref{cor:finiterank:subgroup}, the action of $K$ on $H$ is flat of finite rank.  Thus by applying Theorem~\ref{thm:relative_discrete:general} to the action of $H$ on itself, we obtain $\Res^\infty(H) = \Res(H)$, completing the proof of (ii).
\end{proof}

\subsection{Faithful weakly decomposable groups}\label{sec:weakdecomp}

We apply the results of \S\ref{sec:compactnormal} to a class of \tdlc groups considered in \cite{CRW-Part1} and \cite{CRW-Part2}.

\begin{defn}An action of a \tdlc group $G$ on a Boolean algebra $\mc{A}$ is \defbold{(non-degenerate) faithful weakly decomposable}\index{faithful weakly decomposable} if it is faithful, such that for all $\alpha \in \mc{A} \setminus \{0\}$, the stabilizer in $G$ of $\alpha$ is open, and the pointwise stabilizer of the set $\{\beta \in \mc{A} \mid \alpha \wedge \beta = 0\}$ is non-discrete.  We say $G$ is \defbold{faithful weakly decomposable} if it has a non-degenerate faithful weakly decomposable action on some Boolean algebra.\end{defn}

The faithful weakly decomposable property implies several structural properties of $G$, as described in \cite[\S5]{CRW-Part1}.  In particular, if $G$ is faithful weakly decomposable, it follows from results in \cite{CRW-Part1} that $G$ has trivial quasi-center, and given $\triv \not= K \le G$ such that $\N_G(K)$ is open, then $K$ is not abelian.  The Boolean algebra $\mc{A}$ can always be taken to be the \defbold{(global) centralizer lattice} of $G$, that is the set
\[
\{\CC_G(K) \mid K \le G, \; \N_G(K) \text{ is open}\};
\]
given $K,L \le G$ such that $\N_G(K)$ and $\N_G(L)$ are open, if $\CC_G(K)$ and $\CC_G(L)$ have an open subgroup in common, then $\CC_G(K) = \CC_G(L)$.

Moreover, if $H$ is a closed subgroup of $G$ such that $\N_G(H)$ is open, then $H$ is also faithful weakly decomposable (by \cite[Proposition~5.22]{CRW-Part1}).

The faithful weakly decomposable property also has implications for the local dynamics of $G$, as investigated in \cite[\S6]{CRW-Part2}.  In particular, we recall the following.

\begin{prop}\label{Part2:skewering}Let $G$ be a \tdlc group acting on a Boolean algebra $\mc{A}$.  Suppose the action is non-degenerate faithful weakly decomposable.
\begin{enumerate}[(i)]
\item (See \cite[Theorem~6.11]{CRW-Part2}) Suppose that $G$ is compactly generated and that there is an identity neighbourhood in $G$ that contains no non-trivial compact normal subgroups of $G$.  Then there exists $g \in G$ and $\alpha \in \mc{A}$ such that $g\alpha < \alpha$.
\item (See \cite[Proposition~6.7]{CRW-Part2}) Let $g \in G$, and suppose there exists $\alpha \in \mc{A}$ such that $g\alpha < \alpha$.  Then $\nub(g)$ is non-trivial; in other words, $\con(g)$ is not closed.
\end{enumerate}
\end{prop}

We can use this result to establish a dichotomy for faithful weakly decomposable \tdlc groups, as stated in the introduction.  We begin the proof of Theorem~\ref{thmintro:weakdecomp:nonclosed} with a lemma.

\begin{lem}\label{weakdecomp:nubstar}Let $G$ be a faithful weakly decomposable \tdlc group and let $\alpha$ be an automorphism of $G$ that is isotropic on $G$.  Let $S = \overline{G^\dagger_\alpha} \rtimes \langle \alpha \rangle$, equipped with a topology such that $\overline{G^\dagger_\alpha}$ is embedded as an open subgroup of $S$.  Suppose $\snub_G(\alpha)=\triv$.  Then $S$ is compactly generated and has no non-trivial compact normal subgroup, and there exists $g \in S$ such that $\nub_{\overline{G^\dagger_\alpha}}(g)$ is non-trivial.\footnote{For the application it would suffice to restrict to the case when $\alpha$ is an inner automorphism, say conjugation by $h \in G$.  However, complications arise in the proof if one considers the closed subgroup $S = \overline{\langle G^\dagger_h, h \rangle}$ of $G$ instead of the semidirect product $\overline{G^\dagger_h} \rtimes \langle h \rangle$; in particular, it can happen that $S/\overline{G^\dagger_h}$ is an infinite compact group, and in this case it is not clear whether $S$ can have non-trivial compact normal subgroups.}\end{lem}

\begin{proof}Let $T = \overline{G^\dagger_\alpha}$.  By Corollary~\ref{almostflat:titscore}, $T$ has open normalizer in $G$.  By Proposition~\ref{prop:commutator_nub}, every compact $\langle \alpha , G^\dagger_\alpha \rangle$-invariant subgroup of $G$ commutes with $T$.

Let $\mc{A}$ be the global centralizer lattice of $G$, on which the action of $G$ is faithful weakly decomposable by \cite[Theorem~5.18]{CRW-Part1}.  As explained in \cite[Proposition~5.22]{CRW-Part1}, whenever $L$ is a closed subgroup of $G$ such that $\N_G(L)$ is open, there is a principal ideal $\mc{I}$ of $\mc{A}$, which can be regarded as a Boolean algebra in its own right, such that the action of $L$ on $\mc{I}$ is faithful weakly decomposable.  In particular, this argument applies to $L = T$.  Moreover $\mc{I}$ is obtained from $\mc{A}$ in a canonical way, so the action extends to an action of $S$ on $\mc{I}$.  The latter action is not necessarily faithful, but clearly the kernel $K$ of the action is discrete so $K \le \QZ(S)$, and in fact $\QZ(S)$ acts trivially on $\mc{I}$, so $K = \QZ(S)$ and $T \cap \QZ(S) = \triv$.  Thus $\QZ(S)$ is a discrete subgroup of $S$; also $\QZ(S)$ is isomorphic to a subgroup of $\langle \alpha \rangle$, so $\QZ(S)$ is torsion-free, and hence $\QZ(S)$ has trivial intersection with every compact subgroup of $S$.  Since $T$ is open in $S$, we have $\QZ(S) \ge \CC_S(T)$.  Since $T$ and $\QZ(S)$ normalize each other and have trivial intersection, in fact $\QZ(S) = \CC_S(T)$.

Observe that $\alpha$ does not leave invariant any proper open subgroup of $T$, so $S = \langle \alpha, U \rangle$, where $U$ is any compact open subgroup of $T$; in particular $S$ is compactly generated.  Given a compact normal subgroup $N$ of $S$, then $N$ commutes with $T$, so $N \le \QZ(S)$ and hence $N$ is trivial.

If $\QZ(S)=\triv$, then $S$ is faithful weakly decomposable and there is an identity neighbourhood (namely $T$) that contains no non-trivial compact normal subgroup of $S$, so by Proposition~\ref{Part2:skewering}, there exists $g \in S$ such that $\nub_S(g)$ is non-trivial; clearly $\nub_S(g) = \nub_{T}(g)$.  If instead $\QZ(S) > \triv$, then there exists $\beta \in \QZ(S)$ and $n > 0$ such that $\beta\alpha^n \in T$.  Since $\beta$ centralizes $T$, we see that 
\[
T^\dagger_{\beta\alpha^n} = G^\dagger_{\alpha^n} = G^\dagger_{\alpha} \text{ and } \snub_T(\beta\alpha^n) \le \snub_G(\alpha^n) = \snub_G(\alpha) = \triv.
\]
The same argument as used for $S$ now shows that $T$ is compactly generated and has no non-trivial compact normal subgroups.  Hence by Proposition~\ref{Part2:skewering}, there exists $g \in T$ such that $\nub_T(g)$ is non-trivial.
\end{proof}

\begin{proof}[Proof of Theorem~\ref{thmintro:weakdecomp:nonclosed}]
It is clear that (i) and (ii) are mutually exclusive.  We may suppose that (ii) fails, that is, every contraction group in $G$ is closed.

By Proposition~\ref{Part2:skewering}, it follows that $G$ has arbitrarily small non-trivial compact normal subgroups.

Let us now suppose there exists $h \in G$ such that $\con(h) \not= \triv$, so $h$ is isotropic on $G$.  Since $\con(h)$ is closed, $\nub(h) = \triv$, so certainly $\snub_G(h)$ is trivial.  Thus by Lemma~\ref{weakdecomp:nubstar}, there exists $g \in S$ such that $\nub_{\overline{G^\dagger_h}}(g)$ is non-trivial, where $S = \overline{G^\dagger_h} \rtimes \langle h\rangle$.  Let $g' \in G$, with the same action on $\overline{G^\dagger_h}$ as $g$ has.  Then $\nub_{\overline{G^\dagger_h}}(g')$ is a non-trivial compact subgroup of $G$ on which $g'$ acts ergodically, so $g'$ has non-trivial nub on $G$.  But then $\con(g')$ is not closed, a contradiction.  Thus $G$ is anisotropic, so (i) holds.
\end{proof}


\end{document}